\documentclass[12pt]{amsart}
\usepackage[utf8]{inputenc}
\usepackage[T2A]{fontenc}
\usepackage{amsmath,amssymb,amsthm,bbold,upgreek}
\usepackage[only=inplus]{stmaryrd}
\usepackage[scr=boondox]{mathalpha}
\usepackage{microtype}

\usepackage{subcaption}

\usepackage{hyperref}
\usepackage[a4paper,hmargin=2.5cm,vmargin=2.5cm]{geometry}
\usepackage[english]{babel}
\usepackage[pdftex]{graphicx}
\usepackage{tikz,tikz-cd}
\usepackage{enumitem}
\usetikzlibrary{arrows,calc,matrix}
\usepackage{csquotes}
\usepackage{amsthm}
\usepackage{comment}
\usepackage{mathtools}
\usepackage{eucal}
\usepackage{bm}
\usepackage{scalerel}
\usepackage{stackrel}
\usepackage[all,cmtip,matrix, arrow]{xy}
\usepackage{ytableau}

\usepackage[backend=biber,date=short, bibencoding=utf8,sorting=nyt,style=alphabetic]{biblatex}
\bibliography{isotropic_Grassmanians}

\usepackage[normalem]{ulem}

\newcommand{\Sp}{\mathrm{Sp}}
\newcommand{\Gr}{\mathrm{Gr}}
\newcommand{\IGr}{\mathrm{IGr}}
\newcommand{\IFl}{\mathrm{IFl}}
\newcommand{\GL}{\mathrm{GL}}
\newcommand{\Fl}{\mathrm{Fl}}
\newcommand{\cU}{\mathcal{U}}
\newcommand{\cS}{\mathcal{S}}
\newcommand{\ZZ}{\mathbb{Z}}
\newcommand{\PP}{\mathbb{P}}
\newcommand{\cE}{\mathcal{E}}
\newcommand{\cG}{\mathcal{G}}
\newcommand{\cH}{\mathcal{H}}
\newcommand{\cO}{\mathcal{O}}
\newcommand{\bG}{\mathbf{G}}
\newcommand{\cC}{\mathcal{C}}
\newcommand{\cA}{\mathcal{A}}
\newcommand{\cB}{\mathcal{B}}
\newcommand{\cF}{\mathcal{F}}
\newcommand{\cV}{\mathcal{V}}
\newcommand{\cW}{\mathcal{W}}
\newcommand{\B}{\mathrm{B}}
\newcommand{\cT}{\mathcal{T}}
\newcommand{\cQ}{\mathcal{Q}}
\newcommand{\cR}{\mathcal{R}}

\newcommand{\bP}{\mathbf{P}}
\newcommand{\bU}{\mathbf{U}}
\newcommand{\bL}{\mathbf{L}}
\newcommand{\Y}{\mathrm{Y}}
\newcommand{\YD}{\mathrm{YD}}

\newcommand{\lvee}[1]{\vphantom{#1}^\vee\!{#1}}

\newcommand{\bV}{\mathbf{V}}
\newcommand{\dd}{\mathbb{d}}
\newcommand{\mm}{\mathbb{m}}
\newcommand{\vv}{\mathbb{v}}
\newcommand{\bVs}{\mathbf{V}^\Sp}

\newcommand{\LL}{\mathbb{L}} 
\newcommand{\RR}{\mathbb{R}} 

\newcommand{\s}{\mathfrak{s}}
\newcommand{\cK}{\mathcal{K}}
\newcommand{\Sh}{\ensuremath{\operatorname{\Sigma}}} 
\newcommand{\E}[1]{\operatorname{Ext}^{#1}}

\newcommand{\Db}{\mathrm{D}^{b}}

\newcommand{\cD}{{\mathcal{D}}}

\newtheorem{theorem}{Theorem}[section]
\newtheorem{proposition}[theorem]{Proposition}
\newtheorem{lemma}[theorem]{Lemma}
\newtheorem{conjecture}[theorem]{Conjecture}

\newtheorem{corollary}[theorem]{Corollary}

\theoremstyle{remark}

\newtheorem{remark}[theorem]{Remark}

\newtheorem{example}[theorem]{Example}

\theoremstyle{definition}
\newtheorem{definition}{Definition}

\newtheorem*{notation*}{Notation}
\newtheorem*{lemma*}{Lemma}

\title{Full exceptional collections on the isotropic Grassmannians}

\begin{document}
\author{Lyalya Guseva}
\address{Université Bourgogne Europe, CNRS, IMB UMR 5584, F-21000 Dijon, France}
\email{lyalya.guseva1994@gmail.com}

\author{Alexander Novikov}
\address{HSE University, Russian Federation}
\email{all.novv@gmail.com}

\thanks{L.G was partially supported by the FanoHK ANR-20-CE40-0023 and by the EIPHI ANR-17-EURE-0002.}
\begin{abstract} 
We prove that the Kuznetsov--Polishchuk exceptional collections on rational homogeneous spaces of the symplectic groups $\mathrm{Sp}(2n,\mathbb{C})$ are full and consist of vector bundles. To achieve this, we construct several classes of complexes, which we call generalized staircase complexes, symplectic staircase complexes and secondary staircase complexes --- each of which may be of independent interest.
\end{abstract}
\maketitle{}
\section{Introduction}
The bounded derived category of coherent sheaves $\Db(X)$ is an important invariant of an algebraic variety~$X$. In general, $\Db(X)$ has a quite complicated structure. However, $\Db(X)$ can be described fairly explicitly when it admits a full exceptional collection~$(E_{1}, E_{2}, \ldots, E_{m})$: in this case every object of~$\Db(X)$ has a unique filtration, with the $i$-th subquotient being a direct sum of shifts of the objects $E_i$. Therefore, an exceptional collection can be considered as a kind of basis for $\Db(X)$.

The first example of a full exceptional collection was constructed by Beilinson in~\cite{Beilinson1978}. He proved that the collection of line bundles $\mathcal{O},\mathcal{O}(1), \ldots, \mathcal{O}(n)$ is a full exceptional collection on the projective space~$\mathbb{P}^n$. Subsequently, Kapranov~\cite{Kapranov1988} constructed full exceptional collections on Grassmannians, flag varieties of groups $\operatorname{SL}_{n}$
and on smooth quadrics.
It has been conjectured afterwards that
\begin{conjecture} \label{conjecture}
If $\bG$ is a semisimple algebraic group over an algebraically closed field $\Bbbk$ of characteristic~$0$ and~$\bP\subset \bG$ is a parabolic subgroup, then there is a full exceptional collection of equivariant vector bundles on~$\bG/\bP$.
\end{conjecture}
The conjecture easily reduces to the case where $\bG$ is a simple group and $\bP$ is its maximal parabolic subgroup, see~\cite[Section~1.2]{KuzPol}.

In the case of simple $\bG$ and maximal $\bP$ full exceptional collections of equivariant vector bundles was constructed for the following series (we use the Bourbaki indexing of simple roots):
\begin{itemize}
\item $\bG$ is of Dynkin type $A$ and any $\bP$, see \cite{Kapranov1988};
\item $\bG$ is of Dynkin type $B$ and $\bP$ corresponds to one of the first two simple roots, see~\cite{Kapranov1988,lines};
\item $\bG$ is of Dynkin type $C$ and $\bP$ corresponds to one of the first two simple roots or the last root, see~\cite{Beilinson1978,lines, F19};
\item $\bG$ is of Dynkin type $D$ and $\bP$ corresponds to one of the first two simple roots, see~\cite{kap84,kuznetsov2018residual};
\end{itemize}
and for the following sporadic and exceptional cases:
\begin{itemize}
\item $(B_3,P_3)$, $(B_4,P_4)$, see~\cite{Kapranov1988, Kuznetsov2007};
\item $(C_4,P_3)$, $(C_5,P_3)$ see~\cite{Guseva_2020, N20}
\item $(D_4,P_3)$, $(D_4,P_4)$, $(D_5,P_4)$, $(D_5,P_5)$, $(D_6,P_5)$, $(D_6,P_6)$, see~\cite{Kapranov1988, Kuznetsov2007, SG6};
\item $(E_6,P_1)$, $(E_6,P_6)$, see~\cite{FaenziManivel};
\item $F_4/P_1$, $F_4/P_4$, see~\cite{Smirnov2021,BKS};
\item $(G_2,P_1)$, $(G_2,P_2)$, see~\cite{Kapranov1988, Kuznetsov2007}.
\end{itemize}

In the recent paper \cite{SamKall24} Samokhin and van der Kallen announced the existence of a full exceptional collection in the bounded derived category of coherent sheaves on a generalized flag scheme $\bG/\bP$ over~$\mathbb{Z}$. In particular, their result implies the existence of a full exceptional collection on $\bG/\bP$ over an algebraically closed field $\Bbbk$ of characteristic~$0$. However, it is unclear at the moment whether the exceptional objects from \cite{SamKall24} are pure with respect to the standard $t$-structure on $\Db(\bG/\bP)$. The relation of the construction of \cite{SamKall24} to the results of the present paper is a subject of further research. 

In~\cite{KuzPol} Kuznetsov and Polishchuk constructed an exceptional collection of maximal possible length (equal to the rank of the Grothendieck group) on $\bG/\bP$ for all classical groups (i.e. for groups of Dynkin types $ABCD$) and all their maximal parabolic subgroups. However, the fullness of these collections was proved only for Lagrangian Grassmannians $C_n/P_n$ in \cite{F19}. It was conjectured~\cite[Conjecture 4.1]{KuzPol} that the exceptional collections constructed in~\cite{KuzPol} consist of vector bundles. 

\subsection{The main result}
The main goal of this article is to prove that the Kuznetsov--Polishchuk exceptional collections from \cite{KuzPol} on the homogeneous varieties $C_n/P_k$ consist of vector bundles and are full for all $n$ and $k$. In particular, this proves the Conjecture~\ref{conjecture} for $\bG$ of type $C$. 

Note that the homogeneous variety $C_n/P_k$ parameterizes $k$-dimensional isotropic subspaces in a~$2n$-dimensional symplectic vector space; therefore, we denote it by $\IGr(k,2n)$. 
\begin{theorem} \label{theo:intro}
Let $V$ be a $2n$-dimensional symplectic vector space over an algebraically closed field of characteristic zero and let $1\le k\le n$ be an integer. Then the Kuznetsov--Polishchuk exceptional collection in the bounded derived category $\Db(\IGr(k, V))$ of coherent sheaves on $\IGr(k,V)$ consists of equivariant vector bundles and is full.   
\end{theorem}
\subsection{Kuznetsov--Polishchuk construction} \label{intro:KP}
First, recall the description of the equivariant vector bundles on $\bG/\bP$. The category $\mathrm{Coh}^{\bG}(\bG/\bP)$ of the $\bG$-equivariant coherent sheaves on $\bG/\bP$ is equivalent to the category of representations of $\bP$:
\begin{equation} \label{identification}
    \mathrm{Coh}^{\bG}(\bG/\bP)\simeq \mathrm{Rep}(\bP).
\end{equation}
Denote by $\bU\subset \bP$ the unipotent radical of $\bP$ and by $\bL=\bP/\bU$ the Levi quotient. For each dominant weight of $\bL$, we consider the corresponding irreducible representation of $\bL$. Extending it (via the projection~$\bP\to \bL$) to a representation of $\bP$ we obtain through the identification~\eqref{identification} a $\bG$-equivariant vector bundle on $\bG/\bP$. The $\bG$-equivariant vector bundles on $\bG/\bP$ obtained from irreducible representations of~$\bL$ will be called \textit{irreducible equivariant vector bundles} on $\bG/\bP$.

Let us briefly describe the construction of Kuznetsov and Polishchuk. First, they observe that irreducible equivariant vector bundles form an infinite full exceptional collection in the equivariant derived category~$\mathrm{D}^{\bG}(\bG/\bP)$. Next, they introduce an explicit condition for a finite subcollection of irreducible equivariant vector bundles which ensures that its dual collection in~$\mathrm{D}^{\bG}(\bG/\bP)$ is already exceptional in the non-equivariant category~$\Db(\bG/\bP)$. Such a subcollection is called an \textit{exceptional block}; see~\cite[Definition~3.1]{KuzPol} and Subsection~\ref{subsection:E} for more details. Thus, a subcategory in $\Db(\bG/\bP)$ generated by an exceptional block admits a full exceptional collection of equivariant objects. Finally, Kuznetsov and Polishchuk for $\bG$ of type $ABCD$ choose several exceptional blocks so that the exceptional collections from different blocks form an exceptional collection of length that equals the rank of the Grothendieck group of $\bG/\bP$.

Let us now present how Kuznetsov--Polishchuk construction looks for~$\IGr(k,V)$, see the details in Subsection~\ref{subsection:E}.

We denote by $\cU_k\subset V\otimes \cO$ and $\cQ_{2n-k}^{\vee}\subset V^{\vee}\otimes \cO$ the tautological subbundle and the dual of the tautological quotient bundle on~$\IGr(k,V)$, of ranks $k$ and $2n-k$, respectively. These fit into the following exact sequence on~$\IGr(k,V)$:
\begin{equation*}
0\to \cQ^{\vee}_{2n-k} \to V^{\vee}\otimes \cO \to \cU^{\vee}_k \to 0.
\end{equation*}
We define the symplectic vector bundle on~$\IGr(k,V)$ by
\begin{equation*}
\cS_{2(n-k)}\coloneqq \cQ^{\vee}_{2n-k}/\cU_k,
\end{equation*}
which inherits a natural symplectic structure from $V$, and has rank $2(n-k)$.

Let $\bG=\Sp(2n,\Bbbk)$, and let $\bP$ be a parabolic subgroup corresponding to the $k$-th simple root of $\bG$. The weight lattice of $\bG$ can be identified with $\mathbb{Z}^n$ in such a way that the $l$-th fundamental weight goes to $(1,\dots,1,0,\dots,0)$
(the first $l$ entries are $1$, and the last $n-l$ are $0$). The dominant cone of~$\bL\simeq \GL(k)\times \Sp(2(n-k), \Bbbk)$ can be identified with $\Y_k\times \YD_{n-k}$, where
\begin{equation*}
\Y_k = \{(\alpha_1\ge \alpha_2\ge \ldots\ge \alpha_k)\ |\ \alpha_i \in \mathbb{Z}\},\qquad \YD_{n-k} = \{(\mu_1\ge \mu_2\ge \ldots\ge \mu_{n-k}\ge 0)\ |\ \mu_i \in \mathbb{Z}\}.
\end{equation*}

The equivalence~\eqref{identification} yields that each dominant weight $(\alpha; \mu)\in \Y_k\times \YD_{n-k}$ of $\bL$ corresponds to an irreducible $\bG$-equivariant vector bundle on $\IGr(k,V)$ which we denote by \begin{equation*}
\Sigma^{\alpha}\cU^{\vee}_k\otimes \Sigma^{\mu}_{\Sp}\cS_{2(n-k)}.
\end{equation*}
Here, $\Sigma^{\alpha}$ denotes the Schur functor corresponding to the $\GL(k)$-representation of highest weight $\alpha$, and $\Sigma^{\mu}_{\Sp}$ denotes the symplectic Schur functor corresponding to the $\Sp(2(n-k))$-representation of highest weight $\mu$; see Section~\ref{subsection:schur} for detailed definitions.

The exceptional collection on $\IGr(k,V)$ is formed by $k+1$ exceptional blocks $\B_{t}$, where $t=0,\ldots, k$. These exceptional blocks are formed by the following sets of dominant weights:
\begin{align*} 
 & \B_t= \YD_t^{2n-k-t}\times \YD_{n-k}^{\lfloor (k-t)/2 \rfloor} \qquad \text{for} \quad t=0,\ldots, k, 
\end{align*}
where $\YD_l^{m}\subset \YD_l$ denotes the set of those Young diagrams of length $l$ whose width is at most $m$, i.e.,
\begin{equation*}
 \YD_l^m \coloneqq \{(\lambda_1,\ldots,\lambda_l)\in \mathbb{Z}^{l}\ |\ m\ge \lambda_1\ge \lambda_2\ge \ldots \ge \lambda_l\ge 0\}.  
\end{equation*} 
Let $\cO(1)$ be the ample generator of the Picard group of $\IGr(k,V)$. We denote by 
\begin{equation*}
\cA_t(l)\coloneqq \langle \Sigma^{\alpha}\cU^{\vee}_k\otimes \cO(l)\otimes \Sigma^{\mu}_{\Sp}\cS_{2(n-k)}\rangle_{ (\alpha; \mu) \in \B_t}\subset \Db(\IGr(k,V))
\end{equation*} 
the minimal triangulated subcategory of $\Db(\IGr(k,V))$ that contains all $\cO(l)$-twists of equivariant vector bundles corresponding to the dominant weights from the exceptional block $\B_t$.

Since $\B_t$ is an exceptional block, the construction of Kuznetsov--Polishchuk gives the way to produce a full exceptional collection in the subcategory~$\cA_t(l)$ which will be denoted by
\begin{equation*}
\cA_t(l)= \langle \cE^{\alpha;\mu}_t(l) \rangle_{(\alpha; \mu) \in \B_t}.
\end{equation*}
The exceptional collection constructed in~\cite{KuzPol} is
\begin{equation*}
\scalebox{1.1}{$\displaystyle \langle \cA_0, \cA_1(1),\ldots,\cA_{k}(k)\rangle= \left \langle \langle \cE^{\alpha;\mu}_0 \rangle_{(\alpha; \mu) \in \B_0}, \langle \cE^{\alpha;\mu}_1(1) \rangle_{(\alpha; \mu) \in \B_1} ,\ldots, \langle \cE^{\alpha;\mu}_k(k) \rangle_{(\alpha; \mu) \in \B_k} \right\rangle.$} 
\end{equation*}
\begin{example} \label{example:Beilinson}
In the case $k=1$, that is the case of the projective space $\PP^{2n-1}$, there are two blocks:~$\B_0=(0)$ and~$\B_1=\YD_1^{2n-2}$, so the construction of Kuznetsov--Polishchuk gives $\Db(\PP^{2n-1})=\langle \cA_0,\cA_1(1)\rangle$, where
\begin{equation*}
\cA_0 = \langle \cO \rangle,\qquad \cA_1(1) = \langle \cO(1), \dots, \cO(2n-1) \rangle.
\end{equation*} 
Thus, in the case of~\(\PP^{2n-1}\) we recover Beilinson's collection.
\end{example}
\begin{example} \label{example:OinA}
In general, we have $\B_0= \YD_0^{2n-k}\times \YD_{n-k}^{\lfloor k/2 \rfloor}$, and therefore $\cA_0=\langle \Sigma^{\mu}_\Sp \cS_{2(n-k)} \rangle_{\mu\in \YD^{\lfloor k/2 \rfloor}_{n-k}}$. In particular,~$\cO\in \cA_0$.
\end{example}
\begin{example} \label{example:Bk}
Similarly, we have $\B_k= \YD_k^{2n-2k}\times \YD_{n-k}^{0}$, and therefore $\cA_k=\langle \Sigma^{\alpha} \cU^{\vee}_{k} \rangle_{\alpha\in \YD^{2n-2k}_{k}}$.
\end{example}
The main result of the present paper is that this collection consists of vector bundles and is full. The first part is proved in Theorem~\ref{theo:vector.bundle}, while the second part, fullness, is explained in Section~\ref{subsection:intro-main}.
\subsection{Sketch of the proof of Theorem~\ref{theo:intro}} \label{subsection:intro-main}
First, the cases~\(k=1\) and~\(n\) are already known by Example~\ref{example:Beilinson} and \cite{F19}. So, we assume~\(2 \le k \le n-1\).

We need to show that the subcategory
\begin{equation}\label{eq:def:A}
    \cA\coloneqq \langle \cA_0, \cA_1(1),\ldots,\cA_{k}(k)\rangle\subset \Db(\IGr(k,V))
\end{equation}
coincides with $\Db(\IGr(k,V))$. One way to show the equality $\cA=\Db(\IGr(k,V))$ is to prove that $\cA$ is stable under twists by $\cO(1)$, that is, $\cA(1)\subset \cA$. Because of the block structure of $\cA$, this inclusion is equivalent to 
\begin{equation*}
\cA_{t}(t+1)\subset \cA \quad \text{for} \quad t=0,\ldots, k.
\end{equation*} 
To prove this, we consider larger exceptional blocks (already introduced in~\cite{KuzPol})
\begin{align*} 
 &\tilde \B_t= \YD_t^{2n-k-t+1}\times \YD_{n-k}^{\lfloor (k-t)/2 \rfloor} \qquad \text{for} \quad t=0,\ldots, k,
\end{align*}
and subcategories $\tilde \cA_t\coloneqq \langle \Sigma^{\alpha}\cU^{\vee}_k\otimes \Sigma^{\mu}_{\Sp}(\cQ^{\vee}_{2n-k}/\cU_k) \rangle_{ (\alpha; \mu) \in \tilde \B_t}$ corresponding to them.
Note that 
\begin{equation*}
\B_t \subset \tilde \B_t, \quad \text{thus} \quad \cA_t\subset \tilde \cA_t.
\end{equation*} 
To prove that $\cA_{t}(t+1)\subset \cA$ we use induction on $t$ involving the larger exceptional blocks $\tilde \cA_{t}(t)$. Let us introduce the notation that we use throughout the paper
\begin{equation}\label{eq:def:D_t}
    \begin{aligned}
        \cD_0 &\coloneq \langle \cA_{0},\, \cA_{1}(1) \rangle;\\
        \cD_t &\coloneq \langle \cA_{t - 1}(t),\, \tilde \cA_{t}(t),\, \cA_{t + 1}(t+1) \rangle \qquad \text{for} \quad 1 \le t\le k-1.
    \end{aligned}
\end{equation}
The most intricate result of this paper is the following theorem.
\begin{theorem}\label{lemma:fullness_induction_step}
Suppose~\(0 \le t \le k - 1\) and~\(2\le k\le n-1\). Then~\(\cA_t(t+1),\, \tilde \cA_{t + 1}(t+1) \subset \cD_t\).
\end{theorem}
\begin{proof}
If~\(t=0\) this is Proposition~\ref{lemma:base_of_induction}.
If~\(k < n\), then for even~\(k - t\), this follows from Proposition~\ref{prop:mutations_even_t+1} and Proposition~\ref{cor:mutations_even t}. Finally, if~\(k - t\) is odd, we apply Corollary~\ref{corollary:odd}. 
\end{proof}
Let us show how the fullness of the Kuznetsov--Polishchuck collections follows from Theorem~\ref{lemma:fullness_induction_step}.
\begin{theorem}\label{theo:fullness}
We have
\begin{equation*}
\Db(\IGr(k,V))=\langle \cA_0, \cA_1(1),\ldots,\cA_{k}(k)\rangle.
\end{equation*}
\end{theorem}
\begin{proof}
Our goal is to show that $\cA(1) \subseteq \cA$. By Example~\ref{example:OinA}, we know that $\cO \in \cA_0\subset \cA$. Then, assuming the inclusion $\cA(1) \subseteq \cA$, it follows by induction that $\cO(i) \in \cA$ for all $i > 0$. By \cite[Theorem 4]{Or} for a smooth projective variety of dimension $d$, the object $\oplus_{i=0}^{d}\cO(i)$ is a classical generator of its bounded derived category. Thus, $\cA(1)\subseteq \cA$ implies $\cA=\mathrm{D}^b(\IGr(k,V))$.

To prove the inclusion $\cA(1) \subseteq \cA$, it suffices to show that 
\begin{equation*}
\cA_t(t+1)\subset \cA \qquad \text{for all} \quad 0\le t\le k.
\end{equation*} 

For $t\le k-1$, it will be more convenient to prove the following stronger statement:
\begin{equation} \label{themain}
\cA_t(t+1),\, \tilde \cA_{t + 1}(t+1)\subset \cA,   
\end{equation}
which we establish by induction on $t$. By Theorem~\ref{lemma:fullness_induction_step}
\begin{equation*}
\langle \cA_0(1),\, \tilde \cA_{1}(1)\rangle\subset \cD_0=\langle \cA_{0}, \cA_{1}(1)\rangle \subset \cA ,
\end{equation*}
so the base of induction $t=0$ follows. Suppose we proved~$\cA_t(t+1), \tilde \cA_{t + 1}(t+1)\subset \cA$ for some $t$; then using Theorem~\ref{lemma:fullness_induction_step} we obtain
\begin{equation*}
         \cA_{t+1}(t+2),\, \tilde \cA_{t + 2}(t+2)\subset \cD_{t+1}= \langle\cA_{t}(t+1),\, \tilde\cA_{t+1}(t+1),\, \cA_{t+2}(t+2)\rangle \subset \cA,
\end{equation*}
where the last inclusion follows from the induction hypothesis. Thus, the induction step follows, and we deduce~\eqref{themain}. 

It remains to show the inclusion $\cA_t(t+1)\subset \cA$ for $t=k$. For this we note that
\begin{equation*}
\cA_k(k+1)\subset \tilde \cA_{k}(k) = \left\langle\Sigma^{\lambda}\cU^\vee_k(k) \mid \lambda \in \YD_{k}^{2n - 2k+1}\right\rangle.
\end{equation*}
Thus, the inclusions~\eqref{themain} yield $\cA(1) \subseteq \cA$, and we deduce Theorem~\ref{theo:fullness}. 
\end{proof}
\begin{remark}\label{rem:Lagr_case-intr}
Our approach also allows us to prove Theorems~\ref{lemma:fullness_induction_step} and~\ref{theo:fullness} in the Lagrangian case~\(k=n\), see Remark~\ref{rem:Lagr_case}. This provides an alternative proof for the Lagrangian case, different from~\cite{F19}.
\end{remark}
\begin{proof}[Proof of Theorem~\ref{theo:intro}]
The first statement follows from Theorem~\ref{theo:vector.bundle}. The second statement follows from Theorem~\ref{theo:fullness} and \cite[Theorem 1.2]{F19} in the Lagrangian case $\IGr(n,V)$. 
\end{proof}

\subsection{Sketch of the proof of Theorem~\ref{lemma:fullness_induction_step}}
As we have already mentioned above, Theorem~\ref{lemma:fullness_induction_step} is the most intricate result of the paper, so let us make some comments on its proof. 

We need to show that every exceptional generator of $\cA_t(t+1), \tilde \cA_{t + 1}(t+1)$ is contained in $\cD_t$. We do this by constructing several two types of exact sequences which, roughly speaking, contain one object from~$\cA_t(t+1)$ or $\tilde \cA_{t + 1}(t+1)$, while the rest of the objects belong to $\cD_t$. These classes of exact sequences will be obtained from complexes, which we call the \textit{generalized staircase complexes}. Generalized staircase complexes are defined on the relative Grassmannians~$\Gr_{X}(k,\cV)$ for all~\(\GL_1 \times \GL_{k - 1}\)-dominant weights; see Section~\ref{section:generalized} for more details. 

The first type of exact sequences, which we call \textit{symplectic staircase complexes}, is defined on~$\IGr(2m,V)$ for all~\(m < n/2\) and is constructed as follows. We consider generalized staircase complexes on the flag variety~\(\Fl(2m,n;V)\) (which we consider as relative Grassmannian $\Gr_X(n-m,\cV)$, where $X=\Gr(2m, V)$ and~\(\cV = \cQ_{2n-2m}\)) and restrict them to the isotropic flags~\(\IFl(2m, n; V)\); the symplectic staircase complexes are then defined as the pushforwards of the obtained complexes to $\IGr(2m,V)$, see Subsection~\ref{subsection:symplectic} for more details. The symplectic staircase complexes allow us to prove Theorem~\ref{lemma:fullness_induction_step} for $t$ of the same parity as $k$; see Section~\ref{section:relation.sympl}. 

To define the second type of exact sequences, which we call \textit{secondary staircase complexes}, we need to introduce some special vector bundles~$\cK^\beta_k$ on~$\Gr(k,V)$. More precisely, for each~\(\GL_1 \times \GL_{k - 1}\)-dominant weight $\beta$ we consider the corresponding generalized staircase complex on $\Gr(k,V)$ and define $\cK^\beta_k$ as the cohomology of a certain truncation of this complex, see Definition~\ref{def:K}. The secondary staircase complexes are formed by vector bundles $\cK^{\beta}_k$, see Subsection~\ref{subsection:double} for more details; they allow us to prove Lemma~\ref{lemma:fullness_induction_step}, when~$k-t$ is odd; see Section~\ref{section:odd}.

\subsection{Organization of the paper}
The work is organized as follows. In Section~\ref{preliminaries} we discuss background material on semiorthogonal decompositions and equivariant bundles on flag varieties and state Borel--Bott--Weil theorems. In Section~\ref{section:generalized} we introduce generalized staircase complexes on relative Grassmannians, see Theorem~\ref{thm:gsc}, and define symplectic staircase complexes on isotropic Grassmannians, see Theorem~\ref{theo:ssc}. In Section~\ref{section:KP} we show that the exceptional objects that form the Kuznetsov--Polishchuk collection are vector bundles, see Theorem~\ref{theo:vector.bundle}. In Section~\ref{section:relation.sympl}, using the symplectic staircase complexes, we prove Theorem~\ref{lemma:fullness_induction_step} for~$t=0$ and~$t$ of the same parity as~$k$, see Proposition~\ref{lemma:base_of_induction} and Propositions~\ref{prop:mutations_even_t+1} and~\ref{cor:mutations_even t}, respectively. In Section~\ref{section:secondary.staircase} we define the secondary staircase complexes, see Theorem~\ref{theo:complexes_for_F^*,1^b}. In Section~\ref{section:odd}, using the secondary staircase complexes, we prove Theorem~\ref{lemma:fullness_induction_step}, when~$t$ has the opposite parity to~$k$, see Corollary~\ref{corollary:odd}.

In the appendix, we collect some computational results used throughout the paper.
\subsection{Notation} \label{section:notation}
\begin{itemize}
    \item $\Bbbk$, an algebraically closed field of characteristic zero;
    \item $V$, a $2n$-dimensional vector space over~$\Bbbk$ with a fixed symplectic form;
    \item $\bG\coloneqq \Sp(V)=\Sp(2n, \Bbbk)$, the symplectic group;
    \item $\bP\subset \bG$, the maximal parabolic subgroup; corresponding to the vertex $k$, where $1\le k\le n$ is a fixed integer;
    \item $\bU\subset \bP$, the unipotent radical of $\bP$;
    \item $\bL=\bP/\bU$, the Levi quotient of $\bP$;
    \item $\IGr(k,V)=\bG/\bP$, the isotropic Grassmannian which parameterizes $k$-dimensional isotropic subspaces in~$V$;
    \item $\YD_l^a\coloneqq \{(\lambda_1,\ldots,\lambda_l)\in \mathbb{Z}^{l}\ |\ a\ge \lambda_1\ge \lambda_2\ge \ldots \ge \lambda_l \ge 0\}$, the set of Young diagrams of length at most $l$ whose width is at most $a$;
    \item $a_t \coloneqq \lfloor (k-t)/2 \rfloor$ for $t=0,\ldots, k$.
\end{itemize}

{\bf Acknowledgements:}
We are grateful to Alexander Kuznetsov for his constant support and careful reading of a preliminary version of the paper.


\section{Preliminaries} \label{preliminaries}
\subsection{Semiorthogonal decompositions and exceptional collections} \label{subsection:semiorthogonal}
We recall some well-known facts about semiorthogonal decompositions. Let $\mathcal{T}$ be a $\Bbbk$-linear triangulated category.
\begin{definition}
A sequence of full triangulated subcategories $\mathcal{C}_1, \ldots ,\mathcal{C}_m \in \mathcal{T}$ is \textbf{semiorthogonal}
if for all $1\le i {}<{} j\le m$ and all $E_i\in \mathcal{C}_i$, $E_j\in \mathcal{C}_j$, one has \begin{equation*}
\mathrm{Hom}_{\mathcal{T}} (E_j,E_i) = 0.
\end{equation*}
Let $\langle \mathcal{C}_1, \ldots ,\mathcal{C}_m\rangle$ denote
the smallest full triangulated subcategory in $\mathcal{T}$ containing all $\mathcal{C}_i$.
If $\mathcal{C}_i$ are semiorthogonal and $\langle \mathcal{C}_1, \ldots ,\mathcal{C}_m \rangle = \mathcal{T}$,
we say that the subcategories $\mathcal{C}_i$ form a \textbf{semiorthogonal decomposition} of $\mathcal{T}$.
\end{definition}

\begin{definition}
An object $E$ of  $\mathcal{T}$ is \textbf{exceptional} if $\mathrm{Ext}^{\bullet}(E, E) = \Bbbk$  (i.e., $E$ is simple and has no non-trivial self-extensions).
\end{definition}

If $E$ is exceptional, then $\langle E \rangle$, the minimal triangulated subcategory of $\cT$ containing $E$,
is equivalent to~$\Db(\Bbbk)$, the bounded derived category of $\Bbbk$-vector spaces,
via the functor $\Db(\Bbbk) \to \cT$ that maps a graded vector space $W$ to $W \otimes E \in \cT$.

\begin{definition}
A sequence of exceptional objects $E_{1},\ldots, E_{m}$ in $\mathcal{T}$ is an \textbf{exceptional collection}
if each object~$E_{i}$ is exceptional and $\mathrm{Ext}^{\bullet}(E_{i}, E_{j} ) = 0$ for all $i > j$.
A collection $(E_{1},E_{2},\ldots ,E_{m})$ is \textbf{full} if the minimal triangulated subcategory of $\mathcal{T}$
containing $(E_{1},E_{2},\ldots, E_{m})$ coincides with $\mathcal{T}$.
\end{definition}
A full exceptional collection in $\cT$ is a special case of a semiorthogonal decomposition of $\cT$ with~$\mathcal{C}_k=\langle E_{k}\rangle$.

For an exceptional object $E \in \cT$ we denote by $\LL_E$ and $\RR_E$ the \textbf{left} and \textbf{right mutation} functors through $E$,
which are defined by taking an object $F \in \cT$ to
\begin{equation} \label{definition:mutation}
\mathbb{L}_{E}(F)\coloneqq\mathrm{Cone}(\mathrm{Ext}^{\bullet}(E, F)\otimes E \to F),
\qquad\text{and}\qquad
\mathbb{R}_{E}(F)\coloneqq \mathrm{Cone}(F \to \mathrm{Ext}^{\bullet}(F,E)^{\vee} \otimes E)[-1],
\end{equation}
where the morphisms are given by the evaluation and coevaluation, respectively.

More generally, if $(E_1, \ldots , E_m)$ is an exceptional collection of length~\(m\) in $\mathcal{T}$, then
one can define the left and right mutations of an object through the category $\langle E_1, \ldots , E_m\rangle$
as the compositions of the corresponding mutations through the generating objects:
\begin{equation*}
\LL_{\langle E_1, \ldots , E_m\rangle} = \LL_{E_1} \circ \ldots \circ \LL_{E_m},
\qquad
\RR_{\langle E_1, \ldots , E_m\rangle} = \RR_{E_m} \circ \ldots \circ \RR_{E_1}.
\end{equation*}
\begin{proposition}[\cite{bondal1989representation}]
\label{proposition:mutations}
The functors of left and right mutations through an exceptional collection
induce mutually inverse equivalences of the left and the right orthogonals to the collection:
\begin{equation*}
\xymatrix@1@C=7em{
{}^\perp\langle E_1, \ldots , E_m\rangle \ar@<.5ex>[r]^{\LL_{\langle E_1, \ldots , E_m\rangle}} &
\langle E_1, \ldots , E_m\rangle^\perp \ar@<.5ex>[l]^{\RR_{\langle E_1, \ldots , E_m\rangle}}
}
\end{equation*}
Mutation of a \textup(full\textup) exceptional collection is a \textup(full\textup) exceptional collection.
\end{proposition}
\begin{definition} \label{def:right-dual}
Let $E_{1},\ldots, E_{m}$ be an exceptional collection. The \textbf{right dual} exceptional collection~$(\lvee{E}_m, \lvee{E}_{m-1}, \ldots, \lvee{E}_1)$ is 
\begin{equation*}
    \lvee{E}_m\coloneqq E_m,\ \lvee{E}_{m-1}\coloneqq \mathbb{R}_{E_m}E_{m-1},\ \lvee{E}_{m-2}\coloneqq \mathbb{R}_{\langle E_{m-1},E_m \rangle}E_{m-2},\ \ldots,\ 
    \lvee{E}_1\coloneqq \mathbb{R}_{\langle E_2, E_3, \ldots, E_{m-1},E_m \rangle}E_1.
\end{equation*}
\end{definition}
An object $\lvee{E}_i$ of the right dual exceptional collection can be characterized by the following conditions:
\begin{equation} \label{dual_general}
 \lvee{E}_i\in \langle E_1,E_2,\ldots, E_m\rangle, \quad 
  \mathrm{Ext}^\bullet(\lvee{E}_i, E_j)=
  \begin{cases}
		\Bbbk, & \text{if} \ i=j;\\
		0, & \text{if}\ i\ne j.
    \end{cases} 
\end{equation}

\subsection{Schur functors and Young diagrams} \label{subsection:schur}
Let $W$ be a vector space of dimension $n$.
We will use the standard identification of the weight lattice of the group $\mathrm{GL}(W)$ with $\mathbb{Z}^{n}$
that takes the fundamental weight of the representation~$\Lambda^{l}W^{\vee}$ to the vector $(1, 1,\ldots, 1, 0, 0,\ldots , 0) \in \mathbb{Z}^{n}$
(the first $l$ entries are $1$, and the last $n-l$ are $0$). The sum of the fundamental weights of $\mathrm{GL}(W)$ is equal to 
\begin{equation*}
\rho \coloneqq (n, n- 1,\ldots, 2, 1).
\end{equation*}

The cone of dominant weights of $\mathrm{GL}(W)$ is identified with the set of weakly decreasing sequences of length $n$
\begin{equation} \label{notation_decr}
 \Y_n \coloneqq \{(\lambda_1,\ldots,\lambda_n)\in \mathbb{Z}^{n}\ |\ \lambda_1\ge \lambda_2\ge \ldots \ge \lambda_n\}.
\end{equation}
For $\alpha\in \Y_n$ we denote by $\Sigma^{\alpha}W^{\vee} = \Sigma^{\alpha_{1},\alpha_{2},\ldots,\alpha_{n}}W^{\vee}$
the corresponding representation of $\mathrm{GL}(W)$ of highest weight $\alpha$. 

The Weyl group of $\mathrm{GL}(W)$ is isomorphic to the permutation group~$\mathfrak{S}_{n}$,
and the length function~\(\ell: \mathfrak{S}_{n}\to \mathbb{Z}\) counts the number of inversions in a permutation.

Given a vector bundle $\cV$ of rank $n$ on a scheme $X$, we consider the corresponding principal $\mathrm{GL}(n)$-bundle on $X$ and denote by $\Sigma^{\alpha}\cV$ the vector bundle associated with the $\mathrm{GL}(n)$-representation of highest weight $\alpha\in \Y_n$.

In what follows, for a dominant weight of the form $(\alpha_1,\ldots, \alpha_l,0,\ldots, 0)$, we omit zeros,
i.e. we write just~$(\alpha_1,\ldots,\alpha_l)$~and~$\Sigma^{\alpha_1,\ldots,\alpha_l}\cV$ for such a weight. Note that 
\begin{equation*}
\Sigma^{a,0,\ldots,0}\cV = S^a \cV,  \qquad
\Sigma^{(1)^l}\cV = \Lambda^l \cV.
\end{equation*}
We frequently use the following natural identifications:
\begin{align*}
  &\Sigma^{\alpha}\cV \simeq \Sigma^{-\alpha}\cV^\vee,\\
  & \Sigma^{\alpha-m}\cV \simeq \Sigma^{\alpha}\cV\otimes (\mathrm{det}\cV^{\vee})^{\otimes m},
\end{align*}
where $m\in \mathbb{Z}$ and $-\alpha \coloneqq (-\alpha_n, -\alpha_{n-1},\ldots,-\alpha_1)$.

\subsubsection{Symplectic Schur functors}
Now, let $V$ be a $2n$-dimensional vector space with a fixed symplectic form. The weight lattice of the corresponding symplectic group $\mathrm{Sp}(V)$ can be identified with $\mathbb{Z}^{n}$: under this identification, as before, the $l$-th fundamental weight goes to $(1,\dots,1,0,\dots,0)$. The sum of the fundamental weights of $\mathrm{Sp}(V)$ is also equal to $\rho$. The cone of dominant weights of $\mathrm{Sp}(V)$ is identified with the set of Young diagrams~$\YD_n$ of length at most $n$:
\begin{equation} \label{notation_young}
 \YD_n \coloneqq \{(\lambda_1,\ldots,\lambda_n)\in \mathbb{Z}^{n}\ |\ \lambda_1\ge \lambda_2\ge \ldots \ge \lambda_n \ge 0\} \subset \Y_n 
\end{equation}
For $\alpha \in \YD_n$ we denote by $\Sigma^{\alpha}_{\mathrm{Sp}}V = \Sigma^{\alpha_{1},\alpha_{2},\ldots,\alpha_{n}}_{\mathrm{Sp}}V$ the corresponding representation of $\mathrm{Sp}(V)$. 

Similarly, given a symplectic vector bundle $\cV$ of rank $2n$ on a scheme $X$,
we consider the corresponding principal~$\mathrm{Sp}(V)$-bundle on $X$ and denote by $\Sigma^{\alpha}_{\mathrm{Sp}}\cV$ the vector bundle associated with the~$\mathrm{Sp}(V)$-representation of highest weight $\alpha$. 

The Weyl group of $\mathrm{Sp}(V)$ is equal to the semidirect product of $\mathfrak{S}_{n}$ and $(\mathbb{Z}/2\mathbb{Z})^{n}$,
where $\mathfrak{S}_{n}$ acts on the weight lattice $\mathbb{Z}^n$ by permutations and $(\mathbb{Z}/2\mathbb{Z})^{n}$
acts by changes of signs of the coordinates. 

\subsubsection{Notation}
For $\alpha\in \Y_n$ and $m\in \mathbb{Z}$ we denote by:
\begin{align*}
\bar \alpha & \coloneqq (\alpha_2,\alpha_3,\ldots,\alpha_n);\\
|\alpha| & \coloneqq \alpha_1+ \alpha_2+ \ldots+\alpha_n;\\
-\alpha  & \coloneqq (-\alpha_n, -\alpha_{n-1},\ldots,-\alpha_1);\\
 \alpha-m & \coloneqq (\alpha_1-m, \alpha_2-m,\ldots,\alpha_n-m);\\
((m)^t,\alpha) & \coloneqq (\underbrace{m,\ldots,m}_t,\alpha_1, \alpha_2, \dots, \alpha_n).
\end{align*}
\subsubsection{Decomposition of tensor products of Schur functors}
Let $\cV$ be a vector bundle of rank $n$ on a scheme $X$. Since the group $\GL_n$ is reductive the tensor product $\Sigma^{\alpha}\cV\otimes \Sigma^{\beta}\cV$ can be decomposed into a direct sum
of bundles of the form $\Sigma^{\gamma}\cV$ (the irreducible summands). The result is
\begin{equation*}
    \Sigma^{\alpha}\cV\otimes \Sigma^{\beta}\cV \simeq \bigoplus (\Sigma^{\gamma}\cV)^{c_{\alpha, \beta}^{\gamma}}
\end{equation*}
where the numbers $c_{\alpha,\beta}^{\gamma}$ are determined by the Littlewood–Richardson rule, see \cite[Appendix A]{fulton2011}. If~$c_{\alpha,\beta}^{\gamma}\ne 0$, we write 
\begin{equation*}
\Sigma^{\gamma}\cV \inplus \Sigma^{\alpha}\cV\otimes \Sigma^{\beta}\cV.
\end{equation*}
To compute $c_{\alpha,\beta}^{\gamma}$, we will need the following version of the Littlewood--Richardson rule. 

\begin{lemma}[{\cite[Lemma 2.9]{KuzPol}}]\label{lemma:LR rule}
Let~\(\bG\) be a classical group and $\cG$ be a $\bG$-equivariant vector bundle. Let $\alpha$ and $\beta$ be $\bG$-dominant weights and
\begin{equation*}
\Sh_{\bG}^\alpha\cG \otimes \Sh_{\bG}^\beta\cG = \bigoplus (\Sh_\bG^\gamma \cG)^{c_{\alpha, \beta}^{\gamma}}
\end{equation*}
be the direct sum decomposition of the tensor product of the corresponding Schur functors. 
Then 
\begin{equation*}
\gamma\in \operatorname{conv} \langle \alpha+ \sigma(\beta) \rangle \cap \ZZ^n,
\end{equation*}
where the convex hull $\operatorname{conv}$ is taken over all elements~$\sigma$ of the Weyl group $W(\bG)$ of~\(\bG\) and~\(\gamma\) lies in the~\(\bG\)-dominant Weyl chamber. 
\end{lemma}

\begin{lemma}[{\cite[Corollary~3.8]{ku10}}]\label{lemma:PRV}
In the setting of Lemma~\ref{lemma:LR rule}, if~\(\gamma=\alpha+w\beta\) is~\(\bG\)-dominant for an element of the Weyl group~\(w \in W(\bG)\), then~\(c_{\alpha,\beta}^{\gamma} = 1\).
\end{lemma}

There is the standard partial order on the set $\Y_n$ (and, in particular, on the set $\YD_n$): for~$\alpha,\beta\in \Y_n$
\begin{equation}\label{eq:order-GL}
   \alpha\le \beta \ \Leftrightarrow \ \alpha_i\le \beta_i \quad \text{for}\quad i\in [1,n].
\end{equation}

Let $\alpha,\beta\in \YD_n$ be such that $\beta\le \alpha$. Then we can associate with such a pair of Young diagrams the skew Schur functor $\Sigma^{\alpha/\beta}$, see \cite[Section 2.1]{Weyman}, which satisfies the property 
\begin{equation}
  \Sigma^{\alpha/\beta} \cG \simeq \bigoplus (\Sigma^{\nu} \cG)^{\oplus c^{\alpha}_{\nu,\beta}},
\end{equation}
where $c^{\alpha}_{\nu,\beta}$ are the Littlewood--Richardson coefficients. We will need the following
statement.
\begin{lemma}[\cite{Weyman}] \label{skew_decomposition}
Let $\alpha\in \YD_m$ and let
\begin{equation*}
0\to \cG_1\to \cG \to \cG_2\to 0
\end{equation*}
be a short exact sequence of vector bundles. Then there is a filtration on $\Sigma^{\alpha}\cG$ with the associated graded isomorphic to 
\begin{equation*}
 \bigoplus_{0\le\beta\le \alpha} \Sigma^{\beta}\cG_1 \otimes \Sigma^{\alpha/\beta}\cG_2.
\end{equation*}
\end{lemma}

\subsection{Equivariant bundles on flag varieties}
Let $X$ be a variety and~\(\cV\) be a rank $r$ vector bundle over $X$. For an integer $0<l<r$, we denote by
\begin{equation}\label{eq:pi_l}
    \pi_l\colon \Gr_{X}(l,\cV)\to X,
\end{equation}
the relative Grassmannian of $l$-dimensional subspaces in the fibers of $\cV$. In a special case, when $X = pt$ is a point and $\cV = V$ is a vector space, we use the notation $\Gr(l,V)$. 

The tautological vector subbundle of rank $l$ on $\Gr_{X}(l,\cV)$ is denoted by $\cU_l \subset \pi_l^*\cV$ and the quotient bundle is denoted by $\cQ_{r-l}$. The tautological exact sequences on $\Gr_{X}(l,\cV)$ in this notation take the form:
\begin{align} \label{tautological_exact_sequence}
    & 0\to \cU_l \to \pi_l^*\cV \to \cQ_{r-l} \to 0;\\ \label{tautological_exact_sequence_1}
    & 0\to \cQ^{\vee}_{r-l} \to \pi_l^*\cV^{\vee} \to \cU^{\vee}_l \to 0.
\end{align}
We denote by $\cO(H_{l}) \coloneqq \Lambda^{l}\cU^{\vee}_{l}$ the ample generator of $\mathrm{Pic}(\Gr_{X}(l,\cV)/X)$. Abusing the notation, we denote the pullback $\pi_l^*\cV$ by the same letter~$\cV$, unless this leads to confusion. We denote by $\Fl_{X}(l_1, l_2; \cV)$ the relative flag variety. Abusing the notation, we denote by $\cU_{l_1}\subset \cU_{l_2} \subset \cV$ the tautological flag of the subbundles on $\Fl_{X}(l_1, l_2; \cV)$.

Now assume that~\(\cV\) is endowed with the symplectic structure, in particular the rank of $\cV$ is even~$r=2n$. For an integer $0<l\le n$ we denote by
\begin{equation*}
    \pi_l\colon \IGr_{X}(l,\cV)\to X,
\end{equation*}
the relative isotropic Grassmannian of $l$-dimensional isotropic subspaces in the fibers of~$\cV$. When $X = pt$ is a point and $\cV = V$ is a vector space, we use the notation $\IGr(l,V)$. The isotropic flag variety, that parametrize isotropic flags of vector subspaces in $V$, will be denoted by $\IFl(t, k; V)$.

Note that $\IGr_{X}(l,\cV)\hookrightarrow \Gr_X(l,\cV)$. Abusing the notation, we will denote the restrictions of vector bundles from $\Gr_X(l,\cV)$ to $\IGr_{X}(l,\cV)$ by the same letters. Note that the symplectic structure on $\cV$ allows to define the canonical embedding of vector bundles $\cU_l\hookrightarrow \cQ^{\vee}_{2n-l}$ on~$\IGr_{X}(l,\cV)$. So on~$\IGr_{X}(l,\cV)$ there is a quotient bundle
\begin{equation*}
\cS_{2(n-l)}\coloneqq \cQ^{\vee}_{2n-l}/\cU_l
\end{equation*}
with the symplectic structure induced from $\cV$.

As was already explained in Subsection~\ref{intro:KP}, every irreducible $\bG$-equivariant vector bundle
on $\IGr(k,V)$ is isomorphic to
\begin{equation*}
\Sigma^{\alpha}\cU^{\vee}_k\otimes \Sigma^{\mu}_{\Sp}\cS_{2(n-k)},
\end{equation*}
where $\alpha\in \Y_k$ and $\mu\in \YD_{n-k}$.

For $\IFl(k_1,k_2, V)=\bG/\bP_{k_1,k_2}$ the Levi quotient $\bL_{k_1,k_2}$ of $\bP_{k_1,k_2}$ is isomorphic to \begin{equation*}
\GL(k_1)\times \GL(k_2-k_1)\times \Sp(2(n-k_2), \Bbbk),
\end{equation*} so the dominant cone of $\bL_{k_1,k_2}$ can be identified with $\Y_{k_1}\times \Y_{k_2-k_1}\times \YD_{n-k_2}$. As was already explained in Subsection~\ref{intro:KP}, each dominant weight $(\alpha; \beta; \mu)\in \Y_{k_1}\times \Y_{k_2-k_1}\times \YD_{n-k_2}$ of $\bL_{k_1,k_2}$ corresponds to the irreducible $\bG$-equivariant vector bundle on $\IFl(k_1,k_2, V)$ which is
\begin{equation} \label{equivariant_bundles_on_flags}
  \Sigma^{\alpha}\cU^{\vee}_{k_1}\otimes \Sigma^{\beta}(\cU_{k_2}/\cU_{k_1})^{\vee}\otimes \Sigma^{\mu}_{\Sp}\cS_{2(n-k_2)}.
\end{equation}

\subsection{Borel--Bott--Weil theorems}
Here we explain a way to compute the pushforwards of equivariant sheaves for the relative (isotropic) Grassmannian.
\begin{theorem}[Relative Borel--Bott--Weil Theorem, see~{\cite[Theorem~IV']{Bott57}}]\label{theo:relative_BBW}
Let $X$ be a variety and let~$\cW$ be a rank~$r$ vector bundle on $X$. For an integer $0 < s < r$, consider the relative Grassmannian
\begin{equation*}
\pi_s \colon \Gr_{X}(s, \cW)\to X.
\end{equation*}
Let $\lambda\in \mathrm{Y}_{s}$, $\mu \in \mathrm{Y}_{r-s}$, and denote $(\lambda,\mu)\in \mathbb{Z}^r$ their
concatenation. Put $\rho=(r, r-1, \ldots, 1)$. Then
\begin{equation*}
  R^{\bullet}\pi_{s*}(\Sigma^{\lambda}\cU_{s}^{\vee}\otimes \Sigma^{\mu}\cQ^{\vee}_{r-s})= 
\begin{cases}
\Sigma^{\sigma((\lambda, \mu)+\rho)-\rho}\,\cW^{\vee}[-\ell(\sigma)],&\text{if all elements in $(\lambda, \mu)+\rho$ are distinct;}\\
0,&\text{otherwise,}
\end{cases}  
\end{equation*}
where $\sigma \in \mathfrak{S}_{r}$ denotes the unique permutation such that the sequence $\sigma((\lambda, \mu)+\rho)$ is strictly decreasing and $\ell(\sigma)$ is the number of pairs $1\le i\le j\le r$ such that $\sigma(i)> \sigma(j)$. 
\end{theorem}

\begin{theorem}[Relative symplectic Borel--Bott--Weil Theorem, see~{\cite[Theorem~IV']{Bott57}}]\label{theo:relative_BBW_sym}
Let $X$ be a variety and let~$\cV$ be a symplectic rank~$2r$ vector bundle on $X$. For an integer $0 < t < 2r$, consider the relative isotropic Grassmannian 
\begin{equation*}
\pi_t \colon \IGr_{X}(t, \cV)\to X.
\end{equation*}
Let $\lambda\in \mathrm{Y}_{t}$, $\mu \in \mathrm{YD}_{r-t}$ and denote $(\lambda;\mu)\in \mathbb{Z}^r$ their concatenation. Put $\rho=(r, r-1, \ldots, 1)$. Then
\begin{multline*}
  R^{\bullet}\pi_{t*}(\Sigma^{\lambda}\cU_{t}^{\vee}\otimes \Sigma^{\mu}_{\Sp}\cS_{2(r-t)})= \\
  \begin{cases}
\Sigma_{\Sp}^{\sigma((\lambda, \mu)+\rho)-\rho}\,\cV[-\ell(\sigma)],&\text{if absolute values of all elements in $(\lambda, \mu)+\rho$ are distinct and non-zero;}\\
0,&\text{otherwise,}
\end{cases}  
\end{multline*}
where $\sigma\in \mathfrak{S}_{r}\ltimes (\mathbb{Z}/2\mathbb{Z})^{r}$ is the unique element of the Weyl group such that the sequence $\sigma((\lambda, \mu)+\rho)$ is a strictly decreasing sequence of positive integers and $\ell(\sigma)$ is the length of $\sigma$. 
\end{theorem}
\subsection{Filtration}
Throughout the paper, we use the following statement about filtrations of $\bG$-equivariant vector bundles on $\IFl(k - t, k; V)$, where the associated subquotients are irreducible $\bG$-equivariant bundles.
\begin{lemma}\label{lemma:filtration_on_flags}
    Let~\(\xi \in \Y_t\), 
    and~\(\theta \in \YD^{a}_{n - k + t}\), where $a\in \mathbb{Z}_{\ge 0}$. The vector bundle
    \begin{equation*}
\Sigma^{\xi} (\cU_k/\cU_{k - t}) \otimes p_t^*\Sigma_\Sp^\theta \cS_{2(n - k + t)}
\end{equation*} 
    on $\IFl(k-t,k;V)$ admits a filtration whose associated graded pieces are irreducible $\bG$-equivariant vector bundles of the form
\begin{equation*} 
  \Sigma^{\gamma} (\cU_k/\cU_{k - t}) \otimes \Sigma_\Sp^\eta \cS_{2(n - k)},
    \end{equation*}
   where~\(\gamma \in \Y_t\) satisfies~\(\gamma_i \in [\xi_i - a, \xi_i + a]\), and~\(\eta \in \YD^{a}_{n - k}\) is contained in $\theta$.
\end{lemma}
\begin{proof}
Recall that the irreducible $\bG$-equivariant vector bundles on $\IFl(k-t,k,V)$ are precisely those of the form
\begin{equation*} 
  \Sigma^{\lambda}\cU^{\vee}_{k-t}\otimes \Sigma^{\kappa}(\cU_{k}/\cU_{k-t})^{\vee}\otimes \Sigma^{\mu}_{\Sp}\cS_{2(n-k)},
\end{equation*}
where $\lambda\in \Y_{k-t}$, $\kappa\in \Y_{t}$ and $\mu\in \YD_{n-k}$.

The pullback bundle $p_t^*\cS_{2(n-k+t)}$ admits a filtration whose associated graded pieces are irreducible $\bG$-equivariant vector bundles of the form
\begin{equation*}
p_t^*\cS_{2(n-k+t)}=[(\cU_{k}/\cU_{k-t}), \cS_{2(n-k)},(\cU_{k}/\cU_{k-t})^{\vee}].
\end{equation*}

For a symplectic vector bundle $\cS$, the symplectic Schur functor $\Sigma^{\theta}_{\Sp}\cS$ is a quotient of the classical Schur functor $\Sigma^{\theta}\cS$. Therefore, by Lemma~\ref{skew_decomposition}, the associated subquotients of the bundle \begin{equation*}
\Sigma^{\xi} (\cU_k/\cU_{k - t}) \otimes p_t^*\Sigma_\Sp^\theta \cS_{2(n - k + t)}
\end{equation*} are of the form
\begin{equation}\label{eq:filtr_flags}
 \Sigma^{\xi} (\cU_k/\cU_{k - t})\otimes \Sigma^{\alpha}(\cU_{k}/\cU_{k-t})\otimes \Sigma^{\beta}(\cU_{k}/\cU_{k-t})^{\vee}\otimes \Sigma^{\nu}_{\Sp}\cS_{2(n-k)},   
\end{equation}
where $\alpha,\beta,\nu\subset \theta$.

In particular, if the bundle
\begin{equation*}
\Sigma^{\gamma} (\cU_k/\cU_{k - t}) \otimes \Sigma_\Sp^\eta \cS_{2(n - k)},
\end{equation*} 
appears as a subquotient in this filtration, then Lemma~\ref{lemma:LR rule} implies that~\(\xi - a \le \gamma \le \xi + a\) for each $i$, and~$\eta\subset \theta$ with $\eta \in \YD^{a}_{n - k}$, as claimed.
\end{proof}

\section{Generalized staircase and symplectic staircase complexes} \label{section:generalized}
In this section we introduce a mild generalization of staircase complexes on relative Grassmannians
and define a new important class of exact sequences --- symplectic staircase complexes on isotropic Grassmannians.

\subsection{Reminder: Koszul and staircase complexes}
First recall the well--known notions of Koszul complexes and staircase complexes.

Let $X$ be a variety and let~$\cV$ be a rank~$r$ vector bundle on $X$. For all~$m\in \mathbb{N}$ there is the long exact sequence on~$\Gr_X(l,\cV)$ obtained from~\eqref{tautological_exact_sequence_1}, which we call the \textbf{Koszul complex}
\begin{equation} \label{Koszul}
    0 \to \Lambda^m\cQ_{r-l}^{\vee}\to \Lambda^m\cV^{\vee}\otimes \cO \to \Lambda^{m-1}\cV^{\vee}\otimes \cU^{\vee}_l\to \ldots\to \cV^{\vee}\otimes S^{m-1}\cU^{\vee}_l\to S^{m}\cU^{\vee}_l \to 0.
\end{equation}

Next, recall the definition of staircase complexes that was introduced in~\cite{Fonarev_2013}.
Given a Young diagram $\alpha \in \YD_{k}^{2n-k}$ with~\(\alpha_1 = 2n-k\), we define a sequence of diagrams $\alpha^{(1)}, \ldots, \alpha^{(2n - k)}$ as follows. For each index $i \in [1,2n-k]$, let $j$ be the largest index such that $\alpha_j - 1 \ge \alpha_1 -i$. Then 
\begin{equation} \label{remaind}
  \alpha^{(i)}\coloneqq(\alpha_2-1, \alpha_3-1, \ldots, \alpha_j-1, \alpha_1-i,\alpha_{j+1}, \ldots, \alpha_{k}).
\end{equation}
We also define $b^i\coloneqq|\alpha| - |\alpha^{(i)}|$. \textbf{The staircase complex} on $\Gr(k,V)$ 
is defined as follows:
\begin{equation*}
0 \to \Sigma^{\bar{\alpha}} \cU^{\vee}_k(-1) \to \Lambda^{b^{2n-k}} V\otimes \Sigma^{\alpha^{(2n-k)}}\cU^{\vee}_k \to 
\ldots \to \Lambda^{b^{2}} V \otimes \Sigma^{\alpha^{(2)}} \cU^{\vee}_k \to \Lambda^{b^{1}} V \otimes \Sigma^{\alpha^{(1)}} \cU^{\vee}_k \to \Sigma^{\alpha} \cU^{\vee}_k \to 0,
\end{equation*} 
where $\bar \alpha=(\alpha_2,\ldots, \alpha_k)$.

\subsection{Combinatorial setup.}
We start with some combinatorial notation that will be useful later; it is motivated by a Borel--Bott--Weil computation on $\PP_{\Gr(k,\cV)}(\cU_k)$ for a weight
\begin{equation}
\label{eq:weight-t-bal}
(z;\alpha_2,\dots,\alpha_k) \in \mathbb{Z} \times \Y_{k-1},
\end{equation}
see the proof of Theorem~\ref{thm:gsc} below for the details.

\newcommand{\BBW}{\operatorname{BBW}}
\newcommand{\BBWs}{\operatorname{BBW}^\Sp}
We associate with~$\bar \alpha \coloneqq (\alpha_2,\alpha_3,\dots,\alpha_k)$ the following set of {\sf vanishing values} 
(the values of~$z$ for which the bundle  $S^z\cU_1^{\vee}\otimes \Sigma^{\bar \alpha}\cQ^{\vee}_{k-1}$ with the weight~\eqref{eq:weight-t-bal} has no cohomology on~$\PP_{\Gr(k,\cV)}(\cU_k)$):
\begin{equation}
\label{eq:def-bv}
\bV(\bar \alpha) \coloneqq \{\alpha_k - k + 1, \dots, \alpha_2 - 1\} \subset \ZZ
\end{equation}

Further, we consider the counting function
\begin{equation}
\label{eq:def-dd}
\dd_{\bar \alpha} \colon \mathbb{Z} \setminus \bV(\bar \alpha) 
\to \ZZ,
\qquad 
\dd_{\bar \alpha}(z) \coloneqq |(z,\infty) \cap \bV(\bar \alpha)| + 1. \quad 
\end{equation}
This is a function with jumps equal to~$-1$ at each point of~$\bV(\bar \alpha)$, 
equal to~$k$ near~$-\infty$ and~$1$ near~$+\infty$.
It equals the number (counting from right to left, starting from~1) of the connected component of~$\RR \setminus \bV(\bar \alpha)$ to which~$z$ belongs.

Finally, we consider the following {\sf BBW map}~$\BBW_{\bar \alpha} \colon \ZZ \setminus \bV(\bar \alpha) \to \Y_k$ 
(computing the cohomology of the bundle $S^z\cU_1^{\vee}\otimes \Sigma^{\bar \alpha}\cQ^{\vee}_{k-1}$ on~$\PP_{\Gr(k,\cV)}(\cU_k)$)
defined by
\begin{equation}
\label{eq:def-bbw}
\BBW_{\bar \alpha}(z) \coloneqq
\begin{cases}
(z,\alpha_2,\alpha_3,\alpha_4, \dots,\alpha_k), & 
\text{if~$\dd_{\bar \alpha}(z) = 1$},\\
(\alpha_2-1,z+1,\alpha_3,\alpha_4, \dots,\alpha_k), & 
\text{if~$\dd_{\bar \alpha}(z) = 2$},\\
(\alpha_2-1,\alpha_3-1,z+2,\alpha_4,\dots,\alpha_k), & 
\text{if~$\dd_{\bar \alpha}(z) = 3$},\\
\hspace{3cm}\dots \\
(\alpha_2-1,\alpha_3-1,\dots, \alpha_{k-1}-1, z + k - 2,\alpha_k), & 
\text{if~$\dd_{\bar \alpha}(z) = k - 1$},\\
(\alpha_2-1,\alpha_3-1,\dots, \alpha_{k-1}-1, \alpha_k - 1, z + k - 1), & 
\text{if~$\dd_{\bar \alpha}(z) = k$},\\
\end{cases}
\end{equation} 
\subsection{Generalized staircase complexes} \label{subsection:generalized}
We are going to use this combinatorics to produce an exact sequence of vector bundles on a relative Grassmannian of a vector bundle $\Gr_X(k,
\cV)$.
We denote by~$r$ the rank of a vector bundle $\cV$.

We fix a $\GL_1\times \GL_{k-1}$-dominant weight $\alpha=(\alpha_1;\alpha_2,\dots,\alpha_k) \in \mathbb{Z} \times \Y_{k-1}$ and consider the interval~$[\alpha_1 - r, \alpha_1]$, which we think of as a sequence of~$r + 1$ consecutive integers. 
We remove all vanishing values from this interval and enumerate the remaining integers decreasingly:
\begin{equation}
\label{eq:def-ti}
\{z_0 > z_1 > \dots > z_{\mm(\alpha)}\} \coloneqq [\alpha_1-r,\alpha_1] \cap \ZZ \setminus \bV(\bar \alpha),
\qquad 
\mm(\alpha) \coloneqq |[\alpha_1-r,\alpha_1] \cap \ZZ \setminus \bV(\bar \alpha)|-1.
\end{equation}
Note that
\begin{equation*}
r - k + 1 \le \mm(\alpha) \le r.
\end{equation*}
Furthermore, counting the number of elements of~\(\bV(\bar \alpha)\) between~\(z_{i}\) and~\(z_{i+1}\) for~$i \in [0,\mm(\alpha)-1]$ we get:
\begin{equation} \label{difference_1}
\dd_{\bar \alpha}(z_{i+1})-\dd_{\bar \alpha}(z_{i})= z_{i}-z_{i+1}-1.
\end{equation}
Now we can define the complex.
\begin{theorem}
\label{thm:gsc}
Let~$\cV$ be a vector bundle of rank~$r > k$ on a scheme~$X$.
For each $\GL_1\times \GL_{k-1}$-dominant weight~$\alpha = (\alpha_1;\alpha_2,\alpha_3,\dots,\alpha_k) \in \Y_1\times \Y_{k-1}$,
if~$m = \mm(\alpha)$ and~$z_0,\dots,z_{\mm(\alpha)}$ are defined by~\eqref{eq:def-ti},
there is an exact sequence of vector bundles 
\begin{multline}
\label{eq:gsc}
0 \to 
\wedge^{b^m}\cV^\vee \otimes \Sigma^{\alpha^{(m)}}\cU^\vee \to 
\wedge^{b^{m-1}}\cV^\vee \otimes \Sigma^{\alpha^{(m-1)}}\cU^\vee \to \dots \\
\dots \to 
\wedge^{b^1}\cV^\vee \otimes \Sigma^{\alpha^{(1)}}\cU^\vee \to 
\wedge^{b^0}\cV^\vee \otimes \Sigma^{\alpha^{(0)}}\cU^\vee \to 
0
\end{multline}
on the relative Grassmannian~$\Gr_X(k,\cV)$, where
\begin{equation}
\label{eq:def-ab-i}
\alpha^{(i)} \coloneqq \BBW_{\bar \alpha}(z_i),
\qquad 
b^i \coloneqq \alpha_1 - z_i.
\end{equation}
If an algebraic group~$G$ acts on~$X$ and~$\cV$ is $G$-equivariant, this sequence is also $G$-equivariant.
\end{theorem}

\begin{proof}
We apply the Fourier--Mukai transform with kernel~$\Sigma^{\bar \alpha}(\cU_k/\cU_1)^\vee$ on~$\Fl_X(1,k;\cV)$
to the twist of the Koszul complex~\eqref{Koszul}
\begin{equation}
\label{eq:gsc-Koszul}
0 \to 
\wedge^r\cV^\vee \otimes \cO_{\PP_X(\cV)}(\alpha_1-r) \to 
\dots \to
\cV^\vee \otimes \cO_{\PP_X(\cV)}(\alpha_1-1) \to 
\cO_{\PP_X(\cV)}(\alpha_1) \to 
0.
\end{equation}

To compute the image of the vector bundle~$\wedge^{\alpha_1-z}\cV^\vee \otimes \cO_{\PP_X(\cV)}(z)$, where~$z \in [\alpha_1-r,\alpha_1] \cap \ZZ$,
we apply Theorem~\ref{theo:relative_BBW} (Borel--Bott--Weil) to the tautological bundle associated with the weight
\begin{equation*}
(z;\alpha_2,\dots,\alpha_k;0,\dots,0) \in \Y_1 \times \Y_{k-1} \times \Y_{r-k}.
\end{equation*}
Therefore, for~$z \in \bV(\bar \alpha)$, the result vanishes. Thus, the image is nonzero only for~~$z \in \{z_0,z_1,\dots,z_{m(\alpha)}\}$ and is given by $\BBW_{\bar \alpha}(z)[1-\dd_{\bar \alpha}(z)]$. 

Thus, for~$z = z_i$ the image of $\wedge^{\alpha_1-z}\cV^\vee \otimes \cO_{\PP_X(\cV)}(z)[-z]$ is 
\begin{equation*}
\wedge^{b^i}\cV^\vee \otimes \Sigma^{\alpha^{(i)}}\cU^\vee[1-\dd_{\bar \alpha}(z_i)-z_i].
\end{equation*} 
Thus, by formula~\eqref{difference_1}, the spectral sequence computing the image of the complex~\eqref{eq:gsc-Koszul} degenerates into the exact sequence~\eqref{eq:gsc}.

If a group~$G$ acts on~$X$ then the Koszul complex and the Fourier--Mukai functors are~$G$-equivariant,
hence the resulting exact sequence is $G$-equivariant as well.
\end{proof}
\begin{remark} \label{remark:why.generalized}
If~$\bV(\bar \alpha) \subset [\alpha_1-r+1,\alpha_1-1]$, i.e., $\alpha_1 \ge \alpha_2 \ge \dots \ge \alpha_k \ge \alpha_1 - r + k$, the
exact sequence~\eqref{eq:gsc} coincides with the usual staircase complex and the definition~\eqref{eq:def-ab-i} of $\alpha^{(i)}$ coincides with~\eqref{remaind}.  For that reason, we call~\eqref{eq:gsc} the {\sf generalized staircase complex}.   
\end{remark}

\subsection{Some properties of weights~\texorpdfstring{$\alpha^{(i)}$}{alpha (i)}} \label{subsection:alpha^i}
First, the weights $\alpha^{(i)}$ can be defined for all $i\in \mathbb{Z}$. For a $\GL_1\times \GL_{k-1}$-dominant weight~$\alpha = (\alpha_1;\alpha_2,\alpha_3,\dots,\alpha_k) \in \Y_1\times \Y_{k-1}$ we decreasingly enumerate the consecutive integers in the set $\ZZ \setminus \bV(\bar \alpha)$:
\begin{equation} \label{def:zi}
\{\ldots <z_2<z_1<z_0<z_{-1}<z_{-2}<\ldots \} \coloneqq \ZZ \setminus \bV(\bar \alpha),
\end{equation}
in such a way that $z_0$ is the closest to $\alpha_1$ integer in $\mathbb{Z}\setminus \bV(\bar \alpha)$ with the property $z_0\le \alpha_1$. Then, as before, we define
\begin{equation} \label{def:foralli}
\alpha^{(i)} \coloneqq \BBW_{\bar \alpha}(z_i),
\qquad 
b^i \coloneqq \alpha_1 - z_i, \qquad \dd(\alpha,i) \coloneqq \dd_{\bar \alpha}(z_i) \qquad \text{for} \quad i \in \ZZ. 
\end{equation}
In particular, $\{b^i\}$ is a strictly increasing sequence of integers. Note that 
\begin{equation} \label{eq:def:b}
    b^i= |\alpha| -  |\alpha^{(i)}|.
\end{equation}
\begin{remark}\label{rem:set-bar-alpha}
In other words, up to a shift of enumeration, the set of weights $\{\alpha^{(i)}\}_{i\in \mathbb{Z}}$ depends only on~$\bar \alpha$ and coincides with the set of weights of bundles $\pi_{1*}(S^z\cU_1^{\vee}\otimes \Sigma^{\bar{\alpha}}\cQ^{\vee}_{k-1})_{z\in \mathbb{Z}}$ on~$\Gr(k,\cV)$, where $z\in \mathbb{Z}$ and~$\pi_1\colon\PP_{\Gr(k,\cV)}(\cU_k)\to \Gr(k,\cV)$. The enumeration of the set $\{\alpha^{(i)}\}_{i\in \mathbb{Z}}$ depends on $\alpha_1$.    
\end{remark}
\begin{example}
If~\(\alpha = (3,2,2,1)\) then~\(\bV(\bar \alpha) = \{-2,0, 1\}\), and we have
\begin{equation*}
    z_i = 
    \begin{cases}
        3-i, & \text{if} \ i\le 1;\\
        -1, & \text{if} \ i=2;\\ 
        -i, & \text{if} \ i\ge 3.
    \end{cases}
\end{equation*}
The functions~\(\dd_{\bar \alpha} \colon \mathbb{Z} \setminus \bV(\bar \alpha) \to \ZZ\) and~\(\dd(\alpha,-)\colon \mathbb{Z} \to \ZZ\) have graphs plotted on Figure~\ref{fig:d}.
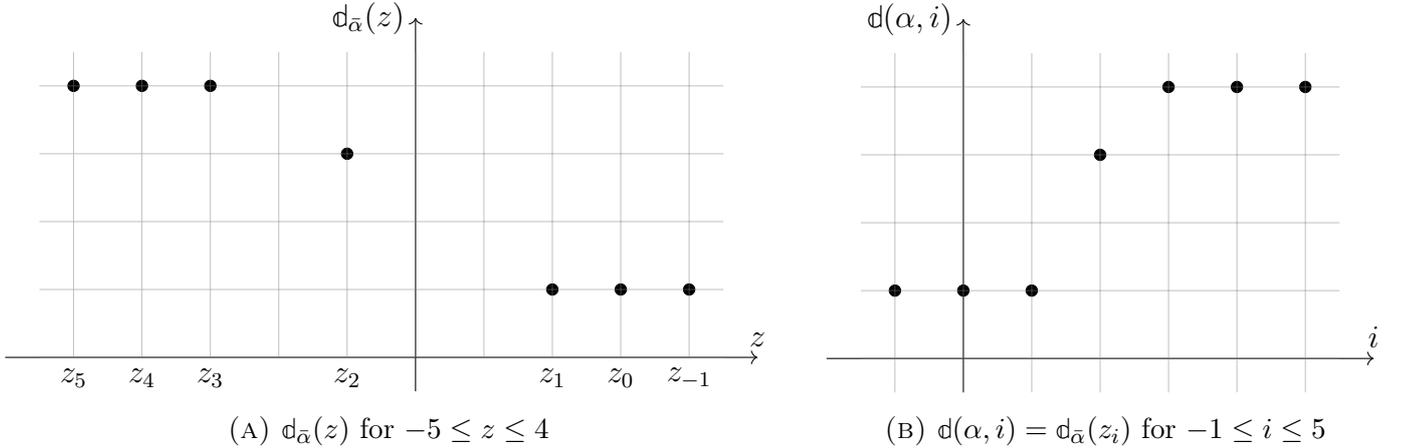
\begin{figure}[h]
\caption{Graphs of~\(\dd\) for~\(\alpha=(3,2,2,1)\)}
\label{fig:d}
\begin{subfigure}[b]{0.55\textwidth}
    \centering
    \begin{tikzpicture}[scale=0.9]
        \draw[->] (-6,0) -- (5,0) node[above] {\(z\)}; 
        \draw[->] (0,-0.5) -- (0,5) node[left] {\(\dd_{\bar \alpha}(z)\)}; 
		
        \foreach \x in {-5,-4,-3} 
        {
            \node[draw, circle, fill, inner sep=1.5pt] at (\x,4) {};
        }
    
        \node[draw, circle, fill, inner sep=1.5pt] at (-1,3) {};
    		
        \foreach \x in {2,3,4} 
        {
	   \node[draw, circle, fill, inner sep=1.5pt] at (\x,1) {};
        }

        \node at (4,-0.3) {\(z_{-1}\)};
        \node at (3,-0.3) {\(z_{0}\)};
        \node at (2,-0.3) {\(z_{1}\)};
        \node at (-1,-0.3) {\(z_{2}\)};
        \node at (-3,-0.3) {\(z_{3}\)};
        \node at (-4,-0.3) {\(z_{4}\)};
        \node at (-5,-0.3) {\(z_{5}\)};
        \draw[opacity=0.5,help lines] (-5.5,0) grid (4.5,4.5);
    \end{tikzpicture}
    \caption{\(\dd_{\bar \alpha}(z)\) for~\(-5 \le z \le 4\)}
    \label{fig:d(z)}
\end{subfigure}
\hspace{0.02\textwidth}
\begin{subfigure}[b]{0.4\textwidth}
          \centering
    \begin{tikzpicture}[scale=0.9]
        \draw[->] (-2,0) -- (6,0) node[above] {\(i\)}; 
        \draw[->] (0,-0.5) -- (0,5) node[left] {\(\dd(\alpha,i)\)}; 
		
        \foreach \x in {-1,0,1} 
        {
            \node[draw, circle, fill, inner sep=1.5pt] at (\x,1) {};
        }
    
        \node[draw, circle, fill, inner sep=1.5pt] at (2,3) {};
    		
        \foreach \x in {3,4,5} 
        {
	   \node[draw, circle, fill, inner sep=1.5pt] at (\x,4) {};
        }
        \node at (0,-0.3) {\(\phantom{z_0}\)};
        \draw[opacity=0.5,help lines] (-1.5,-0.5) grid (5.5,4.5);
    \end{tikzpicture}
    \caption{\(\dd(\alpha,i)=\dd_{\bar \alpha}(z_i)\) for~\(-1 \le i\le 5\)}
    \label{fig:d(i)}
\end{subfigure}
\end{figure}
\end{example}
Note that, analogously to~\eqref{difference_1}, for all~$i\in \mathbb{Z}$ we have
\begin{equation} \label{difference_1,in_Z}
\dd(\alpha,i+1)-\dd(\alpha,i)=\dd(z_{i+1})-\dd(z_{i})= z_{i}-z_{i+1}-1.
\end{equation}

We now present explicit formulas for~\(z_i\), and consequently~\(\alpha^{(i)}\) and $b^i(\alpha)$, for all $i\in \mathbb{Z}$. These formulas will be used in the construction of secondary staircase complexes in Section~\ref{section:secondary.staircase}.

First, the identity~\eqref{difference_1,in_Z} implies
\begin{equation}\label{eq:z_i=z_0-i}
    z_i = z_0+\dd(\alpha,0)-\dd(\alpha,i)-i.
\end{equation}
Thus, it suffices to find an expression for~\(z_0\) in terms of~\(\alpha\) and~\(\dd(\alpha,0)\). 

To do this, we associate to each $\alpha$ an integer defined by
\begin{equation} \label{eq:def-vv1}
    \vv(\alpha)\coloneq \dd(\alpha,0)-1 +z_0-\alpha_1.
\end{equation}
\begin{example}
If~\(\alpha \in \Y_k\), then~\(z_0 = \alpha_1\) and~\(\dd(\alpha,0)=1\). Therefore,
\begin{equation}\label{eq:v=0}
    \vv(\alpha)= 0 \qquad \text{for}\quad \alpha\in\Y_k.
\end{equation}  
\end{example}
Next, substituting~\eqref{eq:def-vv1} 
into~\eqref{eq:z_i=z_0-i} yields the expression:
\begin{equation*}
    z_i=\alpha_1+\vv(\alpha)-\dd(\alpha,i)+1-i \qquad\text{for} \quad i \in \ZZ.
\end{equation*}
Consequently, we have
\begin{equation*}
\alpha^{(i)}=\BBW_{\bar \alpha}(z_i)=\BBW_{\bar \alpha}(\alpha_1+\vv(\alpha)-\dd(\alpha,i)+1-i),
\end{equation*}
so that applying the definition~\eqref{eq:def-bbw} yields the explicit formula:
\begin{equation} \label{explicit_formula_i}
 \alpha^{(i)} =(\alpha_2-1,\ldots, \alpha_{\dd(\alpha, i)}-1,\alpha_1-i+\vv(\alpha),\alpha_{\dd(\alpha, i)+1},\ldots,\alpha_k). 
\end{equation}
In particular, we get an alternative expression for~\eqref{eq:def-dd}:
\begin{equation}\label{eq:d-ineq}
    \dd(\alpha,i) = \max \{1, j \mid \alpha_1-i+\vv(\alpha)< \alpha_j\}.
\end{equation}
Next, by~\eqref{eq:v=0} the formula~\eqref{explicit_formula_i} coincides with~\eqref{remaind}  for~$\alpha\in \Y_k$.

Moreover, combining~\eqref{eq:def:b} with~\eqref{explicit_formula_i}, we obtain
\begin{equation} \label{eq:b=i-v+d}
  b^i(\alpha) = i - \vv(\alpha) + \dd(\alpha, i) - 1.  
\end{equation}

To derive an alternative formula for~\(\vv(\alpha)\), we observe that
\begin{equation*}
\alpha_1-z_0 = |\bV(\bar\alpha)\cap (z_0,\alpha_1]|,
\end{equation*} 
since~\(z_0\) is the largest integer not in~\(\bV(\bar\alpha)\) satisfying~\(z_0 \le \alpha_1\).  
Comparing this with~\eqref{eq:def-dd}, we obtain a more explicit, combinatorial expression for~\(\vv\):
\begin{equation}\label{eq:def-vv}
\vv \colon \ZZ\times \Y_{k-1} \to [0,k-1], \qquad
    \vv(\alpha)= |(\alpha_1,\infty) \cap \bV(\bar \alpha)|. 
\end{equation}

A natural question to ask is how a given Young diagram $\beta\in \YD_t$ can be presented in the form~$\beta=\lambda^{(i)}$ for some $\lambda\in \Y_1\times \YD_{t-1}$ and some $i\in \mathbb{Z}$. 
\begin{lemma}\label{lemma:alpha^i->alpha}
Given a pair $(\beta,d) \in \YD_t\times [1,t]$, and an integer~$b\in \mathbb{Z}$, there exists a unique weight
\begin{equation*} 
    \mathrm{ins}_{b,d}(\beta) \coloneqq (\beta_d +b-d+1; \beta_1 + 1, \dots, \beta_{d - 1} + 1, \beta_{d + 1}, \dots, \beta_t) \in \Y_1 \times \YD_{t - 1}, 
\end{equation*}
constructed from $\beta$, $b$ and $d$, and a unique integer $i \in \mathbb{Z}$ such that 
\begin{equation*}
\mathrm{ins}_{b,d}(\beta)^{(i)}=\beta, \quad b^i(\mathrm{ins}_{b,d}(\beta))=b, \quad \dd(\mathrm{ins}_{b,d}(\beta), i)=d.
\end{equation*}
\end{lemma}
\begin{proof}
Suppose that a weight $\lambda\in \Y_1\times \YD_{t-1}$ satisfies
\begin{equation*}
\lambda^{(i)}=\beta, \quad \dd(\lambda, i)=d, \quad b^i(\lambda)=b
\end{equation*}
for some $i\in \mathbb{Z}$. 
Then, by the formula~\eqref{explicit_formula_i}, we have
\begin{equation*}
\lambda=(\beta_d +i-\vv(\lambda); \beta_1 + 1, \dots, \beta_{d - 1} + 1, \beta_{d + 1}, \dots, \beta_t),
\end{equation*}
and by the formula~\eqref{eq:b=i-v+d}, we follows that
\begin{equation*}
i-\vv(\lambda)=b-d+1.
\end{equation*}
Substituting this into the expression for $\lambda$, we conclude that $\lambda=\mathrm{ins}_{b,d}(\beta)$ is the unique weight satisfying the required conditions. Finally, the index $i$ is also uniquely determined, since the weights $\{\mathrm{ins}_{b,d}(\beta)^{(j)}\}_{j\in \mathbb{Z}}$ are distinct.
\end{proof}
We immediately deduce the following corollary, which will be used in Section~\ref{section:odd}.
\begin{corollary} \label{corollary:lambda_beta}
For a given pair $(\beta,d)\in \YD^w_t\times [1,t]$ the weight
\begin{equation*}
  \varrho=(\beta_d+1; \beta_1 + 1, \dots, \beta_{d - 1} + 1, \beta_{d + 1}, \dots, \beta_t) \in \YD^{w+1}_1\times \YD^{w+1}_{t-1}
\end{equation*} 
is the unique element of $\YD_1\times \YD_{t-1}$ satisfying the properties $\varrho^{(i)}=\beta$ and $b^{i}(\varrho)=\dd(\varrho, i)=d$ for some index $i$.
\end{corollary}
\begin{proof}
    By Lemma~\ref{lemma:alpha^i->alpha}, we obtain that
    \begin{equation*}
\varrho=(\mathrm{ins}_{d,d}(\beta))
\end{equation*}
    and that $\varrho$ is unique. 
\end{proof}

\subsection{Symplectic staircase complexes} \label{subsection:symplectic}
Now we introduce an analogous combinatorial notation, 
this time motivated by a symplectic Borel--Bott--Weil computation for a weight
\begin{equation}
\label{eq:spweight-t-bal}
(z;\alpha_2,\dots,\alpha_l) \in \Y_1 \times \YD_{l-1},
\end{equation}
see the proof of Theorem~\ref{theo:ssc} below for details. 

Let~$\bar \alpha = (\alpha_2,\alpha_3,\dots,\alpha_l) \in \YD_{l-1}$ be a dominant weight of~$\Sp_{2(l-1)}$.
We associate with~$\bar \alpha$ the following set of {\sf symplectic vanishing values} 
(the values of~$z$ for which the bundle $S^z\cU_1^{\vee}\otimes \Sigma^{\bar \alpha}_{\Sp}\cS_{2(l-1)}$ with the weight~\eqref{eq:spweight-t-bal} has no cohomology on~$\IGr(1,2l)=\PP^{2l-1}$):
\begin{equation}
\label{eq:def-bvs}
\bVs(\bar \alpha) \coloneqq \{1 - 2l - \alpha_2, \dots, -1 - l - \alpha_l,\ -l,\ \alpha_l - l + 1, \dots, \alpha_2 - 1\} \subset \ZZ
\end{equation}
Thus, 
\begin{equation}\label{eq:VSp=VuV}
    \bVs(\bar \alpha) = (-\bV(\bar \alpha)-2l) \sqcup \{-l\} \sqcup \bV(\bar \alpha);
\end{equation}
in particular, $\bVs(\bar \alpha)$ is symmetric with respect to~$-l$.

\begin{lemma} \label{lemma:p_*:IFL->IGr}
Let~\(\cS_{2l}\) be a~\(2l\)-dimensional symplectic vector bundle on a scheme~\(X\), and let~\(p \colon \PP_X (\cS_{2l}) \to X\) be the projection. For~\(\mu=(\mu_1,\ldots,\mu_l) \in \YD_{l}\) and~\(0 \le j \le 2(\mu_1+l)\) we have
\begin{equation*}
    p_* \left(S^{\mu_1 - j}\cU_1^\vee \otimes \Sigma_\Sp^{\bar \mu} \cS_{2(l-1)}\right) =
    \begin{cases} 
        \Sigma_{\Sp}^{\mu^{(i)}}\cS_{2l}[1-\dd(\mu, i)], &\text{if~\(\mu_1-j+l>0\) and~\(j=b^i(\mu)\)}; \\
        \Sigma_{\Sp}^{\mu^{(i)}}\cS_{2l}[2l+\dd(\mu, i)-1], &\text{if~\(\mu_1-j+l<0\) and~\(j=2(\mu_1+l)-b^i(\mu)\)}; \\
        0, &\text{otherwise},
    \end{cases}
\end{equation*}
    where~\(i \in [0,\mu_1]\). 
\end{lemma}
\begin{proof}
By Theorem~\ref{theo:relative_BBW_sym} (symplectic Borel--Bott--Weil), we should find the $\Sp(2l,\Bbbk)$-dominant weight in the Weyl group orbit of
\begin{equation*}
\eta = (\mu_1 - j, \bar \mu) + \rho = (\mu_1 - j + l, \mu_2+l-1, \dots, \mu_l+1).
\end{equation*}
First, if~\(\mu_1 - j \in \bVs(\bar \mu)\), then the pushforward $p_* \left(S^{\mu_1 - j}\cU_1^\vee \otimes \Sigma_\Sp^{\bar \mu} \cS_{2(l-1)})\right)$ is trivial. 

Assume that~\(\mu_1 - j \notin \bVs(\bar \mu)\) and~$\eta_1>0$. Then the~$\GL_l$-dominant weight in the permutation group orbit of $\eta$ is also~$\Sp(2l,\Bbbk)$-dominant,
hence the vector bundle produced by Theorem~\ref{theo:relative_BBW_sym} 
is~\(\Sigma_{\Sp}^{\mu^{(i)}}\cS_{2l}[1-\dd(\mu, i)]\), where~\(\mu^{(i)}=\BBW_{\bar \mu}(z_i)\) for~\(z_i=\mu_1-j=\eta_1-l\) in the notation~\eqref{def:foralli} and~\(j=\mu_1-z_i=b^i(\mu)\).

Assume that~\(\mu_1 - j \notin \bVs(\bar \mu)\) and~$\eta_1 < 0$. In this case, the $\Sp(2l,\Bbbk)$-dominant weight in the Weyl group orbit of $\eta$ is obtained by changing the sign of $\eta_1$ and by permutations. 
Thus, by the computation of the previous case for~\((-\eta_1,\bar\eta)\), the vector bundle produced is $\Sigma_{\Sp}^{\mu^{(i)}}\cS_{2l}$, where~$\mu^{(i)} = \BBW_{\bar \mu}(z_i)$ for~$z_i=-\eta_1-l=-\mu_1+j-2l$ in the notation~\eqref{def:foralli}. In particular,~$j=z_i+\mu_1+2l=2(\mu_1+l)-b^i(\mu)$. The element of the symplectic Weyl group used 
is equal to the composition of the permutation~\((2, 3, \dots, l, 1)\), the change of the sign of the last coordinate and the permutation~\((1, \dots, \dd(i)-1,l, \dd(i),\dots, l-1)\). So, the shift is~\((l-1)+1+(l-1-\dd(i))\), as claimed.
\end{proof}
\begin{theorem}[Symplectic staircase complexes]\label{theo:ssc}
Let~\(0 < a < n/2\) and let $\mu \in \YD_{n-2a}$ be such that~\(\mu_1 = a\). There is the following self-dual up to the twist by~\(\cO(-H_{2a})\) complex on~\(\IGr(2a, V)\):
\begin{multline}\label{eq:SSC}
0 \to \Sigma_{\Sp}^{\mu} \cS_{2(n-2a)} (-H_{2a}) \to \Lambda^{b^1} \cQ_{2(n-a)} \otimes \Sigma_{\Sp}^{\mu^{(1)}} \cS_{2(n-2a)} (-H_{2a}) \to \dots\\
\to \Lambda^{b^a} \cQ_{2(n-a)} \otimes \Sigma_{\Sp}^{\mu^{(a)}} \cS_{2(n-2a)} (-H_{n-k}) 
\to \Lambda^{b^a} \cQ^{\vee}_{2(n-a)} \otimes \Sigma_{\Sp}^{\mu^{(a)}} \cS_{2(n-2a)} \to \dots \\
\to \Lambda^{b^1} \cQ_{2(n-a)}^{\vee} \otimes \Sigma_{\Sp}^{\mu^{(1)}} \cS_{2(n-2a)} 
\to \Sigma_{\Sp}^{\mu} \cS_{2(n-2a)} \to 0.
\end{multline}
\end{theorem}
\begin{proof}
Consider the following diagram:
\begin{equation*}
    \PP_{\Gr(2a,V)}(\cQ_{2n-2a})=\Fl(2a,2a+1;V) \leftarrow \Fl(2a,2a+1,n; V) \overset{\iota}{\hookleftarrow} \IFl(2a,2a+1,n; V) \overset{q}{\to} \IFl(2a,2a+1;V).
\end{equation*}
We take the pullback along~$\Fl(2a,2a+1,n; V)\to \Fl(2a,2a+1;V)$ of the twist of the Koszul complex~\eqref{Koszul} on $\PP_{\Gr(2a,V)}(\cQ_{2n-2a})$ and tensor it with~$\Sigma^{\bar{\mu}} (\cU_n/\cU_{2a+1})^\vee$:
\begin{multline*}
    0\to \Lambda^{2(n-a)}\cQ^\vee_{2n-2a}\otimes S^{a-2(n-a)}(\cU_{2a+1}/\cU_{2a})^\vee \otimes \Sigma^{\bar \mu} (\cU_n/\cU_{2a+1})^\vee\to  \ldots\\
    \ldots \to \cQ^\vee_{2n-2a}\otimes S^{a-1}(\cU_{2a+1}/\cU_{2a})^\vee \otimes\Sigma^{\bar \mu} (\cU_n/\cU_{2a+1})^\vee\to S^{a}(\cU_{2a+1}/\cU_{2a})^\vee\otimes \Sigma^{\bar \mu} (\cU_n/\cU_{2a+1})^\vee \to 0.
\end{multline*}
Since~\(\bar \mu \in \YD_{k-1}\), Theorem~\ref{theo:relative_BBW_sym} (symplectic Borel--Bott--Weil) and the projection formula yield 
\begin{align*}
     q_* \circ \iota^* &\left( \Lambda^j \cQ^\vee_{2n-2a} \otimes S^{a-j}(\cU_{2a+1}/\cU_{2a})^\vee \otimes \Sigma^{\bar \mu} (\cU_n/\cU_{2a+1})^\vee\right) = \\
    & \phantom{\Big(} \Lambda^j \cQ^\vee_{2n-2a} \otimes S^{a-j}(\cU_{2a+1}/\cU_{2a})^\vee \otimes \Sigma_\Sp^{\bar \mu} \cS_{2(n-2a-1)}.
\end{align*}
Thus, we obtain the complex on $\IFl(2a,2a+1; V)$, whose terms are
\begin{equation} \label{terms_of_cmpx}
\Lambda^j \cQ^\vee_{2n-2a} \otimes S^{a-j}(\cU_{2a+1}/\cU_{2a})^\vee \otimes \Sigma_\Sp^{\bar \mu} \cS_{2(n-2a-1)} \qquad \text{for}\quad j=0,\ldots, 2(n-a).
\end{equation}
Next, applying Lemma~\ref{lemma:p_*:IFL->IGr} to the map
\begin{equation*}
    \IFl(2a,2a+1; V) = \PP_{\IGr(2a, V)}(\cS_{2(n-2a)}) \to \IGr(2a, V)
\end{equation*}
we obtain from the complex with terms~\eqref{terms_of_cmpx} the following exact complex:
\begin{multline*}
0 \to \Lambda^{2(n-a)} \cQ_{2(n-a)} \otimes \Sigma_{\Sp}^{\mu} \cS_{2(n-2a)} \to \Lambda^{2(n-a)-b^1} \cQ_{2(n-a)} \otimes \Sigma_{\Sp}^{\mu^{(1)}} \cS_{2(n-2a)} \to \dots\\
\to \Lambda^{2(n-a)-b^a} \cQ_{2(n-a)} \otimes \Sigma_{\Sp}^{\mu^{(a)}} \cS_{2(n-2a)}
\to \Lambda^{b^a} \cQ^{\vee}_{2(n-a)} \otimes \Sigma_{\Sp}^{\mu^{(a)}} \cS_{2(n-2a)} \to \dots \\
\to \Lambda^{b^1} \cQ_{2(n-a)}^{\vee} \otimes \Sigma_{\Sp}^{\mu^{(1)}} \cS_{2(n-2a)} 
\to \Sigma_{\Sp}^{\mu} \cS_{2(n-2a)} \to 0.
\end{multline*}
Using the isomorphism~\(\Lambda^{2(n-a)-j} \cQ_{2(n-a)} \simeq \Lambda^{j} \cQ_{2(n-a)} (-H_{2a})\) we get~\eqref{eq:SSC}.
\end{proof}
\begin{remark}
In fact, one can define a complex analogous to~\eqref{eq:SSC} on~\(\IGr(l,2n)\) for all~\(1 \le l \le n-1\) and arbitrary~\(\mu \in \YD_{n-l}\).
However, in our setting, such complexes involve objects that do not lie in twists of the subcategory~$\cA$ defined in~\eqref{eq:def:A}, and hence are not relevant for our purposes.
\end{remark} 

\section{Properties of Kuznetsov--Polishchuk exceptional objects} \label{section:KP}
The goal of this section is to show that the exceptional objects that form the Kuznetsov--Polishchuk collection are vector bundles. We also provide explicit formulas for them that are used in Section~\ref{section:relation.sympl}.
\subsection{Kuznetsov--Polishchuk collection}\label{subsection:E}
In this subsection, we recall the definition of the exceptional objects that form the Kuznetsov--Polishchuk exceptional collection on $\IGr(k,V)$. The main steps of the construction were already briefly described in Introduction, see Subsection~\ref{intro:KP}. 

First, recall the description of an infinite full exceptional collection in the equivariant derived category~$\mathrm{D}^{\bG}(\IGr(k,V))$. Define the partial order on $\Y_k\times \YD_{n-k}$ as follows:
\begin{multline} \label{order}
 (\lambda_1,\ldots, \lambda_k;\mu_1,\ldots,\mu_{n-k}) \preceq (\beta_1,\ldots, \beta_k;\nu_1,\ldots,\nu_{n-k}),
 \\ \text{if} \qquad  \lambda_i\le \beta_i  \quad  \text{for all} \quad i\in [1,k] \qquad \text{and} \qquad |\lambda|+|\mu| \le |\beta| + |\nu|.
\end{multline}
Note that the restriction of the order $\preceq$ to $\Y_k=\Y_k\times 0 \subset \Y_k\times \YD_{n-k}$ gives the standard order~\eqref{eq:order-GL} on~$\Y_k$. 

\begin{theorem}[{\cite[Theorem 3.4]{KuzPol}}]\label{theo:eq-ex-col}
The vector bundles 
\begin{equation*}
(\Sigma^{\lambda}\cU^{\vee}_k\otimes \Sigma^{\mu}_{\Sp}\cS_{2(n-k)})_{(\lambda;\mu)\in \Y_k\times \YD_{n-k}},
\end{equation*}
ordered with respect to any total order refining the order opposite to~\eqref{order}, constitute a full exceptional collection in the derived category of equivariant sheaves $\mathrm{D}^{\bG}(\IGr(k,V))$.    
\end{theorem}

To shorten the notation in the following definitions, we denote simply by~$\cU^{\gamma}$ the irreducible $\bG$-equivariant vector bundle on~$\IGr(k,2n)$ corresponding to a dominant weight~$\gamma\in \Y_k\times \YD_{n-k}$ of the Levi quotient $\bL$.
\begin{definition}\cite[Definition 3.1]{KuzPol} \label{def:exc_blocks}
A subset $\B\subset \Y_k\times \YD_{n-k}$ of $\bL$-dominant weights is called an \textbf{exceptional block} if for all $\alpha, \beta \in \B$ the canonical map
\begin{equation}\label{eq:exc_block}
   \bigoplus_{\gamma \in \B} \mathrm{Ext}^{\bullet}_{\bG}(\cU^{\alpha}, \cU^{\gamma})\otimes \mathrm{Hom}(\cU^{\gamma},\cU^{\beta})
   \to \mathrm{Ext}^{\bullet}(\cU^{\alpha}, \cU^{\beta})
\end{equation}
is an isomorphism, where $\mathrm{Ext}^{\bullet}_{\bG}$ denotes the $\mathrm{Ext}$ groups in the derived category $\mathrm{D}^{\bG}(\IGr(k,V))$.
\end{definition}

For an exceptional block $\B$, consider the subcategories  
\begin{align*}
    &\cA^{\bG}_{\B}\coloneqq \langle \cU^{\alpha}\rangle_{\alpha\in \B}\subset \mathrm{D}^{\bG}(\IGr(k,V)),\\
    &\cA^{\phantom{\bG}}_{\B}\coloneqq \langle \cU^{\alpha}\rangle_{\alpha\in \B}\subset \mathrm{D}^{\mathrlap{b}\phantom{\bG}}(\IGr(k,V)),
\end{align*}
generated by $\cU^{\alpha}$ with $\alpha\in \B$. 
Let $\mathbf{Fg}\colon \cA^{\bG}_{\B}\to \cA_{\B}$ be the forgetful functor.

Note that by Theorem~\ref{theo:eq-ex-col} the collection $\langle \cU^{\alpha}\rangle_{\alpha\in \B}$ is a full exceptional collection in the category~$\cA^{\bG}_{\B}$, so we can consider its right dual exceptional collection $\langle\lvee\  \cU^{\alpha}\rangle_{\alpha\in \B}$, see Definition~\ref{def:right-dual}. We denote by
\begin{equation} \label{def:E_equiv} 
(\cE^{\gamma}_{\B})^{\bG}\coloneqq \RR_{\langle \cU^{\beta} \rangle_{\{\beta\in \B|\beta \prec \gamma\}}}(\cU^{\gamma})\in \cA^{\bG}_{\B}
\end{equation}
the right dual exceptional object. We denote by
\begin{equation} \label{definition:E_abstr}
\cE^{\gamma}_{\B}\coloneqq \mathbf{Fg}\big((\cE^{\gamma}_{\B})^{\bG}\big)\in \cA_{\B}
\end{equation}
the object of $\Db(\IGr(k,V))$ corresponding to $(\cE^{\gamma}_{\B})^{\bG}$.

We say that a subset $\B' \subset \B$ is \textbf{downward closed with respect to the order~$\preceq$} if for any~$\gamma, \gamma'\in \B$ the conditions~$\gamma\in \B'$ and~$\gamma'\preceq \gamma$ imply~$\gamma'\in\B'$.
\begin{proposition} [{\cite[Proposition 3.9, Lemma 3.11]{KuzPol}}] \label{KPgeneralform}
For an exceptional block $\B$, the objects~$\cE^{\gamma}_{\B}$ form a full exceptional
collection in $\cA_{\B}$ with respect to any total order, refining the order~\eqref{order}.

Any subset $\B' \subset \B$ of an exceptional block downward closed with respect to the order~\eqref{order} is again an exceptional block. Moreover, $\cE^{\gamma}_{\B'}=\cE^{\gamma}_{\B}$ for all $\gamma\in \B'$.
\end{proposition}
Now, let us present exceptional blocks for $\IGr(k,V)$. The first $k-1$ blocks coincide with the blocks in \cite{KuzPol}.  Recall the notation~$a_t=\lfloor (k-t)/2 \rfloor$.
\begin{proposition}  [{\cite[Subsections~8.1 and~8.2]{KuzPol}}] \label{KP-first-block}
The sets of weights $\tilde \B_t$ indexed by integers $t=0, \ldots, k-1$ and defined by
    \begin{equation*}
        \tilde \B_t \coloneqq \YD_t^{2n-k-t+1}\times \YD_{n-k}^{a_t}=\{2n-k-t+1\ge  \lambda_1\ge \lambda_2\ge \ldots\ge \lambda_t\ge 0;\  a_t  \ge \mu_{1}\ge \ldots\ge  \mu_{n-k}\ge 0\}
    \end{equation*}
form exceptional blocks for $\IGr(k,V)$, which we call \textbf{big  blocks}. In particular, the following subsets of~$\tilde \B_t$
 \begin{equation*} 
 \B_t\coloneqq \YD_t^{2n-k-t}\times \YD_{n-k}^{a_t} \qquad \text{for} \quad t=0, \ldots, k-1
\end{equation*}
are also exceptional blocks for $\IGr(k,V)$; we call them \textbf{small blocks}.
\end{proposition}
To construct the full exceptional collection we need to define one more exceptional block $\B_k$. It does not coincide with any of the exceptional blocks from \cite{KuzPol}, so we need to show that~$\B_k$ satisfies the condition~\eqref{eq:exc_block}. 

\begin{lemma} \label{prop:bigblockk}
The sets of weights $\B_k\subset \tilde \B_k$ defined by
    \begin{align*}
       & \tilde \B_k \coloneqq \YD_k^{2n-2k+1}=\{2n-2k+1\ge  \lambda_1\ge \lambda_2\ge \ldots \ge \lambda_k\ge 0\}; \\
       & \B_k \coloneqq \YD_k^{2n-2k\phantom{{}+1}}=\{2n-2k\phantom{{}+1}\ge \lambda_1\ge \lambda_2\ge \ldots\ge \lambda_k\ge 0\}
    \end{align*}
form exceptional blocks for $\IGr(k,V)$. Moreover, for~\(t = k-1\) and $t=k$ we have 
\begin{equation*}
\cE^{\lambda}_{\B_t} = \Sigma^{\lambda}\cU_k^\vee.
\end{equation*}
\end{lemma}
\begin{proof}
Note that $\B_{k-1},\B_k\subset \tilde \B_k$ are downward closed subsets with respect to the order~\eqref{order}. Thus, by Proposition~\ref{KPgeneralform}, it is enough to show that $\tilde \B_k$ is exceptional and that $\cE^{\lambda}_{\tilde \B_k} = \Sigma^{\lambda}\cU_k^\vee$.

To check the condition~\eqref{eq:exc_block}, it is enough to prove that the only non-zero equivariant morphisms in the subcategory~$\tilde \B_k$ are~\(\operatorname{Id}\in \mathrm{Ext}_{\bG}^{\bullet}(\Sigma^\alpha \cU_k^\vee,\Sigma^\alpha \cU_k^\vee)\).
First, Lemma~\ref{lemma:exceptional_for_big_t} yields
\begin{equation*}
\E{>0}_\bG(\Sigma^{\alpha}\cU^{\vee}_k,\Sigma^{\beta}\cU^{\vee}_k)=0 \qquad \text{for} \quad \alpha,\beta\in \tilde \B_k.
\end{equation*}
Now, suppose that for some~$\alpha,\beta\in \tilde \B_k\subset \YD_k$ we have
\begin{equation*}
\E{0}_\bG(\Sigma^{\alpha}\cU^{\vee}_k,\Sigma^{\beta}\cU^{\vee}_k)=\big(\operatorname{H}^{0}(\Sigma^\alpha \cU_k \otimes \Sigma^\beta \cU^\vee_k)\big)^{\bG}\ne 0.
\end{equation*} 
Let~\(\Sigma^\gamma \cU^\vee_k \inplus \Sigma^\alpha \cU_k \otimes \Sigma^\beta \cU^\vee_k\) be such that $\operatorname{H}^{0}(\Sigma^\gamma \cU^\vee_k) \neq 0$. 
Then~\(\gamma \in \YD_k\) by Theorem~\ref{theo:relative_BBW_sym} (symplectic Borel--Bott--Weil). Therefore,~\(\alpha \le \beta\) by Lemma~\ref{lemma:LR rule}, where $\le$ is the standard order~\eqref{eq:order-GL}. Next, if~\(\E{0}_\bG(\Sigma^{\alpha}\cU^{\vee}_k,\Sigma^{\beta}\cU^{\vee}_k) \neq 0\) then~\(\beta \le \alpha\) by Theorem~\ref{theo:eq-ex-col}. Thus,~\(\alpha=\beta\), and Theorem~\ref{theo:eq-ex-col} also yields $\E{0}_\bG(\Sigma^{\alpha}\cU^{\vee}_k,\Sigma^{\alpha}\cU^{\vee}_k)=\Bbbk$. So the block~$\tilde \B_t$ is exceptional.

In particular, we proved
\begin{equation*}
\E{\bullet}_\bG(\Sigma^{\alpha}\cU^{\vee}_k,\Sigma^{\beta}\cU^{\vee}_k)=0 \qquad \text{for} \quad \alpha,\beta\in \tilde \B_k \quad \text{such that} \quad \alpha>\beta.
\end{equation*}
Thus, the mutation~\eqref{def:E_equiv} is trivial and~$\cE^{\lambda}_{\tilde \B_t} = \Sigma^{\lambda}\cU_k^\vee$. 
\end{proof}


We use the following notations
\begin{align*} 
\tilde \cA_t & \coloneqq \cA_{\tilde\B_t}={}  \big\langle \Sigma^{\lambda}\cU^{\vee}_k\otimes \Sigma^{\mu}_{\Sp}\cS_{2(n-k)}\big\rangle_{(\lambda;\mu)\in \tilde \B_t}\subset \Db(\IGr(k,V)) \label{tildeA};\\
\cA_t &\coloneqq \cA_{\B_t}= \big\langle \Sigma^{\lambda}\cU^{\vee}_k\otimes \Sigma^{\mu}_{\Sp}\cS_{2(n-k)}\big\rangle_{(\lambda;\mu)\in \B_t}\subset \Db(\IGr(k,V));\\
\cE_t^{\lambda; \mu} & \coloneqq \cE_{\tilde \B_t}^{\lambda; \mu}\in \tilde \cA_t \qquad  \text{for} \quad (\lambda;\mu)\in \tilde \B_t,
\end{align*}
where $\cE_{\tilde \B_t}^{\lambda; \mu}$ is the exceptional object~\eqref{definition:E_abstr} for $\gamma=(\lambda;\mu)\in \tilde \B_t$. 
If~$\mu = 0$ then we use the notation~$\cE_t^{\lambda}\coloneqq \cE_t^{\lambda; 0}$. 
\begin{corollary}
For each $t=0, \ldots, k$ the subcategory $\tilde \cA_t$ is generated by a full exceptional collection
\begin{equation*}
\tilde \cA_t=\langle \cE_t^{\alpha; \mu} \rangle_{(\alpha; \mu)\in \tilde \B_t}
\end{equation*}
with respect to any complete order that refines~\eqref{order}.
\end{corollary}
\begin{proof}
    The statement immediately follows from Propositions~\ref{KPgeneralform} and~\ref{KP-first-block} and Lemma~\ref{prop:bigblockk}. 
\end{proof}
\begin{remark}
Note that by Proposition~\ref{KPgeneralform} for $(\lambda; \mu)\in \B_t\subset \tilde \B_t$ we have $\cE_{\B_t}^{\lambda; \mu}=\cE_{\tilde \B_t}^{\lambda; \mu}$, so the objects~$\cE_t^{\lambda; \mu}$ for~${(\lambda; \mu)\in \B_t}$ form a full exceptional collection in~$\cA_t$. 
\end{remark}
\begin{lemma} \label{E_description}
Fix an element $(\lambda;\mu)\in \tilde \B_t$. Suppose an object $\cE\in \Db(\IGr(k,V))$ has the following two properties: 
\begin{equation*}
\cE\in \tilde \cA_t, \qquad \text{and}\qquad \E{\bullet}_{\bG}\left(\cE,\, \Sigma^{\beta}\cU^{\vee}_k\otimes \Sigma^{\nu}_{\Sp}\cS_{2(n-k)} \right) = 
    \begin{cases}
        \Bbbk, & \text{if} \ (\beta;\nu) = (\lambda;\mu);\\
        0, & \text{if}\ 0 \preceq (\beta;\nu) \prec (\lambda;\mu).
    \end{cases}
\end{equation*} 
Then $\cE\simeq\cE_t^{(\lambda;\mu)}$.
\end{lemma}
\begin{proof}
    The statement immediately follows from the definition~\eqref{def:E_equiv} and the description~\eqref{dual_general}. 
\end{proof}   
\begin{remark}\label{rem:E(a_t)}
In general, the definition of $\cE_t^{\lambda; \mu}$ depends on the block number~$t$. However, if the sets 
    \begin{equation*}
\{\beta\times \nu\in \tilde \B_{t_1}\, |\, \beta\times \nu\prec \lambda\times \mu\} \qquad \text{and} \qquad \{\beta\times \nu \in \tilde \B_{t_2}\,|\,\beta \times \nu\prec \lambda\times \mu\}
\end{equation*}
    coincide, then~\(\cE^{\lambda;\mu}_{t_1}=\cE^{\lambda;\mu}_{t_2}\) by Proposition~\ref{KPgeneralform}. This happens whenever~\(a_{t_1} = a_{t_2}\). In particular, in the Lagrangian case~$\mathrm{LGr}(n,V)$ the exceptional object~$\cE_t^{\lambda}$ is fully characterized by the weight~$\lambda$ and
    does not depend on the block number~$t$.
\end{remark} 

\begin{remark}
    In the construction of~\cite{KuzPol} there are~\(2n-k+1\) blocks. The first~\(k\) blocks coincide with~\(\B_0, \dots, \B_{k-1}\). 
    
    To define the last~\(2n-2k+1\) blocks used in~\cite{KuzPol} we introduce the following notation
    \begin{equation*}
\YD^w_{k-1}(l)\coloneqq \{w+l \ge  \alpha_1\ge \alpha_2\ge \ldots \ge \alpha_k = l\}
\end{equation*}
    for the set of weights corresponding to the set of vector bundles~\(\{\Sigma^{\lambda} \cU_k^\vee(l)\mid \lambda \in \YD_{k-1}^{w} \}\). 
    The last~\(2n-2k+1\) blocks in~\cite{KuzPol} are indexed by~\(t =k,\ldots, 2n-k\) and defined by
    \begin{equation*}
       \B'_t\coloneqq \YD_{k-1}^{2n-k-t}(t) = \{2n-k \ge  \alpha_1\ge \alpha_2\ge \ldots \ge \alpha_k=t;\  \mu_{1}=\ldots=\mu_{n-k} = 0\};\\
    \end{equation*}

Note that the block $\B_k$ defined in Lemma~\ref{prop:bigblockk} can be described as the union
    \begin{equation*}
\B_k = \bigsqcup_{t=k}^{2n-k} \YD_{k-1}^{2n-k-t}(t-k).
\end{equation*}
It follows from the above that the last~\(2n-2k+1\) blocks in~\cite{KuzPol} can be replaced by one exceptional block $\B_k$.
\end{remark}

\begin{theorem}[{\cite[Theorem 9.2]{KuzPol}}]
The collection $\cA_0, \cA_1(1),\ldots,\cA_{k}(k)$ of subcategories is semiorthogonal in $\Db(\IGr(k,V))$, and each subcategory is generated by the exceptional collection:
\begin{equation*}
 \cA_t= \langle \cE_t^{\alpha; \mu} \rangle_{(\alpha; \mu)\in \B_t}.
\end{equation*}
\end{theorem}
\subsection{Description of \texorpdfstring{\((\cE^{\lambda; \mu}_t)^{\vee}\)}{cF}} \label{subsec:general_formulas_for_E}
The goal of this subsection is to provide an explicit description of objects~$(\cE^{\lambda; \mu}_t)^{\vee}$ and, as a consequence, to prove that $\cE^{\lambda; \mu}_t$ is a vector bundle for $(\lambda;\mu)\in \B_t$.

First, for $t=k-1$ and $t=k$ the explicit description of the exceptional objects $\cE^{\lambda}_t$ is provided by Lemma~\ref{prop:bigblockk}, where it was shown that $\cE^{\lambda}_t\simeq \Sigma^{\lambda}\cU^{\vee}_k$ for these $t$. 

Let us now describe the exceptional objects for $t=0$. Recall that $\cA_0=\langle \Sigma^{\mu}_\Sp \cS_{2(n-k)} \rangle_{\mu\in \YD^{a_0}_{n-k}}$.
\begin{lemma}\label{prop:0th_block}
    Let~\(\mu \in \YD^{a_0}_{n-k}=\B_0\). Then~\(\cE_0^{0;\mu} = \Sigma^{\mu}_\Sp \cS_{2(n-k)}\). 
\end{lemma}
\begin{proof}
By~\cite[Lemma~A.8]{F19} for any~\(\mu' \in \YD_{n-k}\) we have \begin{equation*}
\E{\bullet}_{\bG}\left(\Sigma_\Sp^{\mu} \cS_{2(n-k)},\, \Sigma_\Sp^{\mu'} \cS_{2(n-k)} \right) = 
    \begin{cases}
		\Bbbk, & \text{if} \ \mu = \mu';\\
		0, & \text{otherwise.}
    \end{cases}
\end{equation*} 
Thus, the mutation~\eqref{def:E_equiv} in $\cA^{\bG}_0=\langle \Sigma^{\mu'}_\Sp \cS_{2(n-k)} \rangle_{\mu'\in \B_0}$ is trivial and~\(\cE^{0;\mu}_0= \Sigma^{\mu}_\Sp \cS_{2(n-k)}\).
\end{proof}

Thus, for $t=0,k-1,k$ the exceptional objects $\cE^{\lambda; \mu}_t$ are isomorphic to irreducible $\bG$-equivariant vector bundles.

In cases $1 \le t \le k-2$, we cannot describe $\cE^{\lambda; \mu}_t$ so easily. To prove that $\cE^{\lambda; \mu}_t$ are vector bundles, we provide a description of dual objects $(\cE^{\lambda; \mu}_t)^{\vee}$. This  description expresses $(\cE^{\lambda; \mu}_t)^{\vee}$ as the pushforwards of some $\bG$-equivariant irreducible vector bundles on partial flag varieties. It also implies that $(\cE^{\lambda; \mu}_t)^{\vee}$ and, consequently, $\cE^{\lambda; \mu}_t$ are vector bundles.

Let us introduce the notations that we will use
\begin{align*}
    \cF_t^{\lambda;\mu} & \coloneqq (\cE_t^{\lambda;\mu})^{\vee};\\
    \tilde \cA_t^{\vee}& \coloneqq \langle \Sigma^{\lambda}\cU_k\otimes \Sigma^{\mu}_{\Sp}\cS_{2(n-k)}\rangle_{(\lambda;\mu)\in \tilde \B_t}\subset \Db(\IGr(k,V)). 
\end{align*}

Note that the duality functor is an anti-autoequivalence, so the characterization from Lemma~\ref{E_description} for the objects~$\cF_t^{\lambda;\mu}$ can be rewritten:
\begin{equation} \label{F_description}
    \cF_t^{\lambda; \mu}\in \tilde \cA^{\vee}_t\qquad \text{and} \qquad
    \E{\bullet}_{\bG}\left(\Sigma^{\beta}\cU_k\otimes \Sigma^{\nu}_{\Sp}\cS_{2(n-k)},\, \cF_t^{\lambda;\mu} \right) = 
    \begin{cases}
	\Bbbk, & \text{if} \ (\beta;\nu) = (\lambda;\mu);\\
	0, & \text{if}\ 0 \preceq (\beta;\nu) \prec (\alpha;\mu).
    \end{cases}
\end{equation}

Let~$1 \le t \le k-1$. Consider the following diagram:
\begin{equation}\label{diagr}
    \begin{tikzcd} 
	&\IFl(k - t, k; V) \ar[dl, swap, "p_{t}"] \ar[dr, "q_{t}"]\\
	\IGr(k - t, V) && \IGr(k, V).
    \end{tikzcd}
\end{equation}
\begin{definition}
Let~\(1\le t\le k-1\) and~$\lambda \in \Y_{t}$. We define the functor $\hat \Phi_\lambda^t \colon \Db(\IGr(k-t,V))\to \Db(\IGr(k,V))$ as the Fourier--Mukai transform with the kernel~$\Sigma^{\lambda-a_t} (\cU_k/\cU_{k-t})$, which takes 
$E \in \Db(\IGr(k-t,V))$~to
\begin{equation}\label{eq:Phi}
        \hat \Phi^t_{\lambda}(E)\coloneqq q_{t*}(\Sigma^{\lambda - a_t} (\cU_k/\cU_{k-t})\otimes p_t^*(E)).
\end{equation}    
\end{definition}
Furthermore, we denote
\begin{equation}\label{eq:def:hat-mu}
    \hat \mu \coloneqq ((a_t)^t, \mu)\in \YD^{a_t}_{n-k+t} \qquad \text{for} \quad \mu\in \YD^{a_t}_{n-k}.
\end{equation}
Note that~\(\hat \mu\) depends on the block number~\(t\). However, it should be clear from the context, so we omit it. Recall that for $(\lambda; \mu)\in \tilde\B_t$ we have $\lambda\in \YD_t$, hence the weight $\lambda-a_t=(\lambda_1-a_t,\lambda_2-a_t, \ldots, \lambda_t-a_t)$ can be applied to the rank $t$ vector bundle~$\cU_k/\cU_{k-t}$ on the flag variety $\IFl(k-t,k;V)$.

\begin{theorem}\label{theo:formulas_for_F}
Let~\(1 \le t \le k-2\). For even $k-t$ let $(\lambda; \mu)\in \tilde \B_t$ and for odd $k-t$ let $(\lambda; \mu)\in \B_t$. There is an isomorphism:
        \begin{equation*}
\cF^{\lambda; \mu}_t \simeq \hat \Phi^t_{\lambda}(\Sigma_\Sp^{\hat \mu} \cS_{2(n-k+t)})= q_{t*}\left(\Sigma^{\lambda - a_t} (\cU_k/\cU_{k-t})\otimes p_t^*(\Sigma_\Sp^{\hat \mu} \cS_{2(n-k+t)})\right).
\end{equation*}
Moreover, $\cF^{\lambda; \mu}_t$ is a vector bundle.
\end{theorem}
To prove Theorem~\ref{theo:formulas_for_F}, we show that the object~$\hat \Phi^t_{\lambda}(\Sigma_\Sp^{\hat \mu} \cS_{2(n-k+t)})$ satisfies the conditions~\eqref{F_description}. We split the proof into several lemmas. 
\begin{lemma}\label{lemma:Phi(S)_in_A}
Let~\(1 \le t \le k-1\) and~$(\lambda; \mu)\in \tilde \B_t$. Then~\(\hat \Phi^t_{\lambda}(\Sigma_\Sp^{\hat \mu} \cS_{2(n-k + t)})\) is a vector bundle in~\(\tilde \cA^{\vee}_t\). 
\end{lemma}
\begin{proof}
Denote $w=2n-k-t+1$. Recall that 
\begin{equation*}
\tilde \cA^{\vee}_t= \big\langle \Sigma^{\lambda}\cU_k\otimes \Sigma^{\mu}_{\Sp}\cS_{2(n-k)}\big\rangle, \quad 
\text{where} \quad (\lambda;\mu)\in\YD_t^{w}\times \YD_{n-k}^{a_t}.
\end{equation*}

Lemma~\ref{lemma:filtration_on_flags} for~\(\xi=\lambda-a_t\) and~\(\theta=\hat \mu\) yields that the vector bundle~$\Sigma^{\lambda-a_t} (\cU_k/\cU_{k - t}) \otimes p^*_{t}  \Sigma_\Sp^{\hat \mu} \cS_{2(n-k + t)}$ on $\IGr(k-t,k;V)$ has the filtration with the following irreducible~$\bG$-equivariant subquotients:
\begin{equation} \label{filtration_for_interpr}
\Sigma^{\gamma} (\cU_k/\cU_{k - t}) \otimes \Sigma_\Sp^\eta \cS_{2(n - k)}, \quad \text{where} \quad \gamma \in \Y_t \ \text{such that}\ \gamma_i \in [t - k, w] \quad %
    \text{and} \quad \eta \in \YD^{a_t}_{n - k}.
\end{equation}

Let us now prove that for a weight $\gamma$ in the filtration~\eqref{filtration_for_interpr} we have
\begin{equation} \label{direct-image-filtr}
    q_{t*}\Sigma^{\gamma} (\cU_k/\cU_{k - t}) = 
    \begin{cases}
        \Sigma^{\gamma} \cU_k, & \text{for~\(\gamma \in \YD^w_t\)};\\
        0, &\text{otherwise}.
    \end{cases}
\end{equation}
To compute $q_{t*}\Sigma^{\gamma} (\cU_k/\cU_{k - t})$ we apply Theorem~\ref{theo:relative_BBW} (Borel--Bott--Weil) to $X=\IGr(k,V)$, $\cE=\cU_k$ and $\pi_t=q_t$. As an intermediate step, we get:
\begin{equation*}
((0)^{k-t}, -\gamma_t, \ldots, -\gamma_1)+(k,k-1, \ldots, 1)=(k,k-1,\ldots, t+1,t-\gamma_t, \ldots, 1-\gamma_1).
\end{equation*}
From the conditions in~\eqref{filtration_for_interpr} we get $k \ge t-\gamma_t \ge t-w$. Thus, there are two possible cases:
\begin{itemize}
    \item if $k \ge t-\gamma_t \ge t+1$, then $q_{t*} \Sigma^{\gamma} (\cU_k/\cU_{k - t})=0$;
    \item if $t \ge t-\gamma_t$, then $\gamma\in \YD_t^w$, and the weight $(k,k-1,\ldots, t+1,t-\gamma_t, \ldots, 1-\gamma_1)$ is already dominant, so Theorem~\ref{theo:relative_BBW} yields $q_{t*} \Sigma^{\gamma} (\cU_k/\cU_{k - t})= \Sigma^{\gamma} \cU_k\in \tilde \cA_t^{\vee}$. 
\end{itemize}

It follows from~\eqref{filtration_for_interpr}, projection formula and~\eqref{direct-image-filtr} that the object~$\hat \Phi^t_{\lambda}(\Sigma_\Sp^{\hat \mu} \cS_{2(n-k+t)})$ has the filtration with the irreducible subquotients
\begin{equation*} 
q_{t*}\Sigma^{\gamma}\cU_k\otimes \Sigma_\Sp^\eta \cS_{2(n - k)}, \quad \text{where} \quad \gamma \in \YD^w_t \quad \text{and} \quad \eta \in \YD^{a_t}_{n - k},
\end{equation*}
that lie in~\(0\)th cohomological degree, which yields the lemma. 
\end{proof}
Now, we would like to show that the vector bundle $\hat \Phi^t_{\lambda}( \Sigma_\Sp^{\hat \mu} \cS_{2(n-k+t)})$ satisfies the second condition of~\eqref{F_description}:
\begin{equation} \label{Ext_G}
    \E{\bullet}_{\bG}\left(\Sigma^{\alpha}\cU_k\otimes \Sigma^{\eta}_{\Sp}\cS_{2(n-k)},\hat \Phi^t_{\lambda}(\Sigma_\Sp^{\hat \mu} \cS_{2(n-k+t)} )\right) = 
    \begin{cases}
		\Bbbk, & \text{if} \ (\alpha;\eta) = (\lambda;\mu);\\
		0, & \text{if}\ 0 \preceq (\alpha;\eta) \prec (\lambda;\mu),
    \end{cases}
\end{equation}
where the partial order~\(\preceq\) is defined in~\eqref{order}.

\newcommand{\tcircled}[1]{\tikz[baseline=(char.base)]{
    \node[shape=circle,draw,inner sep=1pt] (char) {#1};}}
\begin{lemma}\label{lemma:computationExtg}
For any 
\((\lambda; \mu)\in \YD_t\times \YD_{n-k}\) and $(\alpha; \eta)\in \YD_t\times \YD_{n-k}$ we have
\begin{multline*}
   \E{\bullet}_{\bG}\left(\Sigma^{\alpha}\cU_k\otimes \Sigma^{\eta}_{\Sp}\cS_{2(n-k)},\hat \Phi^t_{\lambda}\big(\Sigma_\Sp^{\hat \mu} \cS_{2(n-k+t)}\big)\right)\simeq \\
   \simeq \E{\bullet}_{\bG}\left(\Sigma_\Sp^{\hat \mu} \cS_{2(n-k + t)},\, p_{t*}\left(\Sigma^\alpha \left(\cU_k / \cU_{k - t} \right)^{\vee} \otimes \Sigma^{ \lambda-a_t} (\cU_k/\cU_{k - t}) \otimes \Sigma_\Sp^\eta \cS_{2(n-k)} \right)\right).
\end{multline*}
\end{lemma}
\begin{proof}
We have
\begin{align*}
    &\E{\bullet}_{\bG}\left(\Sigma^{\alpha}\cU_k\otimes \Sigma^{\eta}_{\Sp}\cS_{2(n-k)},\, \hat \Phi^t_{\lambda} \big(\Sigma_\Sp^{\hat \mu} \cS_{2(n-k+t)}\big)\right) \overset{\tcircled{1}}{=}\\
    &=\E{\bullet}_{\bG}\left(\Sigma^\alpha \cU_k \otimes \Sigma_\Sp^{\eta} \cS_{2(n-k)},\, q_{t*}  \big(\Sigma^{ \lambda-a_t} (\cU_k/\cU_{k - t}) \otimes p^*_{t}  \Sigma_\Sp^{\hat \mu} \cS_{2(n-k + t)}\big)\right) \overset{\tcircled{2}}{\simeq}
    \\
    & \simeq \E{\bullet}_{\bG}\left(q^*_{t} (\Sigma^\alpha \cU_k \otimes \Sigma_\Sp^{\eta} \cS_{2(n-k)}),\, \Sigma^{ \lambda-a_t} (\cU_k/\cU_{k - t}) \otimes p^*_{t}  \Sigma_\Sp^{\hat \mu} \cS_{2(n-k + t)}\right) \overset{\tcircled{3}}{\simeq} \\
    & \simeq \operatorname{H}^{\bullet}\left(\IFl(k - t, k; V),\, \Sigma^\alpha \cU_k^{\vee} \otimes \Sigma_\Sp^{\eta} \cS_{2(n-k)} \otimes \Sigma^{\lambda-a_t} (\cU_k/\cU_{k - t}) \otimes p^*_{t} \Sigma_\Sp^{\hat \mu} \cS_{2(n-k + t)}\right)^{\bG} \overset{\tcircled{4}}{\simeq}\\
    &\simeq \operatorname{H}^{\bullet}\left(\IFl(k - t, k; V),\, \Sigma^\alpha (\cU_k/\cU_{k-t})^{\vee} \otimes \Sigma_\Sp^{\eta} \cS_{2(n- k)} \otimes \Sigma^{\lambda-a_t} (\cU_k/\cU_{k - t}) \otimes p^*_{t} \Sigma_\Sp^{\hat \mu} \cS_{2(n-k + t)}\right)^{\bG} \overset{\tcircled{5}}{\simeq}
    \\
    &\simeq \E{\bullet}_{\bG}\left(p^*_{t} \Sigma_\Sp^{\hat \mu} \cS_{2(n-k + t)},\, \Sigma^\alpha \left(\cU_k / \cU_{k - t} \right)^{\vee} \otimes \Sigma^{\lambda-a_t} (\cU_k/\cU_{k - t}) \otimes \Sigma_\Sp^\eta\cS_{2(n-k)} \right) \overset{\tcircled{6}}{\simeq}
    \\
    & \simeq \E{\bullet}_{\bG}\left(\Sigma_\Sp^{\hat \mu} \cS_{2(n-k + t)},\, p_{t*}\left(\Sigma^\alpha \left(\cU_k / \cU_{k - t} \right)^{\vee} \otimes \Sigma^{ \lambda-a_t} (\cU_k/\cU_{k - t}) \otimes \Sigma_\Sp^\eta \cS_{2(n-k)} \right)\right),
\end{align*}
where
\begin{enumerate}
    \item[\tcircled{1}] follows from the definition of~$\hat\Phi^t_{\lambda}$;
    \item[\tcircled{2}] follows from the adjunction;
    \item[\tcircled{3}] follows from the fact that~\(q^*_{t} (\Sigma^\alpha \cU_k \otimes \Sigma_\Sp^{\eta} \cS_{2(n-k)}) \simeq \Sigma^\alpha \cU_k \otimes \Sigma_\Sp^{\eta} \cS_{2(n-k)}\) is a vector bundle;
    \item[\tcircled{4}] follows by Corollary~\ref{symp_coh};
    \item[\tcircled{5}] follows from the fact that~\(p^*_{t} \Sigma_\Sp^{\hat \mu} \cS_{2(n-k + t)}\) is a self-dual vector bundle;
    \item[\tcircled{6}] follows from the adjunction.\qedhere
\end{enumerate} 
\end{proof}
Lemma~\ref{lemma:computationExtg} shows that to prove~\eqref{Ext_G} we need to investigate the pushforwards
\begin{equation*}
p_{t*}\left(\Sigma^\alpha \left(\cU_k / \cU_{k - t} \right)^{\vee} \otimes \Sigma^{ \lambda-a_t} (\cU_k/\cU_{k - t}) \otimes \Sigma_\Sp^\eta \cS_{2(n-k)} \right).
\end{equation*}
Note that $p_t\colon \IGr_{\IGr(k-t,V)}(t,\cS_{2(n-k + t)})\to \IGr(k-t,V)$ is the projection of the relative isotropic Grassmannian, so Theorem~\ref{theo:relative_BBW_sym} (symplectic Borel--Bott--Weil) yields:
\begin{equation} \label{obvious-decomposition}
    p_{t*}\left(\Sigma^\alpha \left(\cU_k / \cU_{k - t} \right)^{\vee} \otimes \Sigma^{\lambda-a_t} (\cU_k/\cU_{k - t}) \otimes \Sigma_\Sp^\eta \cS_{2(n-k)} \right) \simeq \bigoplus \Sigma_\Sp^{\mu'} \cS_{2(n-k + t)},
\end{equation}
for some Young diagrams $\mu'\in \YD_{n-k+t}$.
\begin{lemma} \label{incl} 
Let~\(1 \le t \le k-2\). For even $k-t$ let $(\lambda; \mu)\in \tilde \B_t$ and for odd $k-t$ let $(\lambda; \mu)\in \B_t$. For the partial order~\eqref{order} let~$0 \preceq (\alpha;\eta) \preceq (\lambda;\mu)$. Recall that $\hat \mu = ((a_t)^t, \mu)$. Then
\begin{equation} \label{S:inclu}
\Sigma_\Sp^{\hat \mu} \cS_{2(n - k + t)} \inplus p_{t*}\left(\Sigma^\alpha (\cU_k / \cU_{k - t})^{\vee} \otimes \Sigma^{\lambda-a_t} (\cU_k/\cU_{k - t}) \otimes \Sigma_\Sp^\eta \cS_{2(n-k)} \right) \quad \Leftrightarrow \quad (\alpha; \eta) = (\lambda; \mu).
\end{equation}
Moreover, if~\((\alpha; \eta) = (\lambda; \mu)\), then the multiplicity of $\Sigma_\Sp^{\hat \mu} \cS_{2(n - k + t)}$ in the decomposition~\eqref{S:inclu} is~\(1\).    
\end{lemma}
 \begin{proof}
Consider a direct summand 
\begin{equation*}
\Sigma^{\gamma} \left(\cU_k / \cU_{k - t} \right)^{\vee} \inplus \Sigma^\alpha \left(\cU_k / \cU_{k - t} \right)^{\vee} \otimes \Sigma^{\lambda-a_t} (\cU_k/\cU_{k - t}).
\end{equation*}
Lemma~\ref{lemma:LR rule} and our assumptions imply
\begin{equation*}
\gamma_t \ge \alpha_t +a_t -\lambda_1 \ge a_t-\lambda_1 \ge -2(n-k)-a_t-1.
\end{equation*}
Thus,~\eqref{S:inclu} holds if and only if~\(\gamma = (a_t)^t\) and~$\eta=\mu$ by Lemma~\ref{lemma:BBW_tilde_mu}. Next, by Lemma~\ref{lemma:LR rule}
\begin{equation*}
\Sigma^{(a_t)^t}(\cU_k/\cU_{k-t})^{\vee}\inplus \Sigma^\alpha \left(\cU_k / \cU_{k - t} \right)^{\vee} \otimes \Sigma^{\lambda-a_t} (\cU_k/\cU_{k - t}) \quad \Leftrightarrow \quad \alpha = \lambda.
\end{equation*}
Moreover,~\(c^0_{\lambda, -\lambda} = 1\) by Lemma~\ref{lemma:PRV}. 
\end{proof}   

\begin{proof}[Proof of Theorem~\ref{theo:formulas_for_F}]
By~\cite[Lemma~A.8]{F19} we have:
    \begin{equation*}
\E{\bullet}_{\bG}\left(\Sigma_\Sp^{\hat \mu} \cS_{2(n-k + t)},\, \Sigma_\Sp^{\mu'} \cS_{2(n-k + t)} \right) = 
    \begin{cases}
		\Bbbk, & \text{if} \ \mu' = \hat \mu;\\
		0, & \text{otherwise.}
    \end{cases}
\end{equation*}
Thus, Lemma~\ref{lemma:computationExtg}, formula~\eqref{obvious-decomposition} and Lemma~\ref{incl} imply
\begin{equation*}
    \E{\bullet}_{\bG}\left(\Sigma^{\alpha}\cU_k\otimes \Sigma^{\eta}_{\Sp}\cS_{2(n-k)},\hat \Phi^t_{\lambda}(\Sigma_\Sp^{\hat \mu} \cS_{2(n-k+t)} )\right) = 
    \begin{cases}
		\Bbbk, & \text{if} \ (\alpha;\eta) = (\lambda;\mu);\\
		0, & \text{if}\ 0 \preceq (\alpha;\eta) \prec (\lambda;\mu),
    \end{cases}
\end{equation*}
that is the second condition of the characterization~\eqref{F_description} of $\cF^{\lambda; \mu}_t$. The first condition of~\eqref{F_description} follows from Lemma~\ref{lemma:Phi(S)_in_A}. Thus, 
\begin{equation*}
\cF^{\lambda; \mu}_t \simeq \hat \Phi^t_{\lambda}( \Sigma_\Sp^{\hat \mu} \cS_{2(n-k+t)}).
\end{equation*}
Lemma~\ref{lemma:Phi(S)_in_A} implies that $\cF^{\lambda; \mu}_t$ is a vector bundle. 
\end{proof}
\begin{remark}\label{rem:F-Lagr}
In the Lagrangian case, another interpretation of the object~\(\cF^\lambda\) for~\(\lambda \in \tilde\B_t\) is given in~\cite[Subsection~3.1]{F19}:
\begin{equation*}
\cF^\lambda\simeq q_{t*}p_t^*(\Sigma_\Sp^\lambda \cS_{2t}),
\end{equation*}
where 
\begin{equation*}
\xymatrix@1{
\IGr(n-t, V) &
\IFl(n-t, n; V) \ar[l]_{p_{t}} \ar[r]^(0.55){q_{t}} &
\IGr(n, V).
}
\end{equation*}
As noted in Remark~\ref{rem:E(a_t)}, in the Lagrangian case the object~\(\cF^\lambda\) does not depend on~\(t\). 
\end{remark}
\begin{theorem} \label{theo:vector.bundle}
For $(\lambda; \mu)\in \B_t$ the object $\cE^{\lambda; \mu}_t$ is a vector bundle. 
\end{theorem}
\begin{proof}
    The statement immediately follows from Lemma~\ref{prop:bigblockk}, Lemma~\ref{prop:0th_block} and Theorem~\ref{theo:formulas_for_F}.
\end{proof}

\section{Relations generated by symplectic staircase complexes} \label{section:relation.sympl}
In this section we assume that~\(k \neq n\). The main goal of this section is to prove Theorem~\ref{lemma:fullness_induction_step} for $t$ of the same parity as $k$ and $t=0$, that is to show 
\begin{equation*}
    \begin{aligned}
        \cA_0(1),\, \tilde \cA_{1}(1) \subset \cD_0 &= \langle \cA_0,\, \cA_1(1)\rangle,\\
        \cA_t(t+1),\, \tilde \cA_{t + 1}(t+1) \subset \cD_t &= \langle \cA_{t-1}(t),\,\tilde \cA_{t}(t),\, \cA_{t + 1}(t+1) \rangle\qquad \text{for}\quad 1 \le t\le k-2, \quad t\equiv k\  (\mathrm{mod}\ 2). 
    \end{aligned}
\end{equation*}
Note that if~\(k-t\) is even then~\(a_{t-1}=\lfloor \frac{k-t+1}{2}\rfloor=a_t\) and~\(\cA_{t-1} \subset \tilde\cA_t\).
Moreover,~\(\cA_0=\tilde \cA_0\). Therefore, in both cases,~\(\cD_t = \langle\tilde \cA_{t}(t),\, \cA_{t + 1}(t+1) \rangle\), and we can reformulate these inclusions in the following form:
\begin{equation}\label{eq:A-in-D-even}
    \cA_t(t+1),\, \tilde \cA_{t + 1}(t+1) \subset \cD_t = \langle \tilde \cA_{t}(t),\, \cA_{t + 1}(t+1) \rangle \quad \text{for either} \quad t=0 \quad \text{or} \quad \text{\(k-t =2a_t\)}, \quad 0< a_t.
\end{equation}
These relations follow from the symplectic staircase complexes~\eqref{eq:SSC}.


Let us sketch the proof.
Let $k-t=2a_t$. First, recall that we denote by 
\begin{equation*}
\hat \Phi^t_{\lambda+1}(-) \coloneqq q_{t*}(\Sigma^{\lambda+1- a_t} (\cU_k/\cU_{k-t})\otimes p_t^*(-))\colon \IGr(k-t,V)\to \IGr(k,V)
\end{equation*} 
the Fourier--Mukai transform~\eqref{eq:Phi}, where
$\xymatrix@1{
\IGr(k-t, V) &
\IFl(k-t, k; V) \ar[l]_{p_{t}} \ar[r]^(0.55){q_{t}} &
\IGr(k, V).
}$
Recall also that for a Young diagram~\(\mu \in \YD^{a_t}_{n-k}\) we denote~$\hat \mu=((a_t)^t,\mu)\in \YD^{a_t}_{n-k+t}$.

Consider the symplectic staircase complexes on~\(\IGr(k-t, V)\):
\begin{multline} \label{eq:SSSk-t}
0 \to \Sigma_{\Sp}^{\hat \mu} \cS_{2(n-k+t)} (-H_{k-t}) \to \Lambda^{b^1} \cQ_{2n-k+t} \otimes \Sigma_{\Sp}^{\hat \mu^{(1)}} \cS_{2(n-k+t)} (-H_{k-t}) \to \dots\\
\to \Lambda^{b^{a_t}} \cQ_{2n-k+t} \otimes \Sigma_{\Sp}^{\hat \mu^{(a_t)}} \cS_{2(n-k+t)} (-H_{n-k}) 
\to \Lambda^{b^{a_t}} \cQ^{\vee}_{2n-k+t} \otimes \Sigma_{\Sp}^{\hat \mu^{(a_t)}} \cS_{2(n-k+t)} \to \dots \\
\to \Lambda^{b^1} \cQ_{2n-k+t}^{\vee} \otimes \Sigma_{\Sp}^{\hat \mu^{(1)}} \cS_{2(n-k+t)} 
\to \Sigma_{\Sp}^{\hat \mu} \cS_{2(n-k+t)} \to 0,
\end{multline} 
These symplectic staircase complexes provide some explicit relations between vector bundles $\Sigma_{\Sp}^{\hat \mu} \cS_{2(n-k+t)}$ and $\Sigma_{\Sp}^{\hat \mu} \cS_{2(n-k+t)} (-H_{k-t})$, see Lemma~\ref{lemma:SSC-middle-terms} for the details. For $t=0$ these relations allow us to prove that~$\cA_0(1), \tilde \cA_{1}(1) \subset \cD_0$, see Proposition~\ref{lemma:base_of_induction}. 

To prove the inclusions~\eqref{eq:A-in-D-even} 
we apply~$\hat \Phi^t_{\lambda+1}$ to the relations constructed in Lemma~\ref{lemma:SSC-middle-terms}, see Lemma~\ref{lemma:F-inclusion}. By analyzing the filtrations on the objects appearing in Lemma~\ref{lemma:F-inclusion}, we obtain explicit relations between the objects $\cE_t^{\lambda; \mu}(1)$, $\cE_t^{\lambda+1;\mu}$, and the subcategory $\cA_{t+1}(1)$; see Proposition~\ref{proposition:incl_e}. These relations allow us to conclude the proof of Theorem~\ref{lemma:fullness_induction_step} for $t$ of the same parity as $k$.

First, let us prove that the middle terms of the complex~\eqref{eq:SSSk-t} belong to the following subcategory:
    \begin{equation} \label{def:B_t}
        \cB_t \coloneqq \langle S^{m}\cU_{k-t}(-H_{k-t}) \otimes \Sigma_{\Sp}^{\nu} \cS_{2(n-k+t)} \mid 0\le m \le 2n-k-1, \nu \in \YD^{a_t-1}_{n-k+t}\rangle \subset \Db(\IGr(k-t,V)).
    \end{equation}
\begin{lemma}\label{lemma:SSC-middle-terms}
Suppose~\(k-t>0\) is even and let~\(\mu \in \YD^{a_t}_{n-k}\).
\begin{enumerate}
\item If~\(\mu_1 = a_t\) then
\begin{equation*}
  \Sigma_{\Sp}^{\hat \mu} \cS_{2(n-k+t)} \in \left\langle \cB_{t}, \Sigma_{\Sp}^{\hat \mu} \cS_{2(n-k+t)} (-H_{k-t}) \right\rangle, \quad \text{and}\quad \Sigma_{\Sp}^{\hat \mu} \cS_{2(n-k+t)}(-H_{k-t}) \in \left\langle \cB_{t}, \Sigma_{\Sp}^{\hat \mu} \cS_{2(n-k+t)}  \right\rangle. 
  \end{equation*}
\item If~\(\mu_1 < a_t\) then 
  \begin{equation*}
  S^{2n - k}\cU_{k - t}(-H_{k-t})\otimes \Sigma_\Sp^{((a_t-1)^t,\, \mu)} \cS_{2(n-k+t)} \in \left\langle \cB_{t}, \Sigma_{\Sp}^{\hat \mu} \cS_{2(n-k+t)} (-H_{k-t}), \Sigma_{\Sp}^{\hat \mu} \cS_{2(n-k+t)} \right\rangle.      
  \end{equation*}
\end{enumerate}
\end{lemma}
\begin{proof}
To prove the statement it is enough to show that the middle terms of the complex~\eqref{eq:SSSk-t} on~$\IGr(k-t,V)$ satisfy:
\begin{align} \label{incl-B-1}
  \Lambda^{b^i} \cQ_{2n-k+t}(-H_{k-t}) \otimes \Sigma_{\Sp}^{\hat \mu^{(i)}} \cS_{2(n-k+t)} & \in \cB_{t} &\text{for} \quad i = 1, \dots, a_t;\\
  \label{incl-B-2}
  \Lambda^{b^i} \cQ_{2n-k+t}^\vee \otimes \Sigma_{\Sp}^{\hat \mu^{(i)}} \cS_{2(n-k+t)} & \in \cB_{t} &\text{for} \quad i = 2, \dots, a_t,
\intertext{and}
\label{incl-B-3}
\Lambda^{b^1} \cQ^{\vee}_{2n-k+t} \otimes \Sigma_{\Sp}^{\hat \mu^{(1)}} \cS_{2(n-k+t)}& \in \cB_t & \text{if} \quad \mu_1 = a_t;\\ \label{incl-B-4}
 S^{2n - k}\cU_{k - t}(-H_{k-t})\otimes \Sigma_\Sp^{((a_t-1)^t,\,\mu)} \cS_{2(n-k+t)} & \in \left\langle \cB_t, \Lambda^{b^1} \cQ^{\vee}_{2n-k+t} \otimes \Sigma_{\Sp}^{\hat \mu^{(1)}} \cS_{2(n-k+t)} \right\rangle & \text{if}\quad \mu_1 < a_t.
\end{align} 

Consider the dual complexes to the Koszul complexes~\eqref{Koszul} on $\IGr(k-t,V)$ for~\(0 \le m \le 2n-k\):
\begin{equation}\label{eq:Koszul-dual}
     0\to \Lambda^{m}\cQ_{2n-k+t}\to  \Lambda^{m}V \otimes \cO \to \ldots \to V\otimes S^{m-1}\cU_{k-t} \to S^{m}\cU_{k-t} \to 0.
\end{equation}
Together with the isomorphism $\Lambda^{m} \cQ^\vee_{2n-k + t}\simeq \Lambda^{2n-k+t-m}\cQ_{2n-k+t}(-H_{k-t})$, these 
complexes imply the following inclusions:
\begin{align} \label{eq:Q-incl}
    \Lambda^{m} \cQ_{2n-k+t} (-H_{k-t}) \otimes \Sigma_{\Sp}^{\hat \mu^{(i)}} \cS_{2(n-k+t)} &\in \left\langle S^{j}\cU_{k-t}(-H_{k-t})  \otimes \Sigma_{\Sp}^{\hat \mu^{(i)}} \cS_{2(n-k+t)}\right\rangle_{0\le j \le m},  \\\label{eq:Q-incl1}
    \Lambda^{2n-k+t-m} \cQ_{2n-k+t} \otimes \Sigma_{\Sp}^{\hat \mu^{(i)}} \cS_{2(n-k+t)} &\in  \left\langle S^{j}\cU_{k-t}(-H_{k-t}) \otimes \Sigma_{\Sp}^{\hat \mu^{(i)}} \cS_{2(n-k+t)} \right\rangle_{0\le j\le m}.
\end{align}
Moreover, using the complex~\eqref{eq:Koszul-dual} as a left resolution for~\(S^m \cU_{k-t}\), we analogously get
\begin{multline}\label{eq:Q-incl2}
    S^{m}\cU_{k-t}(-H_{k-t}) \otimes \Sigma_{\Sp}^{\hat \mu^{(i)}} \cS_{2(n-k+t)} \in  \\
    \left\langle \Lambda^{2n-k-t-m} \cQ_{2n-k+t} \otimes \Sigma_{\Sp}^{\hat \mu^{(i)}} \cS_{2(n-k+t)},\, S^{j}\cU_{k-t}(-H_{k-t}) \otimes \Sigma_{\Sp}^{\hat \mu^{(i)}} \cS_{2(n-k+t)} \right\rangle_{0\le j \le m-1}.
\end{multline}

Applying the explicit formula~\eqref{explicit_formula_i} to $\hat \mu^{(i)}$, we deduce:
\begin{equation} \label{range}
    \hat \mu^{(i)}\in \YD^{a_t-1}_{n-k+t} \quad \text{and}\quad b^i=|\hat \mu| - |\hat \mu^{(i)}| \in [t, n-a_t-1] \qquad \text{for}\quad i = 1, \ldots, a_t.
\end{equation}
The constraints described in~\eqref{range} together with the inclusion~\eqref{eq:Q-incl} for~\(m=b^i\) imply~\eqref{incl-B-1}. 

Moreover, note that for~\(i\ge 2\) we have~\(b^i>b^1\ge t\). Thus, the inclusion~\eqref{eq:Q-incl1} for~\(m=2n-k+t-b^i\), combined with the conditions~\eqref{range}, implies~\eqref{incl-B-2}. Furthermore, if~\(\mu_1 = a_t\), then~\(b^1>t\), and the same argument implies~\eqref{incl-B-3}. 

Now, we consider the last case $\mu_1<a_t$. 
In this case~\(\hat \mu^{(1)}= ((a_t-1)^t, \mu)\) and~\(b^1=t\), so the inclusion~\eqref{eq:Q-incl2} for $m=2n-k$ and $i=1$ implies~\eqref{incl-B-4}.
\end{proof}
We are now ready to prove Theorem~\ref{lemma:fullness_induction_step} for $t=0$. 
\begin{proposition} \label{lemma:base_of_induction}
    We have~\(\cA_0(1),\, \tilde \cA_{1}(1) \subset \cD_0= \langle \cA_0,\, \cA_1(1)\rangle\). 
\end{proposition}
\begin{proof}
Recall that 
\begin{align*}
    &\cA_0 \phantom{(1)}= \left\langle \Sigma_{\Sp}^{\mu} \cS_{2(n-k)} \mid \lfloor k/2 \rfloor \ge \mu_1\ge \mu_2\ge \ldots \ge \mu_{n-k}\ge 0\right\rangle;\\
    &\cA_{1}(1) = \left\langle S^{j}\cU^{\vee}(1)\otimes \Sigma_{\Sp}^{\mu} \cS_{2(n-k)} \mid 2n-k-1\ge j \ge 0;\, \lfloor (k-1)/2 \rfloor \ge \mu_1\ge \mu_2\ge \ldots \ge \mu_{n-k}\ge 0\right\rangle;\\
    &\tilde \cA_{1}(1) = \left\langle S^{j}\cU^{\vee}(1)\otimes \Sigma_{\Sp}^{\mu} \cS_{2(n-k)} \mid 2n-k \phantom{{}-1}\ge j\ge 0;\, \lfloor (k-1)/2 \rfloor \ge \mu_1\ge \mu_2\ge \ldots \ge \mu_{n-k}\ge 0\right\rangle.
\end{align*}

To prove $\tilde \cA_{1}(1) \subset \cD_0$, 
it is enough to show 
\begin{equation*}
S^{2n-k}\cU^{\vee}(1)\otimes \Sigma_{\Sp}^{\mu} \cS^{\vee}_{2n-k} \in \cD_0 \qquad \text{for} \quad \mu \in \YD^{\lfloor (k-1)/2 \rfloor}_{n-k}.
\end{equation*} 
Consider the Koszul complex~\eqref{Koszul} for $m = 2n-k$ tensored with~$\Sigma_{\Sp}^{\mu} \cS_{2(n-k)}(1)$:
\begin{multline*}
    0\to \Sigma_{\Sp}^{\mu} \cS_{2(n-k)}\to  \Lambda^{2n-k}V^{\vee}\otimes \Sigma_{\Sp}^{\mu} \cS_{2(n-k)}(1)\to \Lambda^{2n-k-1}V^{\vee}\otimes \cU^{\vee}(1)\otimes \Sigma_{\Sp}^{\mu} \cS_{2(n-k)} \to \ldots \\
    \to V^{\vee}\otimes S^{2n-k-1}\cU^{\vee}(1)\otimes \Sigma_{\Sp}^{\mu} \cS_{2(n-k)}\to S^{2n-k}\cU^{\vee}(1)\otimes \Sigma_{\Sp}^{\mu} \cS_{2(n-k)}\to 0.
\end{multline*}
Since all terms of the complex, except for $S^{2n-k}\cU^{\vee}(1)\otimes \Sigma_{\Sp}^{\mu} \cS_{2(n-k)}$, belong to~\(\cD_0\), it follows that the vector bundle~$S^{2n-k}\cU^{\vee}(1)\otimes \Sigma_{\Sp}^{\mu} \cS_{2(n-k)}$ must also belong to~\(\cD_0\). Thus, the inclusion~$\tilde \cA_{1}(1)\subset \cD_0$ follows.

If $k$ is odd then $\lfloor k/2 \rfloor= \lfloor (k-1)/2 \rfloor$, i.e. $\cA_0(1)\subset \cA_1(1)$, so we deduce the proposition. 

Next, if $k$ is even, it remains to prove 
\begin{equation*}
\Sigma_{\Sp}^{\mu} \cS_{2(n-k)}(1) \in \cD_0 \qquad \text{for} \quad \mu \in \YD^{k/2}_{n-k}\quad \text{with} \quad \mu_1=k/2.    
\end{equation*}
Observe that when $t=0$ and $k$ is even, the category dual to~\eqref{def:B_t} coincides with~$\cA_1(1)$:
\begin{equation*}
        \cB_0^{\vee} \coloneqq \left\langle S^{m}\cU^{\vee}_{k}(1) \otimes \Sigma_{\Sp}^{\nu} \cS_{2(n-k)} \mid m \in \YD_1^{2n-k-1},\, \nu \in \YD^{k/2-1}_{n-k}\right\rangle = \cA_1(1).
    \end{equation*}
Thus, applying Lemma~\ref{lemma:SSC-middle-terms} in the case $t=0$ and $\mu_1=k/2$, we deduce
\begin{equation*}
\Sigma_{\Sp}^{\mu} \cS_{2(n-k)}(1) \in \big\langle \cB_{0}^{\vee}, \Sigma_{\Sp}^{\mu} \cS_{2(n-k)} \big\rangle\subset \big\langle \cA_1(1),\, \cA_0\big\rangle=\cD_0. \qedhere
\end{equation*}
\end{proof}
The remaining part of Section~\ref{section:relation.sympl} is devoted to proving Lemma~\ref{lemma:fullness_induction_step} for~$t>0$ of the same parity as $k$. 
First, we apply the Fourier--Mukai transform $\hat \Phi^t_{\lambda+1}$, defined in~\eqref{eq:Phi}, to the inclusions in Lemma~\ref{lemma:SSC-middle-terms}.
\begin{lemma} \label{lemma:F-inclusion}
   Suppose~\(k-t>0\) is even, and~\((\lambda;\mu) \in \B_t=\YD_t^{2n-k-t}\times \YD_{n-k}^{a_t}\). Then 
 \begin{align*}
  \cF_t^{\lambda; \mu}(-1) & \in \left\langle \hat \Phi^t_{\lambda+1}(\cB_{t}),\, \cF_t^{\lambda+1;\mu} \right\rangle, \qquad & \text{if}\quad \mu_1 = a_t; \\
  \hat \Phi^t_{\lambda+1}\left(S^{2n - k}\cU_{k - t}(-H_{k-t})\otimes \Sigma_\Sp^{((a_t-1)^t,\, \mu)} \cS_{2(n-k+t)}\right) & \in \left\langle \hat \Phi^t_{\lambda+1}(\cB_{t}),\, \cF_t^{\lambda+1;\mu},\, \cF_t^{\lambda; \mu}(-1)\right\rangle,  \qquad & \text{if} \quad \mu_1 < a_t. 
\end{align*}
\end{lemma}
\begin{proof}
Theorem~\ref{theo:formulas_for_F} yields~\(\cF_t^{\lambda+1;\mu}\simeq \hat \Phi^t_{\lambda+1}(\Sigma_\Sp^{\hat \mu} \cS_{2(n-k+t)})\) and, again, Theorem~\ref{theo:formulas_for_F} and the projection formula yield
\begin{align*}
  & \cF_t^{\lambda; \mu}(-1)\simeq \hat \Phi^t_{\lambda}\left(\Sigma_\Sp^{\hat \mu} \cS_{2(n-k+t)}\right) \otimes \cO(-H_k) = \\
  & = q_{t*}\left(\Sigma^{\lambda - a_t \phantom{{}+1}} (\cU_k/\cU_{k-t})\otimes p^*_{t}\big(\Sigma_{\Sp}^{\hat \mu} \cS_{2(n-k+t)}\big)\right)\otimes \cO(-H_k) \simeq \\
  & \simeq q_{t*}\left(\Sigma^{\lambda - a_t+1} (\cU_k/\cU_{k-t})\otimes p^*_{t}\big(\Sigma_{\Sp}^{\hat \mu} \cS_{2(n-k+t)}\big)\otimes \cO(-H_{k-t})\right)\simeq \\
  & \simeq q_{t*}\left(\Sigma^{\lambda- a_t+1} (\cU_k/\cU_{k-t})\otimes p^*_t\big(\Sigma_{\Sp}^{\hat \mu} \cS_{2(n-k+t)}(-H_{k-t})\big) \right) \simeq\\
  &\simeq \hat \Phi^t_{\lambda+1}\left(\Sigma_\Sp^{\hat \mu} \cS_{2(n-k+t)}(-H_{k-t})\right).  
\end{align*}
Thus, applying the Fourier--Mukai transform $\hat \Phi^t_{\lambda+1}$ to the inclusions in Lemma~\ref{lemma:SSC-middle-terms} we deduce the lemma. 
\end{proof}
Our next goal is to study in more detail the objects appearing in Lemma~\ref{lemma:F-inclusion}, namely the object 
\begin{equation*}
\hat \Phi^t_{\lambda+1}\left(S^{2n - k}\cU_{k - t}(-H_{k-t})\otimes \Sigma_\Sp^{((a_t-1)^t,\, \mu)} \cS_{2(n-k+t)}\right),
\end{equation*}
and the subcategory
\begin{samepage}
\begin{equation} \label{Phi_B_t}
    \hat \Phi^t_{\lambda+1}(\cB_t) = \left\langle \hat \Phi^t_{\lambda+1}\big(S^{m}\cU_{k-t}(-H_{k-t}) \otimes \Sigma_{\Sp}^{\nu} \cS_{2(n-k+t)}\big)\right\rangle \subset \Db(\IGr(k,V)), 
\end{equation}
where $0\le m \le 2n-k-1$ and $\nu \in \YD^{a_t-1}_{n-k+t}$. 
\end{samepage}
\begin{lemma} \label{lemma:phi-filtr}
Suppose~\(k-t > 0\) is even. 
Let~$(\lambda; \nu) \in \YD_{t}\times \YD^{a_t-1}_{n-k+t}$, and let $m\in \mathbb{Z}$. 
The object
\begin{equation*}
\hat \Phi^t_{\lambda+1}\big(S^{m}\cU_{k-t}(-H_{k-t}) \otimes\Sigma_{\Sp}^{\nu} \cS_{2(n-k+t)}\big)\in \IGr(k,V)
\end{equation*}
admits a filtration whose associated subquotients are of the form
\begin{equation*}
q_{t*}\big(S^{m} \cU_{k-t}\otimes \Sigma^{\gamma} (\cU_k/\cU_{k - t})\big) \otimes \Sigma_\Sp^{\eta} \cS_{2(n - k)}(-1),
    \end{equation*} 
where~\(\gamma \in \Y_t\) satisfies~\(\gamma_i \in [-k+t+1, \lambda_i-1]\), 
and~\(\eta \in \YD^{a_t-1}_{n - k}\).
\end{lemma}
\begin{proof}
First, observe that: 
\begin{align*}
     \hat \Phi^t_{\lambda+1}&\big(S^{m}\cU_{k-t}(-H_{k-t}) \otimes\Sigma_{\Sp}^{\nu} \cS_{2(n-k+t)}\big)=\\
    =q_{t*} &\left(S^{m} \cU_{k-t} \otimes \Sigma^{\lambda- a_t+1} (\cU_k/\cU_{k - t}) \otimes p^*_t\Sigma_\Sp^{\nu} \cS_{2(n-k+t)} \otimes \cO(-H_{k-t}) \right) \simeq  \\
    \simeq q_{t*} & \left(S^{m} \cU_{k-t} \otimes\Sigma^{\lambda- a_t\phantom{{}+1}} (\cU_k/\cU_{k - t}) \otimes  p^*_t \Sigma_\Sp^{\nu} \cS_{2(n-k+t)} \otimes \cO(-H_{k})\right) \simeq \\
    \simeq q_{t*} &\left(S^{m} \cU_{k-t}\otimes \Sigma^{\lambda-a_t\phantom{{}+1}} (\cU_k/\cU_{k - t})\otimes p^*_t \Sigma_\Sp^{\nu} \cS_{2(n-k+t)} \right) \otimes \cO(-1),
\end{align*}
where the final isomorphism follows from the projection formula.

Next, by Lemma~\ref{lemma:filtration_on_flags} applied to~\(\xi=\lambda-a_t\) and~\(\theta=\nu\), under our assumptions, the vector bundle
\begin{equation*}
    S^{m} \cU_{k-t} \otimes \Sigma^{\lambda-a_t} (\cU_k/\cU_{k - t})\otimes  p^*_t \Sigma_\Sp^{\nu} \cS_{2(n-k+t)}
\end{equation*}
admits a filtration whose associated subquotients are of the form
\begin{equation*}
S^{m} \cU_{k-t}\otimes \Sigma^{\gamma} (\cU_k/\cU_{k - t}) \otimes \Sigma_\Sp^{\eta} \cS_{2(n - k)}, 
    \end{equation*}    
where~\(\gamma \in \Y_t\) satisfies~\(\gamma_i \in [- k+t+1, \lambda_i-1]\) and~\(\eta \in \YD^{a_t-1}_{n - k}\).

Applying the projection formula again yields:
\begin{equation*}
q_{t*}\left(S^{m} \cU_{k-t}\otimes \Sigma^{\gamma} (\cU_k/\cU_{k - t}) \otimes \Sigma_\Sp^{\eta} \cS_{2(n - k)} \right) = q_{t*}\big(S^{m} \cU_{k-t}\otimes \Sigma^{\gamma} (\cU_k/\cU_{k - t})\big)\otimes \Sigma_\Sp^\eta \cS_{2(n-k)}.
\end{equation*}
Combining all of the above we get the lemma.
\end{proof}
\begin{corollary} \label{corollary:second-filtr-phi}
Suppose~\(k-t > 0\) is even. 
Let~\((\lambda;\mu) \in \YD_t\times \YD_{n-k}^{a_t-1}\).
The object 
\begin{equation*}
\hat \Phi^t_{\lambda+1}\big(S^{2n-k}\cU_{k-t}(-H_{k-t}) \otimes\Sigma_{\Sp}^{((a_t-1)^t,\, \mu)} \cS_{2(n-k+t)}\big)\in \IGr(k,V)
\end{equation*}
admits a filtration whose associated subquotients are:
\begin{itemize}
    \item the term 
    \begin{equation*}
q_{t*}\big(S^{2n-k} \cU_{k-t}\otimes\Sigma^{\lambda-1}(\cU_k/\cU_{k-t})\big)\otimes \Sigma_\Sp^\mu \cS_{2(n-k)}(-1)
\end{equation*} 
    appearing with multiplicity exactly~one;
\item  terms of the form 
\begin{equation*}
q_{t*}\big(S^{2n-k} \cU_{k-t}\otimes \Sigma^{\gamma} (\cU_k/\cU_{k-t})\big)\otimes \Sigma_\Sp^\eta \cS_{2(n-k)}(-1),
\end{equation*}
where~\(\gamma \in \Y_t\) satisfies~\(\gamma < \lambda-1\) and~\(\eta\in \YD_{n-k}^{a_t-1}\).
\end{itemize}
\end{corollary}
\begin{proof}
We obtain the filtration with weights~\(\gamma \le \lambda-1\) by Lemma~\ref{lemma:phi-filtr} for~$m=2n-k$ and~$\nu=((a_t-1)^t, \mu)$. 
We compute the multiplicities of the terms with~\(\gamma=\lambda-1\) by Lemma~\ref{lemma:filtration_on_flags1} for~\(a=a_t-1\).
\end{proof}
Using the filtrations from Lemma~\ref{lemma:phi-filtr} and Corollary~\ref{corollary:second-filtr-phi}, together with the Borel–-Bott–-Weil theorem, we can explicitly describe the objects appearing in Lemma~\ref{lemma:F-inclusion}. 
\begin{lemma} \label{lemma:phi-BBW}
Suppose~\(k-t>0\) is even. Let~\((\lambda;\mu) \in \YD_t^{2n-k-t}\times \YD_{n-k}^{a_t-1}\). 
\begin{enumerate}
\item We have 
\begin{equation*}
\big(\Phi^t_{\lambda+1}(\cB_t)\big)^{\vee} \in \cA_{t+1}(1).
\end{equation*}
\item The object 
\begin{equation*}
\Phi^t_{\lambda+1}\big(S^{2n-k}\cU_{k - t}(-H_{k-t}) \otimes \Sigma_\Sp^{((a_t-1)^t,\, \mu)} \cS_{2(n-k+t)}\big)^{\vee}
\end{equation*} 
admits a filtration whose associated subquotients are:
\begin{itemize}
    \item the term $\Sigma^{(2n-k-t,\, \lambda)}\cU_k^\vee(1)\otimes \Sigma_\Sp^\mu \cS_{2(n-k)}[-t]$
appearing with multiplicity exactly one;
\item terms of the form~\(\Sigma^{(2n-k-t, \, \beta)} \cU_k^\vee(1)\otimes \Sigma_\Sp^\eta \cS_{2(n-k)}[-t]\), where~\(\beta \in \YD_t^{2n-k-t}\) satisfies~\(\beta < \lambda\), and~\(\eta\in \YD_{n-k}^{a_t-1}\).
\end{itemize} 
\end{enumerate}
\end{lemma}
\begin{proof}
Examining~\eqref{Phi_B_t}, Lemma~\ref{lemma:phi-filtr}, and Corollary~\ref{corollary:second-filtr-phi}, we see that to prove the lemma we need to compute
\begin{equation*}
q_{t*}\big(S^m \cU_{k - t} \otimes \Sigma^{\gamma} (\cU_k/\cU_{k - t})\big) \qquad \text{for}\quad 0\le m \le 2n-k \quad \text{and}\quad \gamma \in \Y_t,
\end{equation*}
where~\(\gamma\) satisfies~\(-k+t+1\le \gamma_i \le \lambda_i-1\le 2n-k-t-1\). 

To simplify the notation, set
\begin{align*}
&\alpha = (-m,-\gamma_t,\ldots,-\gamma_1);\\
& w=2n-k-t.   
\end{align*}

To compute $q_{t*}(S^m \cU_{k - t} \otimes \Sigma^{\gamma} (\cU_k/\cU_{k - t}))$, we apply Theorem~\ref{theo:relative_BBW} (Borel--Bott--Weil) to $X=\IGr(k,V)$, the vector bundle~$\cU_k$ and the projection $q_t$. As an intermediate step, we obtain:
\begin{equation*}
((0)^{k-t-1},\alpha) + (k,k-1, \ldots, 1)=(k,k-1,\ldots, t+1-m,t-\gamma_t, \ldots, 1-\gamma_1).
\end{equation*}
Since~\(\gamma_i \in [- k+t+1, w-1]\), we deduce
\begin{equation*}
k-1 \ge  t-\gamma_t \ge t-w+1.
\end{equation*} 
If $k-1 \ge t-\gamma_t \ge t+2$, then Theorem~\ref{theo:relative_BBW} gives 
\begin{equation*}
q_{t*}\big(S^m \cU_{k - t} \otimes \Sigma^{\gamma} (\cU_k/\cU_{k - t})\big)=0.
\end{equation*}

Next, assume~$t+1 \ge t-\gamma_t \ge t-w+1$. It is enough to consider the case where the pushforward is non-zero. In this case, Lemma~\ref{BBW-SC} implies 
\begin{equation} \label{filtr_goal}
    q_{t*}\big(S^m \cU_{k - t} \otimes \Sigma^{\gamma} \big)=\Sigma^{-\alpha^{(0)}} \cU_k[-\dd(\alpha)+1], 
\end{equation} 
where~\(-\alpha^{(0)}\in \YD^{\max\{\gamma_1,\, m-t\}}_{t+1}\).

By~\eqref{Phi_B_t} and Lemma~\ref{lemma:phi-filtr}, to prove (1), it suffices to show that for $m\le 2n-k-1$, we have
\begin{equation*} 
q_{t*}\big(S^m \cU_{k - t} \otimes \Sigma^{\gamma} (\cU_k/\cU_{k - t})\big)\in \cA_{t+1}^{\vee}.    
\end{equation*}
If~$m\le 2n-k-1$ then~$\max\{\gamma_1, m-t\}\le w-1$, so~\(-\alpha^{(0)} \in \YD_{t+1}^{w-1}\), and hence~\eqref{filtr_goal} implies (1).

Using Corollary~\ref{corollary:second-filtr-phi}, to prove (2) it remains to establish
\begin{equation} \label{dir1}
   q_{t*}\big(S^{2n-k} \cU_{k-t}\otimes\Sigma^{\lambda-1}(\cU_k/\cU_{k-t})\big)=\Sigma^{(w,\, \lambda)}\cU_k[-t],
\end{equation}
and to verify that for all~\(\gamma \in \Y_t\) with~\(\gamma < \lambda-1\), we have
\begin{equation} \label{dir2}
  q_{t*}\big(S^{2n-k} \cU_{k-t}\otimes \Sigma^{\gamma} (\cU_k/\cU_{k-t})\big)=\Sigma^{(w,\, \beta)}\cU_k[-t],\quad \text{where} \quad \beta\in \YD_t \quad \text{such that}\quad \beta<\lambda.  
\end{equation}

If $m=2n-k$ then~\(\dd_{-\gamma}(-m) = t+1\), so by the last line of~\eqref{eq:def-bbw}:
\begin{equation*}
\alpha^{(0)}=(k-2n, -\gamma_t,\ldots, -\gamma_1)^{(0)}=(-\gamma_t-1,\ldots, -\gamma_1-1, -w),
\end{equation*}
so~\(-\alpha^{(0)} = (w,\gamma+1)\) and~\(q_{t*}\big(S^m \cU_{k - t} \otimes \Sigma^{\gamma} (\cU_k/\cU_{k - t})\big) \neq 0\). Thus, the formulas~\eqref{dir1} and~\eqref{dir2} follow from~\eqref{filtr_goal}, and we deduce (2).
\end{proof}
We are now ready to prove the key proposition.
\begin{proposition} \label{proposition:incl_e}
Suppose~\(k-t>0\) is even and~\((\lambda;\mu) \in \B_t=\YD_t^{2n-k-t}\times \YD_{n-k}^{a_t}\). 
\begin{enumerate}
    \item If~\(\mu_1 = a_t\), then
\begin{equation*}
  \cE_t^{\lambda; \mu}(1) \in \left\langle \cE_t^{\lambda+1;\mu},\, \cA_{t+1}(1) \right\rangle.
  \end{equation*}
\item If~\(\mu_1 < a_t\), then 
  \begin{equation*}
  \Sigma^{(2n-k-t,\, \lambda)}\cU_k^\vee(1)\otimes \Sigma_\Sp^\mu \cS_{2(n-k)} \in \left\langle \cE_t^{\lambda; \mu}(1), \,\cE_t^{\lambda+1;\mu},\, \cA_{t+1}(1), \, \Sigma^{(2n-k-t,\, \beta)} \cU_k^\vee(1)\otimes \Sigma_\Sp^\eta \cS_{2(n-k)}\right\rangle, 
  \end{equation*}
  where the last group of terms runs over various pairs $(\beta,\eta) \in \YD_t^{2n-k-t}\times \YD_{n-k}^{a_t-1}$ satisfying~\(\beta < \lambda\).
\end{enumerate}
\end{proposition}
\begin{proof}
The dual statements corresponding to the inclusions in Lemma~\ref{lemma:F-inclusion} are:
 \begin{align*}
 \cE_t^{\lambda; \mu}(1)  & \in \left\langle \big(\hat \Phi^t_{\lambda+1}(\cB_{t})\big)^{\vee},\, \cE_t^{\lambda+1;\mu} \right\rangle \qquad & \text{if} \quad \mu_1 = a_t; \\
  \left(\hat \Phi^t_{\lambda+1}(S^{2n - k}\cU_{k - t}(-H_{k-t})\otimes \Sigma_\Sp^{((a_t-1)^t,\mu)} \cS_{2(n-k+t)})\right)^{\vee} & \in \left\langle (\hat \Phi^t_{\lambda+1}\big(\cB_{t})\big)^{\vee},\, \cE_t^{\lambda+1;\mu},\, \cE_t^{\lambda; \mu}(1) \right\rangle  \qquad & \text{if} \quad\mu_1 < a_t. 
\end{align*}
Thus, the statements follow from Lemma~\ref{lemma:phi-BBW}.
\end{proof} 
It remains to prove the inclusions~\eqref{eq:A-in-D-even} for~\(t>0\). We split the proof into two parts.
\begin{proposition}\label{prop:mutations_even_t+1}
    Let~\(1 \le t \le k-2\) be such that $k-t$ is even. Then
    \begin{equation*}
\tilde \cA_{t + 1}(t+1) \subset 
    \cD_t.
\end{equation*}
\end{proposition}
\begin{proof}
By assumption, $k-t$ is even, so $a_{t+1}=a_t-1$ and~\(\cD_t = \langle \tilde \cA_{t}(t),\, \cA_{t + 1}(t+1) \rangle\) by~\eqref{eq:A-in-D-even}. Recall that 
\begin{align*}
    \cA_{t + 1}(t+1) = \left\langle \Sigma^\alpha \cU^{\vee}_k(t+1) \otimes \Sigma_\Sp^\mu \cS_{2(n-k)} \mid \alpha \in \YD_{t+1}^{2n-k-t-1}; \mu\in \YD_{n-k}^{a_t-1}\right\rangle,\\
    \tilde \cA_{t + 1}(t+1) = \left\langle \Sigma^\alpha \cU^{\vee}_k(t+1) \otimes \Sigma_\Sp^\mu \cS_{2(n-k)} \mid \alpha \in \YD_{t+1}^{2n-k-t\phantom{{}-1}}; \mu\in \YD_{n-k}^{a_t-1}\right\rangle.
\end{align*}
We will prove that $\Sigma^\alpha \cU^{\vee}_k(t+1) \otimes \Sigma_\Sp^\mu \cS_{2(n-k)}\in \cD_t$ by induction on the number of entries of $\alpha$ equal to~$2n-k-t$, denoted 
\begin{equation*}
l(\alpha)=\#\ \{i\mid \alpha_i = 2n-k-t\}.
\end{equation*}
Note that $0\le l(\alpha)\le t+1$.

Base case: if $l(\alpha)=0$, then~\(\alpha\in \YD_{t+1}^{2n-k-t-1}\). In this case, by definition of~\(\cA_{t+1}\) we have
\begin{equation*}
\Sigma^\alpha \cU^{\vee}_k(t+1) \otimes \Sigma_\Sp^\mu \cS_{2(n-k)} \in \cA_{t+1}(t+1)\subset \cD_t.
\end{equation*}

Induction step: assume that 
\begin{equation*}
\Sigma^\alpha \cU^{\vee}_k(t+1) \otimes \Sigma_\Sp^\mu \cS_{2(n-k)} \in \cD_t
\end{equation*} for all $\alpha\in \YD^{2n-k-t}_{t+1}$ with $l(\alpha)\le m$. Let~\(\lambda \in \YD_{t}^{2n-k-t}\) be such that $l(\lambda)=m+1$. Consider~$\bar \lambda=(\lambda_2,\ldots,\lambda_{t+1})\in \YD^{2n-k-t}_{t}$. Then, by Proposition~\ref{proposition:incl_e}, we have
\begin{equation} \label{inclusion_Sigma0}
  \Sigma^{\lambda}\cU_k^\vee(t+1)\otimes \Sigma_\Sp^\mu \cS_{2(n-k)} \in \left\langle \cE_t^{\bar \lambda; \mu}(t+1), \,\cE_t^{\bar \lambda+1;\mu}(t),\, \cA_{t+1}(t+1),\, \Sigma^{\beta} \cU_k^\vee(t+1)\otimes \Sigma_\Sp^\eta \cS_{2(n-k)}\right\rangle,    
  \end{equation}
where the last group of terms runs over various pairs $(\beta,\eta)\in \YD_t^{2n-k-t}\times \YD_{n-k}^{a_t-1}$ satisfying~\(\beta < \lambda\). 
 
Therefore, it enough to check that the subcategory on the right side of~\eqref{inclusion_Sigma0} lies in~\(\cD_t\). First,~\(\cA_{t+1}(t+1)\subset \cD_t\) and~\(\cE_t^{\bar \lambda+1;\mu}(t) \in \tilde\cA_{t}(t) \subset \cD_t\) by definition. Next, note that~$l(\bar \lambda)=m$, so by Proposition~\ref{KPgeneralform} and the induction hypothesis, we have
\begin{equation*}
\cE_t^{\bar \lambda; \mu}(t+1) \in \left\langle \Sigma^\alpha \cU^{\vee}_k(t+1) \otimes \Sigma_\Sp^\eta \cS_{2(n-k)} \mid 0 \preceq(\alpha;\eta)\preceq (\bar\lambda,\mu) \right\rangle \subset \cD_t.
\end{equation*}
Lastly, since the set~\(\{\beta < \lambda\}\) is finite, by another induction we may assume that the last set in~\eqref{inclusion_Sigma0} also lies in~\(\cD_t\). Thus,~\(\Sigma^{\lambda}\cU_k^\vee(t+1)\otimes \Sigma_\Sp^\mu \cS_{2(n-k)} \in \cD_t\), and this proves the induction step.
\end{proof}
\begin{proposition}\label{cor:mutations_even t}
Let~\(1 \le t \le k-2\) be such that $k-t$ is even. Then
\begin{equation*}
\cA_t(t+1) \subset \cD_t.
\end{equation*}
\end{proposition}
\begin{proof}
Recall that
$\cA_{t}(t+1) = \langle \cE_t^{\lambda;\mu}(t+1) \mid \lambda \in \YD_{t}^{2n-k-t}; \mu\in \YD_{n-k}^{a_t}\rangle$.

If~$\mu_1=a_t$, then Proposition~\ref{proposition:incl_e} implies
\begin{equation*}
\cE_t^{\lambda;\mu}(t+1) \in \langle \cE_t^{\lambda+1;\mu}(t),\, \cA_{t+1}(t+1) \rangle\subset \cD_t,
\end{equation*}
since $\cE_t^{\lambda+1;\mu}(t)\in \tilde \cA_{t}(t)$ by definition.

If~$\mu_1<a_t$, then~\(\cE_t^{\lambda;\mu}(t+1) \in \tilde \cA_{t+1}(t+1)\), so the claim follows from Proposition~\ref{prop:mutations_even_t+1}.
\end{proof}

\section{Secondary staircase complexes} \label{section:secondary.staircase}
The goal of this section is to define secondary staircase complexes that allow us to prove, in Section~\ref{section:odd}, Lemma~\ref{lemma:fullness_induction_step} for~\(t\) such that~\(k-t = 2a_t+1\).
\subsection{Truncations of the generalized staircase complexes}
In this subsection, we introduce vector bundles that form secondary staircase complexes. To do this, we need to describe certain truncations of generalized staircase complexes.

Recall the notation from Section~\ref{section:generalized}. Let~\(\alpha \in \Y_1\times \Y_{l-1}\). Associated with $\alpha$, we have the integers~\(\mm(\alpha)\) and~$\vv(\alpha)$, characterized by~\eqref{eq:def-ti} and~\eqref{eq:def-vv}, 
and the explicit formula~\eqref{explicit_formula_i} for~$\alpha^{(i)}$ given by:
\begin{equation*} 
    \alpha^{(i)}\coloneqq (\alpha_2-1,\ldots, \alpha_{\dd(\alpha, i)}-1,\alpha_1-i+\vv(\alpha),\alpha_{\dd(\alpha, i)+1},\ldots,\alpha_l).   
\end{equation*}

Consider the generalized staircase complex~\eqref{eq:gsc} on~\(\Gr(l,V)\):
\begin{equation} \label{generalized_staircase1}
    0\to \Lambda^{b^{m}} V^\vee \otimes\Sigma^{\alpha^{(m)}}\cU_l^{\vee}\to \ldots\to \Lambda^{b^i} V^\vee \otimes \Sigma^{\alpha^{(i)}}\cU_l^{\vee}\to \ldots \to  \Lambda^{b^0} V^\vee \otimes \Sigma^{\alpha^{(0)}}\cU_l^{\vee}\to 0, \qquad m=\mm(\alpha),
\end{equation} 
whose rightmost term lies in degree~\(0\). Recall that~\(2n-k+1 \le \mm(\alpha) \le 2n\). 
We consider the truncation of the above complex at degree~$-(\alpha_1+\vv(\alpha))$.
\begin{definition} \label{def:K}
Let~\(\alpha \in \Y_1\times \Y_{l-1}\). 
The vector bundle~\(\cK_l^{\alpha}\) on~\(\Gr(l,V)\) is defined as the cohomology of the truncation of the complex~\eqref{generalized_staircase1} in degree $-(\alpha_1 + \vv(\alpha))$; that is, $\cK_l^{\alpha}$ is defined by the following exact sequence
\begin{equation}\label{eq:kernels}
	0 \to \cK_l^{\alpha} \to \Lambda^{b^{\alpha_1 + \vv(\alpha)}} V^\vee \otimes \Sigma^{\alpha^{(\alpha_1 + \vv(\alpha))}} \cU_l^{\vee} \to  \ldots  \to \Lambda^{b^0} V^\vee \otimes \Sigma^{\alpha^{(0)}}\cU_l^{\vee}\to 0.
\end{equation} 
\end{definition}
Furthermore, taking the other part of the complex~\eqref{generalized_staircase1}, we can represent~\(\cK^\alpha_l\) as follows:
\begin{equation}\label{eq:cokernels}
	0 \to \Lambda^{b^{m}} V^\vee \otimes \Sigma^{\alpha^{(m)}}\cU_l^{\vee} \to \ldots \to  \Lambda^{b^{\alpha_1 + \vv(\alpha) + 1}} V^\vee \otimes \Sigma^{\alpha^{(\alpha_1 + \vv(\alpha) + 1)}}\cU_l^\vee\to \cK_l^{\alpha}\to 0, \qquad m=\mm(\alpha).
\end{equation}
In particular, from the resolution~\eqref{eq:cokernels} we obtain
\begin{equation}\label{eq:K^{2n-k,alpha}=U(-1)}
    \text{if}\quad \alpha_1=2n-l \quad \text{and} \quad \bar \alpha \in \YD^{2n-l}_{l-1} \qquad \text{then} \qquad \cK_l^\alpha \simeq \Sigma^{\bar \alpha}\cU_l(-1).
\end{equation}
Moreover, if~\(\alpha_1+\vv(\alpha) \notin [0,\mm(\alpha)-1]\), then comparing~\eqref{generalized_staircase1} and~\eqref{eq:kernels}, we get~\(\cK^{\alpha}_l \simeq 0\).

The following property of the weights appearing in the resolution~\eqref{eq:kernels} will be used in Section~\ref{section:odd}. 
\begin{lemma}\label{lemma:alpha^i_in_YD}
    Let~\(\alpha \in \YD^w_1\times \YD^w_{l-1}\) and~\(0 \le i \le \alpha_1+\vv(\alpha)\). 
    Then
    \[\alpha^{(i)} \in \YD_l^{w}.\]
    If, additionally,~\(b^i \ge 1\) then~\(\alpha^{(i)} \in \YD_l^{w-1}\). 
\end{lemma}
\begin{proof}
    Recall from~\eqref{explicit_formula_i} that $\alpha^{(i)}$ is given by:
    \[
        \alpha^{(i)} =(\alpha_2-1,\ldots, \alpha_{\dd(\alpha, i)}-1,\alpha_1+\vv(\alpha)-i,\alpha_{\dd(\alpha, i)+1},\ldots,\alpha_l). 
    \]
    
First, observe that~\(\alpha^{(i)} \in \YD_l\) for all~$0\le i\le \alpha_1+\vv(\alpha)$.

If~\(\dd(\alpha,i)>1\), then~\(\alpha^{(i)} \in \YD_l^{\alpha_2-1}\subset \YD^{w-1}_l\).

If instead~\(\dd(\alpha,i)=1\), then~\(\dd(\alpha,0)=1\), since $i\ge 0$. This implies that \(\alpha \in \YD_l\), and hence~\(\vv(\alpha)=0\) by~\eqref{eq:v=0}. From the formula above, it follows that in this case~\(\alpha^{(i)} \in \YD_l^{\alpha_1-i} \subset \YD_l^{w}\). Moreover, by equation~\eqref{eq:b=i-v+d}, we have~\(i = b^i\), so~\(\alpha^{(i)} \in \YD_l^{w-b^i}\).

Therefore, since~\(i\ge 0\), we get~\(\alpha^{(i)} \in \YD_l^{w}\) in both cases. Furthermore, if~\(b^i\ge 1\) then~\(\alpha^{(i)} \in \YD_l^{w-1}\).
\end{proof}

We will need the following interpretation for the objects~\(\cK_l^\alpha\).
\begin{lemma} 
Let~\(\alpha \in \Y_1\times \Y_{l-1}\) and let~\(\psi_{l} \colon \Fl(1, l; V) \to \Gr(l, V)\) be the projection. We have:
\begin{equation} \label{K:descr}
    \cK_l^\alpha \simeq \psi_{l*}(\Lambda^{\alpha_1 + r_{\alpha}} \cQ_{2n - 1}^{\vee} \otimes S^{r_{\alpha}} \cU_1 \otimes \Sigma^{\bar \alpha} (\cU_l/\cU_1)^{\vee}), 
\end{equation} 
where~\(r_{\alpha} = |\{j \in [2,l] \mid 0 < \alpha_j\}|\).
\end{lemma}
\begin{proof}
Recall from the proof of Theorem~\ref{thm:gsc} that the complex~\eqref{generalized_staircase1} is obtained by applying the Fourier--Mukai transform with the kernel $\Sigma^{\bar \alpha} (\cU_l/\cU_1)^{\vee}$ to the twist of the Koszul complex~\eqref{eq:gsc-Koszul}. In particular, the object~$\cK_l^{\alpha}$, the cohomology of the truncation of~\eqref{generalized_staircase1} in 
degree~$-(\alpha_1 + \vv(\alpha))$, is isomorphic to the Fourier--Mukai transform with the kernel $\Sigma^{\bar \alpha} (\cU_l/\cU_1)^{\vee}$ applied to the truncation of the twisted Koszul complex, i.e. 
\begin{equation*}
\cK_l^{\alpha}\simeq \psi_{l*}\big(\Lambda^{b^{\alpha_1 + \vv(\alpha)}} \cQ_{2n - 1}^{\vee} \otimes S^{\alpha_1-b^{\alpha_1 + \vv(\alpha)}} \cU_1^{\vee} \otimes \Sigma^{\bar \alpha} (\cU_l/\cU_1)^{\vee}\big).
\end{equation*}
By the formula~\eqref{eq:def:b} we have $b^{\alpha_1 + \vv(\alpha)}=\alpha_1+r_{\alpha}$, so the lemma follows.
\end{proof}
Denote by
\begin{equation}\label{eq:C=K-dual}
    \cC_l^{\alpha}\coloneqq (\cK_l^{\alpha})^{\vee},
\end{equation}
the dual vector bundle on $\Gr(l,V)$.
\begin{lemma}\label{lemma:C=K(1)}
Let~\(\alpha \in \Y_l\) be such that~\(0 \le \alpha_1\) and~\(\alpha_1 - \alpha_l \le 2n - l\). Then
\begin{equation*}
\cC_l^{\alpha} \simeq \cK_l^{(2n-l-\alpha_1, -\alpha_l, \ \ldots\ ,-\alpha_2)}(H_l).
\end{equation*}
\end{lemma}
\begin{proof}
Denote 
\begin{equation*}
\s(\alpha)=(2n-l-\alpha_1, -\alpha_l,\ldots,-\alpha_2).
\end{equation*}
By assumptions $\alpha$ and $\s(\alpha)$ are $\GL_l$-dominant, thus $\vv(\alpha)=\vv(\s(\alpha))=0$. The dual complex to~\eqref{eq:kernels} has the form
\begin{equation}
    0 \to \Sigma^{\alpha^{(0)}}\cU_l \to  \ldots  \to \Lambda^{b^{\alpha_1}} V \otimes \Sigma^{\alpha^{(\alpha_1)}}\cU_l \to \cC_l^{\alpha} \to 0.
\end{equation}
The complex~\eqref{eq:cokernels} for~$\s(\alpha)$ has the form
\begin{equation}
    0 \to \Sigma^{\s(\alpha)^{(2n-l+1)}}\cU_l^{\vee} \to \ldots \to  \Lambda^{b^{2n-l-\alpha_1 + 1}} V^{\vee} \otimes \Sigma^{\s(\alpha)^{(2n-l-\alpha_1+ 1)}}\cU_l^\vee\to \cK_l^{\s(\alpha)}\to 0.
\end{equation}
All maps in both complexes are unique~\(\GL_l\)-equivariant and nonzero. Thus, to prove the statement, it is enough to show that 
\begin{equation}
  -\alpha^{(i)}=\s(\alpha)^{(2n-l+1-i)}+1 \qquad \text{and}\qquad  b^{i}(\alpha)=  2n-b^{2n-l+1-i}(\s(\alpha))
\end{equation}
for $i=0,\ldots, \alpha_1$. Indeed, using~\eqref{explicit_formula_i} and~\eqref{eq:b=i-v+d}, we deduce
\begin{align*}
 -&\alpha^{(i)}= (-\alpha_l, \ldots, -\alpha_{\dd(\alpha, i) + 1},  i -\alpha_1,1-\alpha_{\dd(\alpha, i)}, \ldots, 1-\alpha_2) =\s(\alpha)^{(2n-l+1-i)}+1,\\
 & b^{i}(\alpha) = \dd(\alpha, i) +i-1 =  2n-b^{2n-l+1-i}(\s(\alpha)). \qedhere
\end{align*}
\end{proof}

\subsection{Secondary staircase complexes} \label{subsection:double}
The restrictions of the vector bundles~$\cK_l^{\alpha}$ and $\cC_l^{\alpha}$ to the isotropic Grassmannian $\IGr(l,V)\subset \Gr(l,V)$ will be denoted by the same letters. Secondary staircase complexes are formed by vector bundles~$\cK_l^{\alpha}$ on~$\IGr(l,V)$ for~\(2 \le l\). In the case~\(\alpha \in \YD_2^{2n-k}\)
these complexes were defined in~\cite{Novikov}. 
\begin{proposition}[{\cite[Theorem~1.1]{Novikov}}]\label{prop:Cohomologies_of_H}
Let $(2n-2,b)\in \YD_2^{2n-2}$. There is the following complex on~\(\IGr(2, V)\):
\begin{equation}\label{eq:def:cH_2}
	\cH_2^{(2n-2,b)} \coloneqq \cK_2^{(2n-2-b,0)}\to \dots \to \cK_2^{(2n-2-1,b-1)} \to \cK_2^{(2n-2,b)},
\end{equation}
    where~\(\cK_2^{(2n-2-j,b-j)}\) lies in degree~\(b - j\).
    Its cohomologies are given by
    \begin{equation*}
\operatorname{H}^\bullet\left(\cH_2^{(2n-2,b)}\right) = 
    \begin{cases}
	\Lambda_\Sp^{b}\cS_{2(n-2)}(-H_2), &\text{if } 0 \le b \le n - 2;\\
        0, & \text{if } b = n - 1;\\
        \Lambda_\Sp^{2n - 2 - b}\cS_{2(n-2)}(-H_2)[-1], &\text{if } n \le b \le 2n - 2.
    \end{cases}
\end{equation*}
\end{proposition}
Consider the diagram of maps
\begin{equation*}
\xymatrix@1{
\IGr(2, V) &
\IFl(2, k; V) \ar[l]_{p_{k-2}} \ar[r]^{q_{k-2}} &
\IGr(k, V).
}
\end{equation*}
\begin{definition}
Let $\lambda \in \Y_{k-2}$. We define the functor $\Phi_{\lambda} \colon \Db(\IGr(2,V))\to \Db(\IGr(k,V))$ as the Fourier--Mukai transform with the kernel~$\Sigma^{\lambda} (\cU_k/\cU_{2})$ that takes an object $E \in \Db(\IGr(2,V))$~to
\begin{equation*}
   \Phi_{\lambda}(E)\coloneqq q_{k-2*}(\Sigma^{\lambda} (\cU_k/\cU_{2})\otimes p_{k-2}^*(E)).
\end{equation*} 
\end{definition}
In particular, note that~\(\Phi_{\lambda-a_{k-2}} = \hat \Phi_{\lambda}^{k-2}\).

Secondary staircase complexes on $\IGr(k,V)$ are defined for any weight~\(\alpha\in \YD^{2n-k}_1\times \YD^{2n-k}_{k-2}\) and will be denoted by $\cH_k^{(2n-k, \alpha)}$.
Consider~\(\beta = (2n-2,\, \alpha_1+k-2)\in \YD_2^{2n-2}\). The dual to the secondary staircase complex $\left(\cH_k^{(2n-k, \alpha)}\right)^{\vee}$ will be obtained as the twist by $\cO(1)$ of the Fourier--Mukai transform of the complex $\left(\cH_2^{\beta}(H_2)\right)^\vee$, more precisely we will show that
\begin{equation*}
      \left(\cH_k^{(2n-k,\, \alpha)}(1)\right)^\vee = \Phi_{\bar\alpha-1}\left(\big(\cH_2^{\beta}(H_2)\big)^\vee\right). 
  \end{equation*}
By~\eqref{eq:def:cH_2} and~\eqref{eq:C=K-dual} 
we have
\begin{equation} \label{mult12}
\left(\cH_2^{\beta}(H_2)\right)^\vee \simeq \cC_2^{\beta}(-H_2) \to \cC_2^{\beta-1}(-H_2) \to \ldots \to \cC_2^{\beta-\beta_2}(-H_2). 
\end{equation}
First, we compute the Fourier--Mukai functor applied to the terms of $\left(\cH_2^{\beta}(H_2)\right)^\vee$. 
\begin{lemma}\label{lemma:Phi(C_2)}
Let~\(\alpha = (\alpha_1; \alpha_2,\ldots, \alpha_{k-1}) \in \YD^{2n - k}_1 \times \YD^{2n - k}_{k - 2}\). For an integer~\(j\in [0,\alpha_1+k-2]\), there is an isomorphism
    \begin{equation}\label{eq:Phi(C_2)=C_k}
        \Phi_{\bar\alpha-1}\left(\cC_2^{(2n-2-j,\,\alpha_1+k-2-j)}(-H_2)\right) \simeq 
        \begin{cases}
            \cC_k^{(2n-k-b^i(\alpha),\,\alpha^{(i)})}(-H_k)[\dd(\alpha,i)- k+1],& \text{if}\ j = b^i(\alpha);\\
            0, &\text{otherwise},
        \end{cases}
    \end{equation}
where~\(i \in [0,\alpha_1+\vv(\alpha)]\).
\end{lemma}
\begin{proof}
Consider the following commutative diagram:
\begin{equation}\label{eq:flags}
    \begin{tikzcd}
		&\IFl(1,2,k;V) \ar[d, "\phi"] \ar[dr, "\phi_{k}"] \ar[dl, swap, "\phi_{2}"]\\
		\IFl(1,2;V) \ar[d, swap,"\psi_2"]  &  \IFl(2, k;V) \ar[dr, swap, "q_{k-2}"] \ar[dl,"p_{k-2}"]& \IFl(1,k; V) \ar[d,"q_{k-1}"]  &\\ 
		\IGr(2, V)&& \IGr(k,V)
    \end{tikzcd}
\end{equation}    
Denote
\[b=\alpha_1+k-2.\]
By Lemma~\ref{lemma:C=K(1)} we have~\(\cC_2^{(2n-2-j,b-j)} \simeq \cK_2^{(j,j-b)}(H_2)\). Next, from~\eqref{K:descr}, the base change, commutativity of~\eqref{eq:flags}, and projection formula, we get:
\begin{equation}
\begin{aligned} \label{mult2}
\Phi_{\bar\alpha-1}\left(\cK_2^{(j, j - b)}\right) 
&=(q_{k-2})_*\left(\Sigma^{\bar \alpha - 1} (\cU_k/\cU_{2}) \otimes p^*_{k-2} \cK_2^{(j, j - b)} \right) \simeq \\ 
&\simeq (q_{k-2})_*\left(\Sigma^{\bar \alpha - 1} (\cU_k/\cU_{2}) \otimes p^*_{k-2} \left(\psi_{2*}(S^{b - j} (\cU_2/\cU_1) \otimes \Lambda^{j} \cQ_{2n-1}^\vee)\right)\right)\simeq\\
&\simeq (q_{k-2})_*\left(\Sigma^{\bar \alpha - 1} (\cU_k/\cU_{2}) \otimes \phi_* \left(\phi_{2}^*(S^{b - j} (\cU_2/\cU_1) \otimes \Lambda^{j} \cQ_{2n-1}^\vee)\right)\right)\simeq \\
&\simeq (q_{k-2})_*\left(\phi_{*\phantom{k}}\big(\Sigma^{\bar \alpha - 1} (\cU_k/\cU_{2}) \otimes S^{b - j} (\cU_2/\cU_1) \otimes \Lambda^{j} \cQ_{2n-1}^\vee\big)\right)\simeq \\
&\simeq (q_{k-1})_*\left(\phi_{k*}\big(\Sigma^{\bar \alpha - 1} (\cU_k/\cU_{2}) \otimes S^{b - j} (\cU_2/\cU_1) \otimes \Lambda^{j} \cQ_{2n-1}^\vee\big)\right)\simeq \\ 
&\simeq (q_{k-1})_*\left(\phi_{k*}\big(\Sigma^{\bar \alpha - 1} (\cU_k/\cU_{2}) \otimes S^{b - j} (\cU_2/\cU_1)\big) \otimes \Lambda^{j} \cQ_{2n-1}^\vee\right).
\end{aligned}
\end{equation}
By Lemma~\ref{cor:BBW-12k} it follows that
\begin{equation} \label{direct_secondary}
\phi_{k*}\left(\Sigma^{\bar \alpha - 1} (\cU_k/\cU_{2}) \otimes S^{b- j}(\cU_2/\cU_1)\right) \simeq 
\begin{cases}
\Sigma^{\alpha^{(i)}} (\cU_k/\cU_{1})[\dd(\alpha,j)- k+1],& \text{if} \ j = b^i(\alpha);\\
0, &\text{otherwise}.
\end{cases}
\end{equation}
So, the second line of~\eqref{eq:Phi(C_2)=C_k} follows.

Let us show that for $j \in [0,\alpha_1 + k - 2]$, we have~
$\alpha^{(i)} \ge 0$ in the formula~\eqref{direct_secondary}. Since we assume~\(\alpha_2,\ldots, \alpha_{k-1}\ge 0\), the explicit formula~\eqref{explicit_formula_i} for $\alpha^{(i)}$ implies that it suffices to prove that~\(i \le \alpha_1+\vv(\alpha)\). 

We argue by contradiction, assume~$\alpha_1< i -\vv(\alpha)$. 
Then,~$\alpha_1+\vv(\alpha)-i < 0 \le \alpha_{k-1}$, so in this case~$\dd(\alpha,i) = k-1$ by~\eqref{eq:d-ineq}. 
Thus, applying~\eqref{eq:b=i-v+d}, we obtain~$j=b^i(\alpha) = i-\vv(\alpha)+k-2 > \alpha_1+k-2$, contradicting the assumption~\(j \in [0, \alpha_1 + k - 2]\).
Therefore,~\(0 \le \alpha_1 + \vv(\alpha)-i\) and~\(0\le \alpha^{(i)}\). 
In particular, we deduce the last statement~$0 \le i\le \alpha_1+\vv(\alpha)$. 

Finally, combining everything together for~\(j=b^i(\alpha)\in [0,\alpha_1+k-2]\) we obtain:
\begin{multline*}
\Phi_{\bar\alpha-1}\left(\cC_2^{(2n-2-j,\, \alpha_1+k-2-j)}(-H_2)\right) \simeq \Phi_{\bar\alpha-1}\left(\cK_2^{(j,\, j-(\alpha_1 + k - 2))}\right)\simeq \\
\simeq (q_{k-1})_*\left(\Sigma^{\alpha^{(i)}} (\cU_k/\cU_{1})[\dd(\alpha,i)- k+1] \otimes \Lambda^{j} \cQ_{2n-1}^\vee\right) \simeq \\
\simeq \cK_k^{(b^i(\alpha),\, -\alpha^{(i)})}[\dd(\alpha,i)- k+1]
\simeq \cC_k^{(2n-k - b^i(\alpha),\, \alpha^{(i)})}(-H_k)[\dd(\alpha,i)- k+1],
\end{multline*}
where the first and the last isomorphism follow from Lemma~\ref{lemma:C=K(1)}, the second one follows from~\eqref{mult2} and~\eqref{direct_secondary}, and the third isomorphism follows from the description~\eqref{K:descr} in the case~\(-\bar\alpha^{(i)}\le 0\).
\end{proof}

\begin{theorem}[Secondary staircase complexes]\label{theo:complexes_for_F^*,1^b}
For each~\(\alpha = (\alpha_1; \alpha_2,\ldots, \alpha_{k-1}) \in \YD^{2n - k}_1 \times \YD^{2n - k}_{k - 2}\) there is a complex of vector bundles on~\(\IGr(k, V)\):
\begin{equation}\label{eq:cH=k^bullet}
    \cH^{(2n-k,\, \alpha)}_k \coloneqq \cK_k^{(2n-k- b^{\delta(\alpha)}(\alpha),\, \alpha^{(\delta(\alpha))})} \to \ldots \to \cK_k^{(2n-k - b^{1}(\alpha),\, \alpha^{(1)})} \to   \cK_k^{(2n-k - b^0(\alpha),\, \alpha^{(0)})},
\end{equation}
where~\(\delta(\alpha)=\alpha_1+\vv(\alpha) \in \mathbb{Z}\). Its cohomology is given by 
\begin{equation}\label{eq:H(cH)}
\operatorname{H}^\bullet(\cH_k^{(2n-k,\, \alpha)})=
    \begin{cases}
        \left(\Phi_{\bar\alpha-1}\left(\Lambda^{\alpha_1 + k - 2}_\Sp \cS_{2(n-2)}\right)(H_k)\right)^\vee, & \text{if~\(0 \le \alpha_1 \le n - k\);}\\
        0, & \text{if~\(\alpha_1 = n - k + 1\);} \\
        \left(\Phi_{\bar\alpha-1}\left(\Lambda^{2n - k - \alpha_1}_\Sp \cS_{2(n-2)}\right)(H_k)\right)^\vee[-1],  & \text{if~\(n - k+2 \le \alpha_1 \le 2n - k\).}
    \end{cases}
\end{equation}
\end{theorem}
\begin{proof}
Let~\(\beta = (2n - 2, \alpha_1 + k - 2)\in \YD_2^{2n-2}\). Consider the complex: 
\begin{equation} \label{mult1}
(\cH_2^{\beta}(H_2))^\vee \simeq \cC_2^{\beta}(-H_2) \to \cC_2^{\beta-1}(-H_2) \to \ldots \to \cC_2^{\beta-\beta_2}(-H_2), 
\end{equation}
that is the dual complex to~\eqref{eq:def:cH_2} twisted by $\cO(-H_2)$. 
Lemma~\ref{lemma:Phi(C_2)} yields that the terms of the complex~\(\Phi_{\bar\alpha-1}\big((\cH_2^{\beta}(H_2))^\vee \big)\) are 
\begin{equation*}
    \text{\(\cC_k^{(2n-k-b^i(\alpha),\,\alpha^{(i)})}(-H_k)\)} \qquad \text{for} \quad 0 \le i \le \alpha_1+\vv(\alpha),
\end{equation*} 
and the term $\cC_k^{(2n-k-b^i(\alpha),\,\alpha^{(i)})}(-H_k)$ lies in degree
\begin{equation*}
    (b^i(\alpha)-(\alpha_1+k-2))-(\dd(\alpha,i)-k+1)=i-\alpha_1-\vv(\alpha),
\end{equation*}
where we use the equality~\eqref{eq:b=i-v+d}. Finally, we obtain the complex~\eqref{eq:cH=k^bullet} as the dual complex to~\(\Phi_{\bar\alpha-1}\left(\big(\cH_2^{\beta}(H_2)\big)^\vee \right)\), twisted by~$\cO(1)$, that is
\begin{equation} \label{def:H1}
    \cH^{(2n-k,\, \alpha)}_k \coloneqq \left(\Phi_{\bar\alpha-1}\left(\big(\cH_2^{\beta}(H_2)\big)^\vee\right)\right)^{\vee}(1).
\end{equation}

By Proposition~\ref{prop:Cohomologies_of_H}, over~\(\IGr(2, V)\) we have:
\begin{equation} 
 \left(\cH^{\beta}(H_2)\right)^\vee \simeq 
    \begin{cases}
	\Lambda_\Sp^{\alpha_1 + k - 2} \cS_{2(n-2)}, & \text{if} \ 0 \le \alpha_1 \le n - k;\\
	0, & \text{if} \ \alpha_1 = n-k+1;\\
	\Lambda_\Sp^{2n-\alpha_1-k} \cS_{2(n-2)}[1], & \text{if} \ n-k+2 \le \alpha_1\le 2n-k,
    \end{cases}   
\end{equation}
so applying~\eqref{def:H1} we deduce~\eqref{eq:H(cH)}.
\end{proof}
Recall that $\cA_{l} = \langle \Sigma^{\alpha}\cU^{\vee}_k\otimes \Sigma^{\mu}_{\Sp}\cS_{2(n-k)} \mid \alpha \in \YD_{l}^{2n-k-l};\, \mu\in \YD_{n-k}^{a_l}\rangle$. The following property of secondary staircase complexes plays an important role in the proof of Theorem~\ref{lemma:fullness_induction_step} when~$k-t$ is odd. 
\begin{lemma}\label{lemma:H_lies_in_blocks}
Let $l\in [0,k-2]$,~\(\alpha \in \YD^{2n-k}_1 \times \YD^{2n-k-l}_{l}\), and~\(\mu \in \YD^{a_l - 1}_{n - k}\). Then \begin{equation*}
\cH_k^{(2n-k,\, \alpha)}(1) \otimes \Sigma^{\mu}_\Sp \cS_{2(n-k)} \in \cA_l.
\end{equation*}
\end{lemma}
\begin{proof}
Theorem~\ref{theo:complexes_for_F^*,1^b} implies that it is enough to show:
\begin{equation*}
(q_{k-2})_*\left(\Sigma^{\bar \alpha - 1} (\cU_{k}/\cU_2) \otimes p^*_{k-2}(\Lambda^{r}_\Sp \cS_{2(n-2)}) \right) \otimes \Sigma^{\mu}_\Sp \cS_{2(n-k)} \in \cA^{\vee}_l \qquad \text{for}\quad r\in [0,n-2].
\end{equation*}
By Lemma~\ref{lemma:filtration_on_flags}, the vector bundle $\Sigma^{\bar \alpha - 1} (\cU_{k}/\cU_2) \otimes p^*_{k-2}(\Lambda^{r}_\Sp \cS_{2(n-2)})$
on $\IFl(2,k;V)$ admits a filtration whose irreducible $\bG$-equivariant associated subquotients are of the form:
\begin{equation*} 
  \Sigma^{\gamma} (\cU_k/\cU_{2}) \otimes \Lambda_\Sp^s \cS_{2(n - k)}, \quad  \text{where} \quad \text{\(\gamma \in \Y^{n-k-l}_{k-2}\) is such that $\gamma_{k-2}\ge -2$}  \quad \text{and} \quad s \in [0,n-k].
    \end{equation*}    
By Theorem~\ref{theo:relative_BBW} (Borel–Bott–Weil), we have $(q_{k-2})_*\big(\Sigma^{\gamma} (\cU_k/\cU_{2})\big)\in \cA_l^{\vee}$. Moreover, using our assumption that $\mu\in \YD^{a_l-1}_{n-k}$ and applying Lemma~\ref{lemma:LR rule}, we deduce: 
\begin{equation*}
\Lambda_\Sp^s \cS_{2(n - k)} \otimes \Sigma^{\mu}_\Sp \cS_{2(n-k)} \in \YD_{n-k}^{a_l}.
\end{equation*}
Combining this with the projection formula, we obtain:
\begin{equation*}
    (q_{k-2})_*\big(\Sigma^{\gamma} (\cU_k/\cU_{2}) \otimes \Lambda_\Sp^s \cS_{2(n - k)}\big) \otimes \Sigma^{\mu}_\Sp \cS_{2(n-k)}\simeq  (q_{k-2})_*\big(\Sigma^{\gamma} (\cU_k/\cU_{2})\big) \otimes \Lambda_\Sp^s \cS_{2(n - k)} \otimes \Sigma^{\mu}_\Sp \cS_{2(n-k)}  \in \cA^{\vee}_l. \qedhere
\end{equation*}
\end{proof}

\section{Relations generated by secondary staircase complexes} \label{section:odd}
The goal of this section is to prove Theorem~\ref{lemma:fullness_induction_step} in the case when $t$ has parity opposite to that of $k$; that is to show the inclusion
\begin{equation*}
\cA_t(t+1),\, \tilde \cA_{t + 1}(t+1) \subset \cD_t=\big\langle \cA_{t-1}(t),\, \tilde \cA_t(t),\, \cA_{t+1}(t+1) \big\rangle.
\end{equation*}
We begin by observing that when~\(k-t\) is odd, we have~\(a_{t+1}=\lfloor \frac{k-t-1}2\rfloor = a_t\).
This immediately implies the inclusion~\(\cA_t(t+1) \subset \tilde \cA_t(t+1)\), reducing our task to proving 
\[\tilde \cA_{t + 1}(t+1) \subset \cD_t.\]

Next, we note that, to prove the above inclusion, it suffices to establish that 
\begin{equation} \label{goal:K}
   \cK_k^{(2n-k- t,\, \beta)}(t + 1) \otimes \Sigma_\Sp^{\mu} \cS_{2(n-k)} \in \cD_t(t) \qquad \text{for all} \quad \beta \in \YD^{2n-k- t}_{t} \quad \text{and} \quad \mu \in \YD^{a_{t}}_{n - k}. 
\end{equation}
Indeed, using the resolution~\eqref{eq:kernels} for all~\(\cK_k^{(2n-k- t,\, \beta)}\), we easily deduce the inclusion~$\tilde \cA_{t + 1}(t+1) \subset \cD_t$ from~\eqref{goal:K}, see Corollary~\ref{corollary:odd}.

To prove the inclusion~\eqref{goal:K} we use secondary staircase complexes. Specifically, we begin by showing that for any weights~\(\alpha \in \YD^{2n-k}_1 \times \YD^{2n-k- t+1}_{t - 1}\) and~\(\mu \in \YD^{a_{t}}_{n - k}\), the corresponding tensored staircase complex satisfies
	\begin{equation*}
		\cH_k^{(2n-k,\, \alpha)}(t+1)\otimes \Sigma_\Sp^{\mu} \cS_{2(n-k)} \in \cD_t,
	\end{equation*}  
see Lemma~\ref{corollary:secondary_in_D}. Next, we observe that some terms of this complex already belong to~\(\cD_t\). The remaining terms form a subcomplex, which we denote by $\tau\big(\cH_k^{(2n-k,\, \alpha)}\big)$;
see Definition~\ref{def:tau}.
Then, using Lemma~\ref{corollary:secondary_in_D}, we show that the truncated complex also lies in $\cD_t$, that is,
\begin{equation} \label{tau-H_in}
   \tau\big(\cH_k^{(2n-k,\, \alpha)}\big)(t+1)\otimes \Sigma_\Sp^{\mu} \cS_{2(n-k)} \in \cD_t, 
\end{equation}
see Lemma~\ref{lemma:where}. 

Then we consider the following subcategory
\begin{equation*}
   \cR_t\coloneqq \left\langle \cK_k^{(2n-k- d,\, \beta)}\mid \beta \in \YD^{2n-k- t}_{t},\, d\in [1,t] \right\rangle \subset \Db(\IGr(k,V)). 
\end{equation*}
To prove the inclusion~\eqref{goal:K}, it is enough to show that the subcategory $\cR_t$ is generated by the following collection of truncated staircase complexes:
\begin{equation} \label{set-tau-terms}
  \cR_t=\left\langle \tau\big(\cH_k^{(2n-k,\, \alpha)}\big)\mid \alpha \in \YD^{2n-k-t+1}_1 \times \YD^{2n-k- t+1}_{t-1}, \, \alpha_1>0 \right\rangle. 
\end{equation}
Indeed, the equality~\eqref{set-tau-terms} together with~\eqref{tau-H_in} implies that
$$\cR_t(t+1)\otimes \Sigma_\Sp^{\mu} \cS_{2(n-k)} \subset \cD_t.$$ In particular, this proves the inclusion~\eqref{goal:K}.

To establish~\eqref{set-tau-terms}, we observe that the set of terms appearing in the truncated staircase complexes on the right-hand side of~\eqref{set-tau-terms} coincides with the generating set of $\cR_t$:
\begin{equation} \label{generating-set-K}
   \{\cK_k^{(2n-k- d,\, \beta)}\mid \beta \in \YD^{2n-k- t}_{t},\, d\in [1,t] \}. 
\end{equation}
This immediately implies the inclusion 
\begin{equation*}
\left\langle \tau\big(\cH_k^{(2n-k,\, \alpha)}\big)\mid \alpha \in \YD^{2n-k-t+1}_1 \times \YD^{2n-k- t+1}_{t-1}, \, \alpha_1>0 \right\rangle \subset \cR_t. 
\end{equation*}

Thus, to prove \eqref{set-tau-terms}, it remains to show that
\begin{equation} \label{intro:goal:Kin tau}
\cK_k^{(2n-k- d,\, \beta)}\in\left\langle \tau\big(\cH_k^{(2n-k,\, \alpha)}\big)\mid \alpha \in \YD^{2n-k-t+1}_1 \times \YD^{2n-k- t+1}_{t-1}, \, \alpha_1>0 \right\rangle 
\end{equation}
for all vector bundles $\cK_k^{(2n-k- d,\, \beta)}$ from \eqref{generating-set-K}. 

To achieve this, we introduce a special integer-valued function~$f$, defined on the generating set \eqref{generating-set-K}. We show that, when restricted to the terms of each complex~$\tau\big(\cH_k^{(2n-k,\, \alpha)}\big)$ from \eqref{intro:goal:Kin tau}, the function~\(f\) attains a unique minimum; see Lemma~\ref{lemma:minimal_f_in_H}. Then, using Corollary~\ref{corollary:lambda_beta}, we prove that each bundle~\(\cK_k^{(2n-k-d,\, \beta)}\) in~\eqref{generating-set-K} appears as a term in some truncated staircase complex~\(\tau(\cH_k^{{(2n-k,\, \alpha)}})\) from \eqref{intro:goal:Kin tau}, and that the minimum of~\(f\) on \(\tau(\cH_k^{{(2n-k,\, \alpha)}})\) is attained precisely at \(\cK_k^{(2n-k-d,\, \beta)}\). This observation allows us to establish~\eqref{intro:goal:Kin tau} by descending induction; see Proposition~\ref{prop:<H>=<K>1}. 
\begin{remark}\label{rem:Lagr_case}
	The same approach can be adapted to prove Theorem~\ref{lemma:fullness_induction_step} in the Lagrangian case for any~\(1 \le t \le k-1\). 
	More precisely, in the following results exactly the same arguments work for~\(k=n\), if we use the following conventions:~\(\YD^{a_t+1}_{0} = \YD^{a_t}_0=\{0\}\) and~\(\Sigma_\Sp^{0} \cQ_{n}^\vee /\cU_n \coloneq \cO\). This provides an alternative proof of Theorem~\ref{theo:intro} in the Lagrangian case that is different from~\cite{F19}. 
\end{remark}
\begin{lemma} \label{corollary:secondary_in_D}
	Let~\(1 \le t \le k-1\) be such that~\(k - t\) is odd. Let~\(\alpha \in \YD^{2n-k}_1 \times \YD^{2n-k- t+1}_{t - 1}\) and~\(\mu \in \YD^{a_{t}}_{n - k}\). 
	Then 
	\begin{equation*}
		\cH_k^{(2n-k,\, \alpha)}(t+1)\otimes \Sigma_\Sp^{\mu} \cS_{2(n-k)} \in \cD_t.
	\end{equation*}  
\end{lemma}
\begin{proof}
    If~\(k - t\) is odd then~$a_{t-1}-1=a_t=a_{t+1}$,
    so the claim follows from Lemma~\ref{lemma:H_lies_in_blocks} for~$l=t-1$.
\end{proof} 
\begin{definition} \label{def:tau}
    Let~\(1 \le t \le k-1\) and~\(\alpha \in \YD_{t}^{2n-k}\). 
    We define the complex $\tau\big(\cH_k^{(2n-k,\, \alpha)}\big)$ as the piece of the secondary staircase complex $\cH^{(2n-k,\, \alpha)}_k$ formed by $\cK_k^{(2n-k - b^i(\alpha),\, \alpha^{(i)})}$ satisfying~\(1 \le b^i \le t\). 
\end{definition}
\begin{example}\label{ex:tau-rectangle}
Let $0 \le w\le 2n-k-t$, and set~\(\alpha = (w+1)^t \in \YD_t\). Then, by~\eqref{eq:def-bbw}, we compute~\[\alpha^{(0)} = \alpha, \quad \alpha^{(1)} = (w)^t,\quad  \text{and} \quad \alpha^{(2)}= ((w)^{t-1},w-1).\] Applying~\eqref{eq:def:b}, we find \[b^0(\alpha)=0,\qquad  b^1(\alpha)=t,\qquad \text{and} \qquad  b^{i}(\alpha)>t\quad  \text{for all}\quad i\ge 2.\]
Therefore, the truncated staircase complex $ \tau(\cH_k^{(2n-k,\,\alpha)})$ consists of a single term: 
\begin{equation*}
     \tau(\cH_k^{(2n-k,\,\alpha)})=\cK_k^{(2n-k-d,\,(w)^t)}.
\end{equation*}
\end{example}
\begin{lemma} \label{lemma:where}
    Let~\(1 \le t \le k-1\) be such that~\(k - t\) is odd,~\(\alpha \in \YD^{2n-k-t+1}_1 \times \YD^{2n-k-t+1}_{t - 1}\), and~\(\mu \in \YD^{a_{t}}_{n - k}\). 
Then
    \begin{equation*}
	\tau\big(\cH_k^{(2n-k,\, \alpha)}\big)\otimes \Sigma_\Sp^{\mu} \cS_{2(n-k)} (t+1)\in\cD_t.
    \end{equation*}
\end{lemma}
\begin{proof}
    By definition~\eqref{eq:cH=k^bullet}, the complex~\(\cH_k^{(2n-k,\, \alpha)}\) consists of the terms
    \[
        \{\cK_k^{(2n-k-b^i(\alpha),\, \alpha^{(i)})}\}_{0 \le i \le \alpha_1+\vv(\alpha)}.
    \]
    We have~\(\alpha^{(i)} \in \YD_t^{2n-k-t+1}\) for all $0 \le i \le \alpha_1+\vv(\alpha)$ by Lemma~\ref{lemma:alpha^i_in_YD}. 
    
By Definition~\ref{def:tau} of~$\tau(\cH_k^{(2n-k,\, \alpha)})$ and Lemma~\ref{corollary:secondary_in_D}, it is enough to show that for all~$i$ satisfying~$b^i(\alpha)>t$ or~$b^i(\alpha)=0$, we have
    \begin{equation*}
	\cK_k^{(2n-k-b^i(\alpha),\, \alpha^{(i)})}(t+1)\otimes \Sigma_\Sp^{\mu} \cS_{2(n-k)}\in \cD_t=\langle \cA_{t-1}(t),\, \tilde \cA_t(t),\, \cA_{t+1}(t+1) \rangle. 
    \end{equation*}
    
Recall that
    \begin{equation*}
        \begin{aligned}
            \tilde \cA_{t}(t) &= \left\langle \Sigma^\alpha \cU^{\vee}_k(t)\phantom{{}+1} \otimes \Sigma_\Sp^\mu \cS_{2(n-k)} \mid \alpha \in \YD_{t}^{2n-k-t+1};\, \mu\in \YD_{n-k}^{a_t}\right\rangle,\\
            \cA_{t + 1}(t+1) &= \left\langle \Sigma^\alpha \cU^{\vee}_k(t+1) \otimes \Sigma_\Sp^\mu \cS_{2(n-k)} \mid \alpha \in \YD_{t+1}^{2n-k-t-1};\, \mu\in \YD_{n-k}^{a_t}\right\rangle.
        \end{aligned}    
    \end{equation*}
       
Suppose that~$b^i(\alpha)=0$. Then, by~\eqref{eq:K^{2n-k,alpha}=U(-1)} for~\(l=k\), we have the isomorphism~\[\cK_k^{(2n-k-b^i(\alpha),\, \alpha^{(i)})} \simeq \Sigma^{\alpha^{(i)}} \cU_k^{\vee}(-1),\] and hence we deduce
    \begin{equation*}
        \cK_k^{(2n-k-b^i(\alpha),\, \alpha^{(i)})}(t+1)\otimes \Sigma_\Sp^{\mu} \cS_{2(n-k)}\simeq \Sigma^{\alpha^{(i)}}\cU_k^{\vee} (t)\otimes 	\Sigma_\Sp^{\mu} \cS_{2(n-k)}\in \tilde \cA_{t}(t)\subset \cD_t \qquad \text{if} \quad b^{i}(\alpha)=0.
    \end{equation*}
    
   Now assume that~$b^i(\alpha)>t$, and set $\beta=(2n-k-b^i(\alpha),\, \alpha^{(i)})$. Consider the resolution~\eqref{eq:kernels} of~$\cK_k^{\beta}$; its terms have the form
    \begin{equation*}
        \Lambda^{b^j(\beta)}V\otimes \Sigma^{\beta^{(j)}}\cU^{\vee}_k \qquad \text{for} \quad j=0,\dots, \beta_1+\vv(\beta). 
    \end{equation*}
    Our goal is to show that~\(\beta^{(j)} \in \YD_{t+1}^{2n-k-t-1}\). 
    
    First, by Lemma~\ref{lemma:alpha^i_in_YD} for~\(w=2n-k-t+1\) we get~\(\bar \beta = \alpha^{(i)} \in \YD_{t}^{2n-k-t}\), since~\(b^i(\alpha)>t\ge 1\). 
    
    Next, since~\(\bar \beta \in \YD_{t}^{2n-k-t}\) and~\(\beta_1 = 2n-k-b^i(\alpha) \le 2n-k-t-1\) by assumptions on~\(\beta_1\) and~\(b^i\), we verify again by Lemma~\ref{lemma:alpha^i_in_YD} that~\(\beta^{(j)} \in \YD_{t+1}^{2n-k-t-1}\) for all~\(0 \le j \le \beta_1+\vv(\beta)\). 
    Thus,
    \begin{equation*}
        \cK_k^{\beta}(t+1)\otimes \Sigma_\Sp^{\mu} \cS_{2(n-k)}\in \cA_{t+1}(t+1)\subset \cD_t \qquad \text{if}\quad b^i(\alpha)>t. \qedhere
    \end{equation*}
\end{proof}
Consider now the subcategory
\begin{equation} \label{def:R_t}
   \cR_t\coloneqq \left\langle \cK_k^{(2n-k- d,\, \beta)}\mid \beta \in \YD^{2n-k- t}_{t},\, d\in [1,t] \right\rangle \subset \Db(\IGr(k,V)). 
\end{equation}
We aim to prove~\eqref{set-tau-terms}, namely, that the subcategory~$\cR_t$ is generated by a specific collection of truncated staircase complexes. To show this, we introduce the following definition. 
\begin{definition} \label{def:marked}
	A \textbf{marked weight} is a pair of a $\GL_t$-dominant weight and an index~\(d \in [1, t]\). For a marked weight~\((\beta, d)\), where $\beta=(\beta_1,\ldots,\beta_t)\in \Y_t$, we define an integer-valued function 
	\begin{equation}\label{eq:def_f}
		f\colon \Y_t\times \mathbb{Z}\to \mathbb{Z}, \qquad  f(\beta, d) \coloneqq d^2 + \sum_{i=1}^{t} 2i\beta_i.
	\end{equation}
\end{definition}
To each term~$\cK_k^{(2n-k-b^i(\alpha),\, \alpha^{(i)})}$ of the complex~$\tau(\cH^{(2n-k,\, \alpha)}_k)$ we assign the marked weight~$(\alpha^{(i)},b^i(\alpha))$. 
Thus, the function~$f$ is defined on the set of terms of~$\tau(\cH^{(2n-k,\, \alpha)}_k)$:
\begin{equation} \label{f:defined}
	f\big(\cK^{(2n-k-b^i(\alpha),\, \alpha^{(i)})}_t\big)\coloneqq f\big(\alpha^{(i)},b^i(\alpha)\big).
\end{equation}

The next lemma is proved in Appendix~\ref{subsection:appendix:proof}. 
\begin{lemma}\label{lemma:minimal_f_in_H}
	Let~\(\alpha \in \YD_1 \times \YD_{t-1}\). 
	Consider the set of marked weights
	\begin{equation} \label{alphaset}
		\{(\alpha^{(i)}, b^i(\alpha))\}_{1\le b^i(\alpha) \le t}   
	\end{equation} 
	corresponding to $\tau\big(\cH^{(2n-k,\, \alpha)}_k\big)$. 
	There exists a unique index~\(i_{\alpha}\in [1,t]\) such that~$b^{i_{\alpha}}(\alpha)=\dd(\alpha,i_{\alpha})$. 
	The strictly smallest value of the function $f$ on the set~\eqref{alphaset} is attained at $(\alpha^{(i_{\alpha})}, b^{i_\alpha}(\alpha))$. 
\end{lemma}

We are now ready to prove the key proposition.

\begin{proposition} \label{prop:<H>=<K>1}
Let~\(1 \le t \le k-1\). We have
\begin{equation*}
\cR_t=\left\langle \tau\big(\cH_k^{(2n-k,\, \alpha)}\big)\mid \alpha \in \YD^{2n-k-t+1}_1 \times \YD^{2n-k- t+1}_{t-1}, \, \alpha_1>0 \right\rangle.
\end{equation*}
\end{proposition}
\begin{proof}
By definition~\eqref{def:R_t}, the subcategory $\cR_t$ is generated by the vector bundles
\begin{equation} \label{set:K-d}
   \{\cK_k^{(2n-k- d,\, \beta)}\mid \beta \in \YD^{2n-k- t}_{t},\, d\in [1,t] \}.
   \end{equation}
The inclusion 
\begin{equation*}
\left\langle \tau\big(\cH_k^{(2n-k,\, \alpha)}\big)\mid \alpha \in \YD^{2n-k-t+1}_1 \times \YD^{2n-k- t+1}_{t-1}, \, \alpha_1>0 \right\rangle \subset \cR_t. 
\end{equation*}
follows immediately from Definition~\ref{def:tau} of $\tau\big(\cH_k^{(2n-k,\, \alpha)}\big)$ and the observation that~\(\alpha^{(i)} \in \YD_t^{2n-k-t}\) whenever~$1 \le b^i \le t$ by Lemma~\ref{lemma:alpha^i_in_YD} for~\(w=2n-k-t+1\). 

Thus, to prove the statement, it suffices to show that for all weights~$\beta \in \YD^{2n-k- t}_{t}$ and for all integers~$d\in [1,t]$, we have
\begin{equation} \label{goal:set:K-d}
   \cK_k^{(2n-k- d,\, \beta)} \in \left\langle \tau\big(\cH_k^{(2n-k,\, \alpha)}\big)\mid \alpha \in \YD^{2n-k-t+1}_1 \times \YD^{2n-k- t+1}_{t-1}, \, \alpha_1>0 \right\rangle.
   \end{equation}
   
The function $f$ is defined on the set~\eqref{set:K-d} by the rule
$$f(\cK^{(2n-k-d,\,\beta)}_t)\coloneqq f(\beta,d)= d^2 + \sum_{i=1}^{t} 2i\beta_i.$$
Since the set~\(\YD^{2n-k- t}_{t}\times [1,t]\) is finite, the function~$f$ is bounded both above and below on~\eqref{set:K-d}. We prove~\eqref{goal:set:K-d} by decreasing induction on~$f(\beta,d)$.

For the base case, consider the maximum value of $f$, which corresponds to
$$\beta=(2n-k-t)^t\in \YD_t, \quad\text{and}\quad d=t.$$
Set $\alpha=(2n-k-t+1)^t\in \YD_t$.
By Example~\ref{ex:tau-rectangle}, for this choice of $\alpha$ we obtain: 
\begin{equation*}
\cK_t^{(2n-k-t,\,\alpha^{(1)})} (t + 1) \otimes \Sigma_\Sp^{\mu} \cS_{2(n-k)} = \tau\big(\cH_t^{(2n-k,\, \alpha)}\big)(t + 1) \otimes \Sigma_\Sp^{\mu} \cS_{2(n-k)},
\end{equation*}
so that the base of induction follows.

For the induction step, assume that for some~$z\in \mathbb{Z}$, the statement~\eqref{goal:set:K-d} holds for all marked weights~$(\beta,d)\in \YD^{2n-k- t}_{t}\times [1,t]$ such that~\(f(\beta,d) > z\).

Suppose that a marked weight $(\lambda,d)\in \YD^{2n-k- t}_{t}\times [1,t]$ satisfies $$f(\lambda, d)=z.$$ By Corollary~\ref{corollary:lambda_beta}, there exists~\(\varrho \in \YD^{2n-k- t+1}_1 \times \YD_{t - 1}^{2n-k- t+1}\) and an index $i\in \mathbb{Z}$ such that
\[\varrho_1\ge 1,\quad \varrho^{(i)} = \lambda, \quad \text{and}\quad b^{i}(\varrho)=\dd(\varrho, i) = d.\] 
Hence, the vector bundle~$\cK^{(2n-k- d,\, \lambda)}$ appears as a term in the complex $\tau(\cH^{(2n-k,\, \varrho)})$. 

Now, by Lemma~\ref{lemma:minimal_f_in_H}, we know that $i=i_{\varrho}$ is the index at which the strictly minimal value of $f$ is attained on~$\tau(\cH^{(2n-k,\, \varrho)})$. Thus,
\begin{equation*}
\cK^{(2n-k-d,\, \gamma)}_t \in \left\langle \tau(\cH^{(2n-k,\, \varrho)}),\, \cK_t^{(2n-k-b^i(\varrho),\, \varrho^{(i)})} \right\rangle, \quad \text{where} \quad f(\varrho^{(i)},b^i(\varrho))> z.
\end{equation*} 
Thus, the induction assumption implies the induction step, and the statement of the proposition follows.

Finally, note that the argument above also shows that the set of terms appearing in the truncated staircase complexes on the right-hand side of~\eqref{goal:set:K-d} coincides precisely with the generating set of $\cR_t$.
\end{proof}
\begin{corollary}  \label{prop:<H>=<K>}
Let~$1 \le t \le k-1$ be such that~\(k=t\) is odd, and let $\mu\in \YD^{a_t}_{n-k}$. For~$\beta \in \YD^{2n-k- t}_{t}$ and~$1 \le d \le t$ we have 
\[\cK^{(2n-k-t,\,\beta)} (t + 1) \otimes \Sigma_\Sp^{\mu} \cS_{2(n-k)} \in \cD_t.\]
\end{corollary}
\begin{proof}
By definition~\eqref{def:R_t} of the subcategory $\cR_t$, we have
\begin{equation*}
  \cK^{(2n-k-t,\,\beta)} (t + 1) \otimes \Sigma_\Sp^{\mu} \cS_{2(n-k)}\in \cR_t(t+1)\otimes \Sigma_\Sp^{\mu} \cS_{2(n-k)}.
\end{equation*}
At the same time, Lemma~\ref{lemma:where}, combined with Proposition~\ref{prop:<H>=<K>1}, implies the inclusion
\begin{equation*}
\cR_t(t+1)\otimes \Sigma_\Sp^{\mu} \cS_{2(n-k)} \subset \cD_t,
\end{equation*}
and hence the statement follows. 
\end{proof}

\begin{corollary} \label{corollary:odd}
Let~$1 \le t \le k-1$ such that~$k-t=2a_t+1$. Then 
\begin{equation*}
    \cA_t(t+1)\subset \tilde \cA_{t + 1}(t+1) \subset \cD_t. 
\end{equation*}
\end{corollary}
\begin{proof}
Recall that
\begin{align*}
\cA_{t}(t+1) &= \langle \Sigma^{\alpha}\cU^{\vee}_k(t+1)\otimes \Sigma^{\mu}_{\Sp}\cS_{2(n-k)} \mid \alpha \in \YD_{t}^{2n-k-t};\, \mu\in \YD_{n-k}^{a_t}\rangle;\\
\tilde \cA_{t+1}(t+1) &= \langle \Sigma^{\alpha}\cU^{\vee}_k(t+1)\otimes \Sigma^{\mu}_{\Sp}\cS_{2(n-k)} \mid \alpha \in \YD_{t+1}^{2n-k-t};\, \mu\in \YD_{n-k}^{a_{t+1}}\rangle.
\end{align*}
If $k-t$ is odd, then $a_t=a_{t+1}$, and hence the inclusion $\cA_t(t+1) \subset \tilde \cA_{t + 1}(t+1)$ follows. 

Therefore, to prove the statement, it suffices to show that for any Young diagrams~\(\beta \in \YD^{2n-k- t}_{t}\) and~\(\mu \in \YD^{a_{t}}_{n - k}\),
\begin{equation} \label{last-goal}
    \Sigma^{(2n-k-t,\, \beta)}\cU_k^{\vee}(t + 1) \otimes \Sigma_\Sp^{\mu} \cS_{2(n-k)} \in \cD_t. 
\end{equation}

First, Corollary~\ref{prop:<H>=<K>} yields that 
\begin{equation*}
\cK_t^{(2n-k-t,\, \beta)}(t + 1) \otimes \Sigma_\Sp^{\mu} \cS_{2(n-k)} \in \cD_t.
\end{equation*} 
Next, consider the resolution~\eqref{eq:kernels} of~\(\cK_t^{(2n-k-t,\, \beta)}\), twisted by $\Sigma_\Sp^{\mu} \cS_{2(n-k)}(t+1)$:
\begin{multline} \label{resolution}
    0 \to \cK_t^{(2n-k-t,\, \beta)}(t + 1) \otimes \Sigma_\Sp^{\mu} \cS_{2(n-k)} \to \Lambda^{b^{2n-k-t}} V^\vee \otimes \Sigma^{(2n-k-t,\, \beta)^{(2n-k- t)}} \cU_k^{\vee} (t + 1) \otimes \Sigma_\Sp^{\mu} \cS_{2(n-k)} \to  \\\to \ldots \to  \Sigma^{(2n-k-t,\, \beta)}\cU_k^{\vee}(t + 1) \otimes \Sigma_\Sp^{\mu} \cS_{2(n-k)}\to 0.
\end{multline}  
Observe that for all~\(0 < i \le 2n-k-t\), the terms $$\Sigma^{(2n-k-t,\, \beta)^{(i)}} \cU_k^{\vee} (t+1) \otimes \Sigma_\Sp^{\mu} \cS_{2(n-k)}$$
belong to the subcategory $\cA_{t+1}(t+1)$. Thus, the resolution~\eqref{resolution} and Corollary~\ref{prop:<H>=<K>} imply~\eqref{last-goal}.
\end{proof}

\appendix
\section{Some computations} \label{appendix}
\subsection{Borel--Bott--Weil computations} 
Here we collect the concrete computations used in the article, all of which are consequences of the Borel–Bott–Weil theorems.
\begin{lemma} \label{lemma:exceptional_for_big_t}
    Let~\(\alpha, \beta \in \YD^{2n - 2k + 1}_{k}\). Then~\(\E{>0}(\Sigma^{\alpha}\cU^\vee, \, \Sigma^{\beta}\cU^\vee) = 0\) over~\(\IGr(k, V)\).
\end{lemma}
\begin{proof}
Let~\(\Sigma^\gamma \cU^\vee \) be a direct summand of~\(\Sigma^\alpha \cU \otimes \Sigma^\beta \cU^\vee\).  It suffices to show that~\(\operatorname{H}^{>0} (\IGr(k, V),\, \Sigma^\gamma \cU^\vee) = 0\). We compute this by Theorem~\ref{theo:relative_BBW_sym} (symplectic Borel--Bott--Weil) for~\(\cV = V\). 

First, by Lemma~\ref{lemma:LR rule}, we have 
\begin{equation*}
\gamma_k \in [-\alpha_1, \beta_k] \subseteq [-2n+2k-1, 2n-2k+1].
\end{equation*}

Next, if~\(\gamma_k \ge 0\), then~\(\gamma \ge 0\) and, thus,
\begin{equation*}
\operatorname{H}^\bullet (\IGr(k, V),\, \Sigma^\gamma \cU^\vee) = \operatorname{H}^0 (\IGr(k, V),\, \Sigma^\gamma \cU^\vee),
\end{equation*} 
while higher cohomologies vanish.

In the remaining case~\( \gamma_k <0\), we observe that 
\begin{equation*}
|\gamma_k+\rho_k| \in \{0, \dots, n-k+1\} = \{\rho_{k+1},\dots,\rho_n,0\}.
\end{equation*} 
This implies that the weight $\gamma+\rho$ has at least two entries with equal absolute values. Therefore, by Theorem~\ref{theo:relative_BBW_sym}, in this case all cohomology vanishes.
\end{proof}
\begin{lemma} \label{lemma:BBW_tilde_mu}
Consider the projection
\begin{equation*}
p_t\colon \IFl(k-t,k;V)\to \IGr(k-t,V).
\end{equation*}
Let~\(\gamma \in \Y_t\) and~\(\eta \in \YD^{a}_{n-k}\) be such that~\(\gamma_t \ge -2(n-k) -a - 1\). 
Suppose that for some~\(l\in \mathbb{Z}_{\ge 0}\) we have 
\begin{equation*}
R^lp_{t*}\big(\Sigma^\gamma (\cU_k/\cU_{k-t})^{\vee} \otimes \Sigma^\eta_\Sp\cS_{2(n-k)}\big)=\Sigma^{\hat \mu}_\Sp \cS_{2(n-k+t)},
\end{equation*} 
where~\(\hat \mu \in \YD^{a}_{n-k+t}\) satisfies~\(\hat \mu_1 = \dots = \hat \mu_t = a \). Then necessarily:
\begin{itemize}
    \item \(\gamma = (a)^t\), i.e., all entries of $\gamma$ equal $a$,
    \item \(\hat \mu = (\gamma, \eta)\), i.e., $\hat \mu$ is the concatenation of $\gamma$ and $\eta$,
    \item $l=0$, so the pushforward is concentrated in degree zero.
\end{itemize} 
\end{lemma}
\begin{proof}
Note that $\IFl(k-t,k;V)$ can be viewed as the relative isotropic Grassmannian
\begin{equation*}
\IFl(k-t,k;V)=\IGr_{\IGr(k-t,V)}(t,\cS_{2(n-k + t)}).
\end{equation*}
Thus, to compute the derived pushforward $p_{t*}$, we apply Theorem~\ref{theo:relative_BBW_sym} (symplectic Borel--Bott--Weil) to the symplectic vector bundle~$\cS_{2(n-k+t)}$.

Denote~$\beta=(\gamma;\eta)+\rho$. By Theorem~\ref{theo:relative_BBW_sym}, if the pushforward is the irreducible bundle~$\Sigma^{\hat \mu}_\Sp \cS_{2(n-k+t)}$ with~$\hat \mu_1 = \dots = \hat \mu_t = a$, then the set~$\{|\beta_i|\}$ must contain the values 
\begin{equation}\label{eq:values}
 a+n-k+1, \ldots, a+n-k + t.
\end{equation} 

On the other hand, the assumptions on $\eta$ imply
\begin{equation*}
  0\le \beta_{t+i}=\eta_i + \rho_{t + i} \le n-k + a\qquad \text{for} \quad 1\le i\le n-k.
\end{equation*}
Furthermore, since~\(\gamma_t \ge -2(n-k) -a - 1\), we get
\begin{equation*}
\beta_t=\gamma_t + \rho_t \ge -n+k-a.
\end{equation*}
Thus, the values~\eqref{eq:values} appear among $\{|\beta_i|\}$ only if
\begin{equation*}
\beta_i = \gamma_i+\rho_i = a+ (n-k+t+1-i) \qquad \text{for}\quad 1 \le i \le t.
\end{equation*} 
This condition uniquely determines~\(\gamma = (a)^t\). In this case, the weight~\((\gamma, \eta)\) is already~\(\Sp_{2(n-k+t)}\)-dominant, and by Theorem~\ref{theo:relative_BBW_sym} the pushforward is concentrated in degree~\(0\):
\begin{equation*}
R^0 p_{t*}\big(\Sigma^\gamma (\cU_k/\cU_{k-t})^{\vee} \otimes \Sigma^\eta_\Sp\cS_{2(n-k)}\big) = \Sigma^{(\gamma, \eta)}_\Sp \cS_{2(n-k+t)}.\qedhere
\end{equation*}
\end{proof}

\begin{lemma} \label{cor:BBW-12k}
Consider the projection $\phi_{k}\colon \IFl(1,2, k; V)\to \IFl(1,k,V)$. For a~\(\GL_1 \times \GL_{k - 2}\)-dominant weight~\(\alpha=(\alpha_1;\alpha_2,\ldots,\alpha_{k-1})\) we have
\begin{equation*}
\phi_{k*} \left(S^{\alpha_1 +k -2 - j} (\cU_2/\cU_1)\otimes \Sigma^{\bar \alpha-1}(\cU_k/\cU_2)\right) =
    \begin{cases}
	\Sigma^{\alpha^{(i)}}(\cU_{k}/\cU_1)[\dd(\alpha,i)-k+1], &\text{if~\(j= b^i(\alpha)\)}; \\
		0, &\text{otherwise}.
	\end{cases}
\end{equation*}
\end{lemma}
\begin{proof}
Note that $\IFl(1,2, k; V)$ is the relative projective bundle:
\begin{equation*}
\IFl(1,2, k; V)=\PP_{\IFl(1,k;V)}(\cU_k/\cU_1),
\end{equation*}
and the projection $\phi_k$ corresponds to this projective fibration.

To compute the pushforward, we apply Theorem~\ref{theo:relative_BBW} (Borel--Bott--Weil) to the vector bundle
\begin{equation*}
S^{\alpha_1 +k -2 - j} (\cU_2/\cU_1)\otimes \Sigma^{\bar \alpha-1}(\cU_k/\cU_2),
\end{equation*}
which corresponds to the weight:
\begin{equation*}
\beta = (j-\alpha_1-k+2;1-\alpha_{k-1}, \ldots, 1-\alpha_{2}) \in \Y_1 \times \Y_{k-2}.
\end{equation*}

Recall the notation introduced in Section~\ref{section:generalized}. 
By~\eqref{eq:def-bv}, definition of the set~$\bV(\bar \alpha)$, the pushforward vanishes when~\(\beta_1 \in \bV(\bar \beta)\). 
A direct computation shows that this condition is equivalent to~$\alpha_1-j \in \bV(\bar \alpha)$. 
Furthermore, by~\eqref{def:zi} and~\eqref{def:foralli}, we have~$\alpha_1-j \notin \bV(\bar \alpha)$ if and only if~$j=\alpha_1-z_i=b^i(\alpha)$. For $j=b^i(\alpha)$, we obtain
\begin{equation*}
(b^i(\alpha)-\alpha_1-k+2,1-\alpha_{k-1}, \ldots, 1-\alpha_{2})+\rho=(1-z_i,k-1-\alpha_{k-1}, \ldots, 2-\alpha_{2}),
\end{equation*}
so $\BBW_{1-\bar \alpha}(b^i(\alpha)-\alpha_1-k+2)=-\alpha^{(i)}$. The length of the associated permutation is $k-1-\dd(\alpha,i)$. 
\end{proof}
\begin{lemma}\label{BBW-SC}
Let $q_t\colon \IGr(k-t, k;V)\to \IGr(k,V)$ be the natural projection, and let $m\in \mathbb{Z}_{\ge 0}$. Suppose that~$\gamma \in \Y_t$ satisfies~$\gamma_t\ge -k+t+1$. Denote $\alpha=(-m,-\gamma)\in \mathbb{Z}^{t+1}$. 
Assume that 
\begin{equation*}
R^{\bullet}q_{t*}\left(S^m\cU_{k-t} \otimes \Sigma^{\gamma} (\cU_k/\cU_{k - t})\right)\neq 0.
\end{equation*} 
Then we have
\begin{equation*}
R^{\bullet}q_{t*}\left(S^m\cU_{k-t} \otimes \Sigma^{\gamma} (\cU_k/\cU_{k - t})\right) = 
		\Sigma^{-\alpha^{(0)}} \cU_k[-\dd(\alpha,0)+1],
\end{equation*}
where $-\alpha^{(0)}\in \YD_{t+1}^{\max(\gamma_1, m-t)}$. 
\end{lemma}
\begin{proof}
Note that the isotropic flag variety can be viewed as a relative Grassmannian:
\begin{equation*}
\IFl(k-t,k;V)=\Gr_{\IGr(k,V)}(k-t,\cU_k).
\end{equation*}
To compute the pushforward via~$q_t$, we apply Theorem~\ref{theo:relative_BBW} (Borel--Bott--Weil) to the tautological vector bundle~$\cU_k$.

We consider the weight:
\begin{equation*}
((0)^{k-t-1},\alpha) +\rho 
=(k,k-1,\ldots, t+2, t+1-m,t-\gamma_t, \ldots, 1-\gamma_1).
\end{equation*}
Under the assumptions $\gamma_t\ge -k+t+1$ and $m\in \mathbb{Z}_{\ge 0}$, we obtain 
\begin{equation*}
 t-\gamma_t\le k-1 \qquad \text{and} \qquad t+1-m<t+2.
\end{equation*}
Hence, if 
\begin{equation*}
R^{\bullet}q_{t*}\left(S^m\cU_{k-t} \otimes \Sigma^{\gamma} (\cU_k/\cU_{k - t})\right)\neq 0,
\end{equation*} 
then Theorem~\ref{theo:relative_BBW} (Borel–Bott–Weil) implies that
\begin{equation*}
\sigma((0,\ldots,0,\alpha)+\rho)-\rho= ((0)^{k-t-1},\beta),
\end{equation*}
for some $\beta\in \YD_{t+1}$,  where $\sigma$ is the unique permutation such that the sequence $\sigma((0,\ldots,0,\alpha)+\rho)$ is strictly decreasing. 

It remains to verify that~$\beta=\alpha^{(0)}$ and~$\ell(\sigma)=\dd(\alpha,0)-1$. 
The results of Section~\ref{section:generalized} show that if the pushforward is nonzero, then~$-m\notin \bV(-\bar \gamma)$, so~$z_0=-m$. Therefore, by definition~\eqref{eq:def-bbw}, we obtain
\begin{equation*}
\beta= \BBW_{-\gamma}(-m)=\alpha^{(0)} \qquad \text{and} \qquad \ell(\sigma)=\dd(\alpha,0)-1,
\end{equation*}
where~\(\beta_{t+1} = \min (\alpha_t, t-m)\). 
\end{proof}

\subsection{Filtrations}
Here, we establish results concerning filtrations of $\bG$-equivariant vector bundles on the isotropic flag variety~$\IFl(k - t, k; V)$.

Recall the notation~$((a)^t,\mu)=(\underbrace{a,\ldots,a}_t,\mu)$. 
\begin{lemma}\label{lemma:filtration_on_flags1}
    Let~\(1 \le t \le k-1\),~\(\lambda \in \Y_t\),~\(\mu \in \YD^a_{n-k}\),
    and let~\(\theta = ((a)^{t}, \mu)\), so that $\theta \in \YD^{a}_{n - k + t}$. 
    The vector bundle
    \begin{equation*}
\Sigma^{\lambda} (\cU_k/\cU_{k - t}) \otimes p_t^*\Sigma_\Sp^{\theta} \cS_{2(n - k + t)}
\end{equation*} 
    on $\IFl(k-t,k;V)$ admits a filtration whose associated subquotients are:
\begin{itemize}
    \item the term $\Sigma^{\lambda+a}(\cU_k/\cU_{k-t})\otimes \Sigma_\Sp^\mu \cS_{2(n-k)}$
of multiplicity exactly one;
\item terms of the form~\(\Sigma^{\gamma} (\cU_k/\cU_{k-t})\otimes \Sigma_\Sp^\eta \cS_{2(n-k)}\), where~\(\gamma \in \Y_t\) satisfies~\(\gamma < \lambda+a\), and~\(\eta\in \YD_{n-k}^{a}\).
\end{itemize}
\end{lemma}
\begin{proof}
By Lemma~\ref{lemma:filtration_on_flags}, the vector bundle 
\begin{equation*}
\Sigma^{\lambda} (\cU_k/\cU_{k - t}) \otimes p_t^*\Sigma_\Sp^{\theta} \cS_{2(n - k + t)}
\end{equation*} admits a filtration with associated subquotients of the form
\begin{equation*}
    \Sigma^{\gamma} (\cU_k/\cU_{k-t})\otimes \Sigma_\Sp^\eta \cS_{2(n-k)},
\end{equation*}
where~\(\gamma \in \Y_t\) satisfies~\(\gamma-a \le \gamma \le \lambda+a\), and~\(\eta\in \YD_{n-k}^{a}\).

Thus, it remains to show that if~\(\gamma = \lambda+a\) then~\(\eta = \mu\), and this subquotient appears with multiplicity one.
Recall from the proof of Lemma~\ref{lemma:filtration_on_flags} that the associated subquotients of the bundle \begin{equation*}
\Sigma^{\lambda} (\cU_k/\cU_{k - t}) \otimes p_t^*\Sigma_\Sp^{\theta} \cS_{2(n - k + t)}
\end{equation*} 
are of the form
\begin{equation}\label{eq:filtr_flags1}
 \Sigma^{\lambda} (\cU_k/\cU_{k - t})\otimes \Sigma^{\alpha}(\cU_{k}/\cU_{k-t})\otimes \Sigma^{\beta}(\cU_{k}/\cU_{k-t})^{\vee}\otimes \Sigma^{\nu}_{\Sp}\cS_{2(n-k)},   
\end{equation}
where $\alpha,\beta,\nu\subset \theta$.

Now, for such a summand to contain
\begin{equation*}
\Sigma^{\lambda+a} (\cU_k/\cU_{k-t})\otimes \Sigma_\Sp^\eta \cS_{2(n-k)},
\end{equation*}
we must have 
\begin{equation*}
\Sigma^{\lambda+a} (\cU_k/\cU_{k-t}) \inplus \Sigma^{\lambda} (\cU_k/\cU_{k - t})\otimes \Sigma^{\alpha}(\cU_{k}/\cU_{k-t})\otimes \Sigma^{-\beta}(\cU_{k}/\cU_{k-t}).
\end{equation*} 
By Lemma~\ref{lemma:LR rule}, this inclusion is possible only if~\(\alpha = (a)^t\) and~\(\beta = 0\). Furthermore, by Lemma~\ref{skew_decomposition}, we have~\(\nu = \theta/\alpha = \mu\), and Lemma~\ref{lemma:PRV} gives~\(c^{\theta}_{\alpha,\mu} = 1\). 
\end{proof}

\begin{corollary} \label{symp_coh}
    Let~\(\alpha \in \YD_{k-t}\),~\(\xi \in \Y_t\),~\(\zeta \in \YD_{n-k}\) and $\nu\in \YD_{n-k+t}$. Consider the projection~$p_t\colon \mathrm{IFl}(k-t,k;V)\to \IGr(k-t,V)$. Then we have 
    \begin{align*}
        \mathrm{H}^{\bullet}&\left(\mathrm{IFl}(k-t,k;V),\,  \Sigma^{\xi}(\cU_{k}/\cU_{k-t}) \otimes \Sigma^{\zeta}_{\Sp}\cS_{2(n-k)}\otimes p^*_{t} \Sigma_\Sp^{\nu} \cS_{2(n-k + t)} \otimes \Sigma^{\alpha} \cU^{\vee}_{k}\right)^{\bG} = \\ 
        = \mathrm{H}^{\bullet}&\left(\mathrm{IFl}(k-t,k;V),\, \Sigma^{\xi}(\cU_{k}/\cU_{k-t}) \otimes \Sigma^{\zeta}_{\Sp}\cS_{2(n-k)}\otimes p^*_{t} \Sigma_\Sp^{\nu} \cS_{2(n-k + t)} \otimes \Sigma^{\alpha} (\cU_{k}/\cU_{k-t})^{\vee} \right)^{\bG}.
    \end{align*}
\end{corollary}
\begin{proof}
By Lemma~\ref{skew_decomposition} irreducible subquotients of~\(\Sigma^\alpha \cU_k^\vee\) have the form
\begin{equation} \label{decompU_k1}
  \Sigma^{\lambda} \cU_{k-t}^{\vee} \otimes \Sigma^{\eta} (\cU_k/\cU_{k-t})^{\vee} \qquad \text{for} \quad \lambda\in \YD_{k-t} \quad \text{and} \quad \eta \subseteq \alpha/\lambda.  
\end{equation}
By Lemma~\ref{lemma:filtration_on_flags} the vector bundle $\Sigma^{\xi}(\cU_{k}/\cU_{k-t}) \otimes p^*_{t} \Sigma_\Sp^{\nu} \cS_{2(n-k + t)}$ has a filtration with the subquotients of the form
\begin{equation*}
\Sigma^{\gamma} (\cU_k/\cU_{k - t}) \otimes \Sigma_\Sp^\phi \cS_{2(n - k)}.
\end{equation*}
Thus, irreducible subquotients of the vector bundle~$\Sigma^{\alpha} \cU^{\vee}_{k} \otimes \Sigma^{\xi}(\cU_{k}/\cU_{k-t}) \otimes \Sigma^{\zeta}_{\Sp}\cS_{2(n-k)}\otimes p^*_{t} \Sigma_\Sp^{\nu} \cS_{2(n-k + t)}$ have the following form:
\begin{equation} \label{irreducibleagain}
 \Sigma^{\lambda} \cU^{\vee}_{k-t} \otimes \Sigma^{\theta}(\cU_{k}/\cU_{k-t})^{\vee} \otimes \Sigma^{\psi}_{\Sp}\cS_{2(n-k)},\quad \text{where} \quad \lambda\in \YD_{k-t},\quad \theta\in \Y_t,\quad \text{and}\quad \psi\in \YD_{n-k}.   
\end{equation}
Next, if $\lambda>0$ then Theorem~\ref{theo:relative_BBW_sym} implies that for any~$\theta\in \Y_t$ and $\psi\in \YD_{n-k}$ we have
    \begin{equation*}
\mathrm{H}^{\bullet}(\mathrm{IFl}(k-t,k,V),\, \Sigma^{\lambda} \cU_{k-t}^{\vee} \otimes \Sigma^{\theta}(\cU_{k}/\cU_{k-t})^{\vee} \otimes \Sigma^{\psi}_{\Sp}\cS_{2(n-k)})^{\bG} = 0.
\end{equation*}
Thus, in~\eqref{irreducibleagain} only subquotients with $\lambda=0$ have nontrivial~$\bG$-invariant cohomology and only the subquotient $ \Sigma^{\alpha} (\cU_k/\cU_{k-t})^{\vee}$ in~\eqref{decompU_k1} gives a nontrivial contribution to~$\bG$-invariant cohomology of the vector bundle~$\Sigma^{\alpha} \cU^{\vee}_{k} \otimes \Sigma^{\xi}(\cU_{k}/\cU_{k-t}) \otimes \Sigma^{\zeta}_{\Sp}\cS_{2(n-k)}\otimes p^*_{t} \Sigma_\Sp^{\nu} \cS_{2(n-k + t)}$. 
\end{proof}

\subsection{Proof of Lemma~\ref{lemma:minimal_f_in_H}} \label{subsection:appendix:proof}
In this subsection, we use the notation introduced in Section~\ref{section:generalized}. 
Our goal is to establish Lemma~\ref{lemma:minimal_f_in_H}, which concerns function~\(f\) and the set of marked weights, both introduced in Definition~\ref{def:marked}. Before presenting its proof, we first recall the statement.
\begin{lemma*}
Let~\(\alpha \in \YD_1 \times \YD_{t-1}\). 
Consider the set of marked weights
\begin{equation} \label{alphaset1}
 \{(\alpha^{(i)}, b^i(\alpha))\}_{1 \le b^i(\alpha) \le t}.
\end{equation} 
There exists a unique index~\(i_{\alpha} \in [1,t]\) such that
\begin{equation*}
b^{i_{\alpha}}(\alpha) = \dd(\alpha, i_{\alpha}).
\end{equation*}
Moreover, the function~\begin{equation*}
f\colon (\alpha^{(i)}, b^i(\alpha))\mapsto (b^i(\alpha))^2 + \sum_{j=1}^{t} 2j \cdot \alpha^{(i)}_j,
\end{equation*}
attains its strictly minimal value on the set~\eqref{alphaset1} precisely at the marked weight~\((\alpha^{(i_{\alpha})}, b^{i_{\alpha}}(\alpha))\).
\end{lemma*}

We begin by giving a combinatorial description of the family of Young diagrams~\(\{\alpha^{(i)}\}_{0 \le i \le \alpha_1}\) associated with a diagram~\(\alpha \in \YD_k\); see also~\cite[Figure 5.2]{Fonarev_2013}.

We define the \emph{outer strip} of width~1 as the set of boxes that lie in~\(\alpha\) but not in~\(\alpha^{(\alpha_1)}\). To visualize this, draw the diagram of~\(\alpha\), then shade the boxes that would be removed to obtain~\(\alpha^{(\alpha_1)}\). Enumerate these shaded boxes from top to bottom, row by row, starting at the top-right corner. This is illustrated in Figure~\ref{fig:alpha^i} for the partition~\(\alpha = (7,6,6,3,2)\).

To obtain~\(\alpha^{(i)}\), remove the first~\(b^i(\alpha)\) boxes of the outer strip from~\(\alpha\). Equivalently,~\(\alpha^{(i)}\) is obtained by deleting all boxes in the outer strip that lie to the right of the vertical line with abscissa~\(\alpha_1 - i\). For example, in the case illustrated in Figure~\ref{fig:alpha^i}, we have 
\begin{equation*}
(7, 6, 6, 3, 2)^{(4)}=(5,5,3,3,2),
\end{equation*}
where we have removed the first 6 boxes of the outer strip.  

In particular, for $0\le i \le \alpha_1-1$ we can explicitly describe how to obtain $\alpha^{(i+1)}$ from $\alpha^{(i)}$. Namely, we remove~\(b^{i+1}(\alpha) - b^i(\alpha)\) boxes from~\(\alpha^{(i)}\), one in each of the rows indexed by
\begin{equation*}
j = \dd(\alpha, i),\ \dd(\alpha, i)+1,\ \ldots,\ \dd(\alpha, i+1).
\end{equation*}

	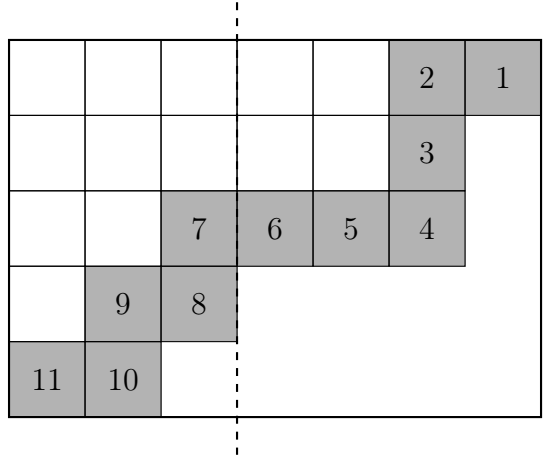
\begin{figure}[h] 
		\centering
		\begin{tikzpicture}[scale=1]
			\def\partitions{{7, 6, 6, 3, 2}}
			\newcounter{cellcount}
			\setcounter{cellcount}{0}
			\foreach \row in {0, 1, 2, 3} {
				\pgfmathsetmacro{\boxes}{\partitions[\row]}
				\pgfmathsetmacro{\nextboxes}{\partitions[\row+1]}
				\foreach \i in {\boxes, ..., \nextboxes} {
					\fill[gray!60] (\i-1, -\row) rectangle (\i, -\row-1);
					\stepcounter{cellcount}
					\node at (\i-0.5, -\row-0.5) {\thecellcount};
				}
			}
			\pgfmathsetmacro{\boxes}{\partitions[4]}
			\foreach \i in {\boxes, ..., 1} {
				\fill[gray!60] (\i-1, -4) rectangle (\i, -5);
				\stepcounter{cellcount}
				\node at (\i-0.5, -4.5) {\thecellcount};
			}
			
			\foreach \row/\boxes in {0/7, 1/6, 2/6, 3/3, 4/2} {
				\foreach \i in {1, ..., \boxes} {
					\draw (\i-1, -\row) rectangle (\i, -\row-1);
				}
			}
			\draw[thick] (0, 0) rectangle (7, -5);
			\draw[thick, dashed] (3, 0.5) -- (3, -5.5);
		\end{tikzpicture}
		\caption{Young diagram for the partition (7, 6, 6, 3, 2) with the shaded outer strip}
            \label{fig:alpha^i}
	\end{figure}

\begin{proof}[Proof of Lemma~\ref{lemma:minimal_f_in_H}]
We begin by analyzing the condition 
\begin{equation*}
b^i(\alpha) = \dd(\alpha, i).
\end{equation*}
Using the identity~\eqref{eq:b=i-v+d},
\begin{equation*}
b^i(\alpha) = i - \vv(\alpha) + \dd(\alpha, i) - 1,
\end{equation*}
we see that this equality holds if and only if
\begin{equation*}
i = \vv(\alpha) + 1.
\end{equation*}
Since~\(\vv(\alpha) \in [0, t-1]\), it follows that~\(i_{\alpha} := \vv(\alpha) + 1 \in [1, t]\), completing the proof of the first part of the lemma.

To prove the second part, we examine how the function $f$ changes between consecutive marked weights. Specifically, we claim that for $1\le b^i(\alpha) \le k$ the function $f$ satisfies the identity:
\begin{equation} \label{diff}
f\big(\alpha^{(i+1)}, b^{i+1}(\alpha)\big)-f\big(\alpha^{(i)}, b^{i}(\alpha)\big)=\big(b^{i+1}(\alpha) - b^i(\alpha)\big)\big(1+2(b^i(\alpha) - \dd(\alpha, i))\big).
\end{equation}
Indeed, to obtain~\(\alpha^{(i+1)}\) from~\(\alpha^{(i)}\), we remove exactly~\(b^{i+1} - b^i\) boxes, one from each of the rows indexed by~\(\dd(\alpha, i), \dd(\alpha, i)+1, \ldots, \dd(\alpha, i+1)\); see the explanation above the proof. Thus,
\begin{multline*} 
 f(\alpha^{(i+1)}, b^{i+1})-f(\alpha^{(i)}, b^{i})=(b^{i+1}-b^{i})(b^{i+1}+b^{i})-\sum_{i=\dd(\alpha, i)}^{\dd(\alpha, i+1)} 2i=\\
 =(b^{i+1}-b^{i})(b^{i+1}+b^{i})-\big(\dd(\alpha, i+1)+\dd(\alpha, i)\big)\big(\dd(\alpha, i+1)-\dd(\alpha, i)+1\big)=\big(b^{i+1} - b^i\big)\big(1+2(b^i - \dd(\alpha, i))\big),
\end{multline*}
where the last equality follows from~\eqref{eq:b=i-v+d}. The equality~\eqref{diff} follows.

Recall that $\{b^i(\alpha)\}$ is strictly increasing in $i$. Also, using the identity~\eqref{eq:b=i-v+d} again, we find that 
\begin{equation*}
b^i(\alpha) - \dd(\alpha, i)=i-\vv(\alpha)-1.
\end{equation*}
This shows that for~\(i < i_{\alpha} \) we have~\( b^i(\alpha) < \dd(\alpha, i) \), so the difference~\eqref{diff} is negative and the function~\(f\) strictly decreases. Conversely, for~\( i > i_{\alpha} \), we have~\( b^i(\alpha) > \dd(\alpha, i)\), so~\( f \) strictly increases. Therefore,~\(f\) attains its strictly minimal value precisely at~\(i = i_{\alpha}\).
\end{proof}

\printbibliography

\end{document}